\documentclass{amsart}
\usepackage[top=2.5cm,left=2cm,right=2cm,bottom=3cm]{geometry}
\usepackage[toc,page]{appendix}
\usepackage{amsmath, amssymb, amstext,amsthm}
\usepackage{enumitem}
\usepackage{stackengine} 
\usepackage{bm} % shorter way to write bold in math mode
\usepackage[british]{babel}
\usepackage{verbatim}
\usepackage{esint}
\usepackage[toc]{appendix}
\usepackage{amsfonts}\usepackage{amscd}\usepackage{url}\usepackage{amsopn}% provides \DeclareMathOperator
\usepackage{bbm}% additional fonts
\usepackage{cite}\allowdisplaybreaks[1]
\usepackage{stmaryrd}
\usepackage{color}
\usepackage{xcolor}
\usepackage{tikz,pgf}
\usetikzlibrary{patterns, decorations.pathmorphing}
\usepackage{tabu}
\usepackage{ulem} % for strikeout, remove later
\usepackage{setspace}
\usepackage{hyperref}

\makeatletter

\newcounter{thm}
\numberwithin{thm}{section}
\newtheorem{remark}[thm]{Remark}
\newtheorem{definition}[thm]{Definition}
\newtheorem{conjecture}[thm]{Conjecture}
\newtheorem{theorem}[thm]{Theorem}
\newtheorem{lemma}[thm]{Lemma}
\newtheorem{proposition}[thm]{Proposition}

\newcommand{\vphi}{\varphi}
\newcommand{\eps}{\varepsilon}
\newcommand{\vt}{\vartheta}

\newcommand{\E}{\mathbb{E}}
\renewcommand{\P}{\mathbb{P}}
\renewcommand{\O}{\Omega}
\renewcommand{\o}{\omega}

\newcommand{\mcB}{\mathcal{B}}

\newcommand{\mcF}{\mathcal{F}}

\newcommand{\F}{\mathcal F}

\newcommand{\R}{\mathbb{R}}
\newcommand{\C}{\mathbb{C}}
\newcommand{\N}{\mathbb{N}}
\newcommand{\Z}{\mathbb{Z}}

\newcommand{\T}{\mathbb{T}}

\newcommand{\D}{\Delta}
\newcommand{\mfp}{\mathfrak{p}}
\newcommand{\e}{\varepsilon}
\newcommand{\adiv}{\mathcal{R}}

\newcommand{\esssup}{\mathrm{ess~sup}}
\DeclareMathOperator*{\supp}{supp}
\DeclareMathOperator*{\sign}{sign}

\DeclareMathOperator{\divv}{div}

\begin{document}

\title{Local Nonuniqueness for Stochastic Transport Equations with Deterministic Drift}
\author{Stefano Modena, Andre Schenke}
\email{stefano.modena@gssi.it, andre.schenke@cims.nyu.edu}
\address{\noindent Gran Sasso Science Institute, 67100 L'Aquila, Italy \newline
Courant Institute of Mathematical Sciences
New York University, New York, N.Y. 10012-1185, USA}

\begin{abstract}
We study well-posedness for the stochastic transport equation with transport noise, as introduced by Flandoli, Gubinelli and Priola \cite{FGP10}. We consider periodic solutions  in $\rho \in L^{\infty}_{t} L_{x}^{p}$ for divergence-free drifts $u \in L^{\infty}_{t} W_{x}^{\theta, \tilde{p}}$ for a large class of parameters. We prove local-in-time pathwise nonuniqueness and compare them to uniqueness results by Beck, Flandoli, Gubinelli and Maurelli \cite{BFGM19}, addressing a conjecture made by these authors, in the case of bounded-in-time drifts for a large range of spatial parameters. To this end, we use convex integration techniques to construct  velocity fields $u$ for which several solutions $\rho$ exist in the classes mentioned above. The main novelty lies in the ability to construct deterministic drift coefficients, which makes it necessary to consider a convex integration scheme \textit{with a constraint}, which poses a series of technical difficulties.
\end{abstract}

\maketitle

\section{Introduction}

We consider the stochastic transport equation with transport noise on the torus $\mathbb{T}^{d} = \R^{d} / \Z^{d}$, $d \geq 1$, informally given by
\begin{equation}\label{eq_STE}
    \begin{cases}
    \partial_{t} \rho + (u \cdot \nabla) \rho &= (\nabla \cdot \rho) \circ \dot{B}(t), \\
    \rho(0,x) &= \rho_{0}(x).     
    \end{cases}
\end{equation}
Here, $\rho \colon \R_{+} \times \mathbb{T}^{d} \times \O \to \R$ denotes the density being transported, $u \colon \R_{+} \times \mathbb{T}^{d} \to \R^{d}$ a deterministic transporting vector field, and $B \colon \R_{+} \times \O \to \R^{d}$ denotes a $d$-dimensional Brownian motion. This equation models transport of scalar properties like heat or dyes in a random medium, i.e. advected by a random, time-dependent velocity field $u - \dot{B}(\o)$. Here, the randomness is modelled as a white-in-time, uniform in space perturbation of the deterministic vector field $u$.

\subsection{Literature Overview: deterministic case}
In the deterministic case ($B = 0$) and for smooth (say, Lipschitz continuous) vector fields, it is a classical fact that there is a strong connection between solutions to the transport equation \eqref{eq_STE} and the associated ordinary differential equation for the \textit{characteristics} of the vector field:
\begin{equation}
 \begin{cases}
        \dot{y}(t) &= u(t,y(t)) \\
    y(0) &= x.
 \end{cases}
\end{equation}
Indeed, calling the solution to the ODE $\Phi_{t}(x)$ (the so called \textit{flow} of the vector field $u$), the solution to the transport equation is given by 
\begin{equation}
\label{eq_flow_map}
\rho(t,x) = \rho_{0}(\Phi_{t}^{-1}(x)). 
\end{equation}
In case of less regular vector fields,  the theory of well-posedness for the transport equation was pioneered by Di Perna and Lions \cite{DPL89} who proved well-posedness in the class
\begin{align}
\label{eq_classes}
    \rho \in L^{\infty}_{t} L^{p}_{x}, \quad u \in L^{1}_{t} W^{1,\tilde{p}}_{x} 
\end{align}
for
\begin{equation}
\label{eq_dpl}
    \frac{1}{p} + \frac{1}{\tilde{p}} \leq 1.
\end{equation}
Their method was based on applying a regularisation operation, e.g. by a mollification $\rho_{\eps} := \phi_{\eps} * \rho$, $u_{\eps} := \phi_{\eps} * u$, where $\phi_{\eps}$ is a mollifier and $*$ denotes convolution and subsequent estimates of the commutator
\begin{align*}
    r_{\eps} := (\rho * u)_{\eps} - \rho_{\eps} * u_{\eps}
\end{align*}
in a strong sense in some $L^{1}_{tx}$. These estimates lead to strong convergence of the commutator and can then be used to show uniqueness. In this argument, the fundamental assumption is that the integrability of the density and the one of the derivative of the vector field ``match'' in the sense that \eqref{eq_dpl} holds. Di Perna and Lions' result was extended by Ambrosio \cite{Ambrosio04} in the case of vector fields of bounded variation and divergence in $L^\infty$, for densities in $L^\infty_{tx}$. Subsequent important improvements can be found, for instance, in \cite{Crippa1} and \cite{Bianchini:2017vf}.

All the uniqueness results we mentioned are strongly based on two assumptions on the vector field $u$: a bound on its full derivative $\nabla u$ (e.g. an $L^p$ bound like in the DiPerna-Lions theory) and a (stronger) bound on the divergence (e.g. $\divv u \in L^\infty$). 

In addition, it is a sort of rule of thumb that, whenever a uniqueness result holds for a vector field (with low regularity) in the class of bounded densities, it is possible to define a unique ``reasonable'' flow map $\Phi_t$ (called \textit{regular Lagrangian flow}) and the solution to the PDE is still given by \eqref{eq_flow_map}.

This, in turn, implies that, whenever one can show non-uniqueness at the ODE level for the characteristics (like e.g. in the classical 1D example $b = - \mathrm{sign}(x) |x|^{\theta}$, $\theta \in (0,1)$), such ``ODE non-uniqueness'' can be used to construct examples of ``PDE non-uniqueness''. More sophisticated higher-dimensional examples are due to Aizenman \cite{Aizenman78}, whose paper inspired further works such as \cite{Depauw03}.

In recent years example of non-uniqueness for the deterministic transport equation have been constructed by pure PDE methods, using \textit{convex integration}, also for vector fields for which existence and uniqueness of a regular Lagrangian flow was well established.

This method originates from pioneering work of Nash \cite{Nash54} on isometric $C^{1}$ embeddings. Gromov \cite{Gromov86}, Tartar \cite{Tartar79}, M\"uller and \v{S}ver\'{a}k \cite{MS03}, and many other authors subsequently extended and refined the method. More recently, De Lellis, Sz\'{e}kelyhidi and others have applied these methods in the context of fluid dynamics, culminating in the proof of Onsager's conjecture \cite{Isett18, BDLSV18} and nonuniqueness of weak solutions for the Navier--Stokes equations \cite{BV19a}. See \cite{BV19b} for a good overview of the literature and the technique  and \cite{novack2}, \cite{novack1}, \cite{giri20233}, \cite{giri2023wavelet} for more recent developments.

Most relevant for us, the first named author together with Sz\'{e}kelyhidi \cite{MS18,MS19}, as well as with Sattig \cite{MS20} extended the method to the deterministic transport equation with Sobolev vector fields $u$, showing that when \eqref{eq_dpl} fails, or, more precisely, when
\begin{equation*}
\frac{1}{p} + \frac{1}{\tilde p} > 1+ \frac{1}{d},
\end{equation*}
then there are vectors fields $u$ and nonzero densities $\rho$ in the classes \eqref{eq_classes}, with $\rho|_{t=0} = 0$ (thus showing non-uniqueness). The crucial point here is, opposite to the DiPerna-Lions theory, that the ``matching'' between integrability of $\rho$ and integrability of $\nabla u$ fails.
% Here, the linear PDE with given vector field $u$ is understood as a nonlinear PDE in $(\rho, u)$, i.e. one constructs both $\rho$ and $u$ and then shows that for this $u$, there are several solutions, hence proving uniqueness. 
 Further applications of convex integration methods to the deterministic transport equation can be found in \cite{CL21}, \cite{CL22} as well as in \cite{brue2021positive}, \cite{pitcho2021almost} and \cite{giri2021non}, where PDE methods have been used to show almost everywhere non-uniqueness of integral curves.   

The aim of this paper is to apply \textit{convex integration} to the stochastic transport equation. 

\subsection{Literature Overview: stochastic case}

Let us now discuss briefly the literature about the stochastic transport equation \eqref{eq_STE}. The rule of thumb here is that the presence of the noise should \textit{regularize} the equation, thus yielding uniqueness, somewhat similarly to the effect of a diffusion term $\Delta \rho$  in the deterministic setting.

Some of the first papers to treat stochastic transport equations seem to be due to Keller \cite{Keller62}, Frisch \cite{Frisch68} and Ogawa \cite{Ogawa73, Ogawa78}, with further contributions due to Funaki \cite{Funaki79}.
Kunita \cite{Kunita84a, Kunita84b, Kunita97} provided a well-posedness theory for smooth initial data and regular coefficients, including a flow representation by means of an SDE.
A theory of Di Perna--Lions type for the stochastic equations has been developed by Figalli \cite{Figalli08}, see also the works of Le Bris and Lions \cite{LBL04,LBL08}. 

Another important landmark result was the celebrated paper of Flandoli, Gubinelli and Priola \cite{FGP10}, in which they proved a \textit{well-posedness by noise} result for the stochastic transport equation perturbed by transport-type noise (i.e. \eqref{eq_STE}) in the case of a deterministic, H\"older-continuous drift coefficient $u \in L^{\infty}(0,T; C^{\alpha}(\R^{d}))$ with $\divv u \in L^{p}([0,T] \times \R^{d})$ for a $p > 2$,  for which, in general, uniqueness  in the deterministic case fails, as we observed before. 

 The authors of \cite{FGP10} introduced a new strategy for proving the uniqueness by \textit{weak} commutator estimates using the flow representation of the stochastic transport equation which holds in much worse situations than in the deterministic case. This fact allowed them to transfer the randomness in the commutator estimates to a test function, if one has enough control on the Jacobian of the flow, which is the case for H\"older continuous vector fields $u$. Interestingly, the H\"older-ness of $u$ is used only to prove existence of a sufficiently regular flow, i.e. only on the SDE level.

We note also that Flandoli \textit{et al.} gave a counterexample \cite[Section 6.2]{FGP10} to uniqueness in the case of a \textit{random} drift coefficient, essentially by using a nonuniqueness result for ODEs with irregular coefficients and a translation. Therefore, one expects the strongest uniqueness results in the class of \textit{deterministic} vector fields.

Much of the subsequent work has been aimed at reducing the above strong regularity assumptions. In fact, as it turned out, the essential condition to have a stochastic flow is not H\"older continuity of the coefficients, but rather the so-called \textit{Ladyzhenskaya--Prodi--Serrin (or LPS) condition} for $r,q \in [2,\infty]$,
\begin{align}\tag{$\text{LPS}_{r}$}\label{eq_LPS_r}
    u \in L_{t}^{r} L_{x}^{q} \quad \text{with} \quad \frac{2}{r} + \frac{d}{q} \leq 1,
\end{align}
in analogy to the transport-diffusion equation (see e.g. \cite{bianchini1}).
The case with strict inequality in the LPS condition is commonly referred to as the \textit{Krylov--R\"ockner (KR) condition} in the SDE literature, after their seminal work \cite{KR05} on SDEs. 

We will give an incomplete, personal selection from the long list of contributions after \cite{FGP10}.
Attanasio and Flandoli \cite{AF11} considered weakly differentiable  (or more precisely, $BV$) drift coefficients and relaxed the assumption on the divergence when compared with deterministic results of \cite{Ambrosio04, DPL89}.
Maurelli \cite{Maurelli11} used a Wiener chaos decomposition technique to get uniqueness for $\rho$ in the class $L^{\infty}$ and also in the class $L^{2}$ given a deterministic drift $u \in L^{q}(\R^{d};\R^{d}) \cap L^{2}_{\mathrm{loc}}(\R^{d};\R^{d})$ with $q > d$ (which implies the KR condition with $q = \infty$) and $\divv u \in L^{r}(\R^{d}) \cap L^{2}_{\mathrm{loc}}(\R^{d})$, $r > \frac{d}{2}$. Similarly, Mohammed, Nilssen and Proske \cite{MNP15} used Malliavin calculus to deal with drift terms that are only bounded and measurable and solutions in the class $C_{b}^{1}(\R^{d})$.
Fedrizzi, Neves and Olivera \cite{FNO18} considered the case $u \in L^{2}_{\mathrm{loc}}([0,T]\times \R^{d})$ and showed uniqueness of weak solutions in $L^{2} \cap L^{\infty}$.
Moreover, Fedrizzi and Flandoli \cite{FF13} showed that no loss of regularity occurs in time if the drift is in the KR class and the initial condition is in $\bigcap_{r \geq 1} W^{1,r}(\R^{d})$. Finally, Neves and Olivera \cite{NO15} extended the result of \cite{FGP10} to the KR class for $\divv u = 0$ and later \cite{NO16} to more general drifts.
We note also that there are many similar results available on the level of characteristics, i.e. in the pure SDE case, cf. \cite{FF11, KR05, Nam20, RZ23, WLW17} and references therein. See also the recent work of Krylov (e.g. \cite{Krylov23} and references therein) on the critical case for the SDE. 

To the best of our knowledge, the state of the art for the stochastic transport equation in terms of uniqueness theory in the LPS class is given by the seminal work of Beck, Flandoli, Gubinelli and Maurelli \cite{BFGM19}. They extended the well-posedness theory to the case of critical parameters, i.e. where equality holds in the LPS condition. However, they also do much more than that. Their more analytical method deviates from the flow-based methods of \cite{FGP10} and many other works, they treat not only stochastic transport equations but also stochastic continuity equations and stochastic ordinary differential equations (SDEs), extend the notion of uniqueness from pathwise to path-by-path uniqueness.\footnote{ Concerning the integrability class for the densities in \cite{BFGM19}, they have to be compatible to the vector fields in the H\"older conjugacy sense, i.e. $\rho \in L^{\theta}_{\o} L^{\theta}_{t} L^{p}_{x,\mathrm{loc}}$ with $\theta \geq 2, p \geq q'$. So they consider all time integrabilities $r,\theta \geq 2$,  contrary to the current work, where we focus exclusively on the $r=\theta = \infty$ case. Note further that their main results, as they are stated, assume that $\theta = p$. However, we think that it should not be too difficult to extend them to the case with different integrabilities in time and space.} 

In particular, they give the following conjecture \cite[p. 6]{BFGM19}:
\begin{conjecture}\label{conj_LPS}
    If the LPS condition \eqref{eq_LPS_r} does not hold, i.e. if $\frac{2}{r} + \frac{d}{q} > 1$, a general result for regularisation by noise will probably be false.
\end{conjecture}

The aim of this paper is to investigate this conjecture in the case $r = \infty$. In this case, the LPS condition becomes 
\begin{align}\tag{$\text{LPS}_{\infty}$}\label{eq_LPS_infty}
    u \in L_{t}^{\infty} L_{x}^{q} \quad \text{with} \quad \frac{d}{q} \leq 1.
\end{align}
We will study this conjecture in the following class of drifts:
\begin{align*}
    u \colon [0,{1}] \times \T^{d} \to \R^{d}, \quad \text{\textit{deterministic}}, \quad \divv u = 0, \quad u \in L_{t}^{\infty} W_{x}^{\theta,\tilde{p}}.
\end{align*}
More precisely, we will assume that $\rho \in L^{\infty}_{t} L_{x}^{p}$, $u \in  L^{\infty}_{t} W^{\theta,\tilde{p}}_{x}$ with $p \in [1,\infty)$, $\tilde{p} \in (1,\infty)$, $\theta \in [0,1]$ such that
\begin{equation}\label{eq_main_scaling_cond}
    \frac{1}{p} + \frac{1}{\tilde{p}} > 1 + \frac{\theta}{d}, \quad \frac{d}{\tilde{p}} > 1 + \theta.
\end{equation}
The first condition is a generalisation to arbitrary regularity of a similar condition in the deterministic case, cf. \cite{MS20}. The second condition is comparable to the LPS condition \eqref{eq_LPS_infty} as we can use the Sobolev embedding to ``trade'' the additional spatial regularity for integrability, cf. Fig. \ref{fig_main_result} below.
Note that in this case, we do not get  that $\rho u \in L^{1}$ automatically from the H\"older inequality, as $\rho$ and $u$ are not contained in H\"older-conjugated Lebesgue spaces. However, our construction will nevertheless yield that the pair $(\rho, u)$ satisfies $\rho u \in L^{1}_{x}$. Moreover, the condition $\divv u = 0$ ensures that we can treat simultaneously the transport and the continuity equation.  

We can allow for additional regularity of $u$, but in the end, our result will depend on the spatial integrability of $u$ after Sobolev embedding, and in particular on whether this integrability still violates the LPS condition.  
Furthermore, we expect that similar nonuniqueness results will hold if one relaxes the assumption on the divergence of $u$, since, as in the deterministic case, the ``worse'' the divergence, the easier it is to lose uniqueness.
We work on the torus $\T^{d}$ for the main reason that our method of proof, convex integration, is by far the most developed in this context. This also means we don't need any growth assumptions and hence get slightly simpler statements.

It might be interesting to try to extend the techniques to include the case of $\rho \in L^{s}_{t} L_{x}^{p}$ with $s < \infty$. There are pioneering works on the deterministic transport equation by Cheskidov and Luo \cite{CL21,CL22}, but they display some highly erratic behaviour, so we will focus here on the case $s=\infty$ only.

\subsection{Convex integration for stochastic PDEs} 

Adapting ideas of Buckmaster and Vicol \cite{BV19a} to the stochastic case, Hofmanov\'{a}, Zhu and Zhu \cite{HZZ19} proved nonuniqueness for martingale solutions of 3D stochastic Navier--Stokes equations. 
This showed that at least simple additive, spatially smooth Brownian noise is not sufficient to induce well-posedness of the perturbed system.  

Since then, the method has been developed further and applied to stochastic Euler equations \cite{HZZ22} (using a Baire category argument similar to \cite{DLS09}), arbitrary $L^{2}$ initial values for the stochastic Navier--Stokes equations \cite{HZZ21a}, Navier--Stokes perturbed by space-time white noise \cite{HZZ21b}, by linear multiplicative noise with prescribed energy \cite{Berkemeier22}, as well as ergodicity questions for Navier--Stokes \cite{HZZ22a, HZZ22b}. It has also been applied to other equations such as the Boussinesq equation \cite{Yamazaki22B}, hyper- \cite{Yamazaki22L} and hypodissipative fractional NSE \cite{RS23, Yamazaki21}, the magnetohydrodynamic (MHD) equations \cite{Yamazaki21}, surface quasigeostrophic equations \cite{Yamazaki22S}, power-law fluids \cite{LZ22} and sharp nonuniqueness for Navier--Stokes \cite{CDZ22}. There are also several works for compressible stochastic systems, such as \cite{BFH20,CFF19}. Most recently, the Euler and Navier--Stokes equations have been studied with transport noise \cite{HLP22, Pappalettera23}. 

The work that is closest to the present article is due to Koley and Yamazaki \cite{KY22}, who prove nonuniqueness in law for the transport-diffusion equation forced by several types of random noise, including transport noise. In this paper we focus only on transport-type noise as it is physically motivated (it can be understood as a noise acting on the level of characteristics as an additive Brownian noise) and allows for the closest comparison to the literature, where most regularisation by noise results are for transport-type noise. In the case of transport noise, following the deterministic transport equation papers \cite{MS18,MS19,MS20}, the authors of \cite{KY22} distribute the dependence on the previous defect $R_0$ over both the velocity and the density and put all the Brownian motion into the velocity. Since the previous defect depends on the previous velocity field in several ways, their previous defect is only $C^{1/2-}$ in time, just as ours. They then proceed to mollify the previous defect before putting it into their new perturbations. However, in this way, a nontrivial dependence on paths of the Brownian motion is created in their velocity fields which means that they are \textit{random} functions (cf. \cite[Eq. $(35)$ and $(39a)$]{KY22}), and owing to this mollification, this randomness persists even after transformation back to the stochastic transport equation. This allows for a line of attack in the proof which is ``closer'' to known works, and seemingly one can avoid the use of stopping times \cite[Remark 6.2]{KY22}, but at the disadvantage of producing, in some sense, a less “surprising” result: it has indeed been known for some time that random drifts may allow nonuniqueness, even in regimes where one can prove uniqueness for deterministic drifts, essentially by shifting nonuniqueness results in the deterministic case, cf. e.g. Section 6.2. in the paper \cite{FGP10} by Flandoli, Gubinelli and Priola. A nonuniqueness result for a stochastic transport equation with a \textit{random} drift is therefore, in a sense, much more to be  “expected” than one for a deterministic drift, in particular in view of the results in \cite{MS20,MS18,MS19}.

Therefore, in the present paper, we need to work hard to attain nonuniqueness for a stochastic transport equation with a \textit{deterministic} drift, making it more directly comparable to the well-posedness by noise literature cited above, in particular to \cite{BFGM19}.
As it will turn out and discussed below, this is a technically nontrivial extension and requires careful analysis and some further arguments, while seemingly preventing us from attaining the strongest possible results in terms of existence (global-in-time solutions).
\newpage

\subsection{Main results and methods}
Our notion of solution will be the following:
\begin{definition}\label{def_local_weak_soln} Let $\tau$ be an $(\mcF_{t})_{t}$-stopping time and $u \in L^{1}([0,1] \times \T^{d};\R^{d})$ be a deterministic vector field.

A \textbf{local weak solution to \eqref{eq_STE_proper} of class} $L^{\infty}_{\o}{L_{t}^{\infty}}L^{p}_{x}$ \textbf{up to time $\tau$} is a random field $\rho \colon  [0,\tau] \times \T^{d} \times \O \to \R$ with the following properties:
\begin{enumerate}[label = (\roman*)]
 \item[(0)]\label{itm_def_soln0} it is weakly progressively measurable with respect to $(\mcF_{t})_{t}$, i.e. for every $\phi \in C^{\infty}(\T^{d})$, the map $(t,\o) \mapsto \int_{\T^{d}} \rho(t,x;\o) \phi(x) dx $ is progressively measurable.
 \item \label{itm_def_soln1} $\rho$ is in $L^{\infty}(\O,{L^{\infty}}([0,\tau];L^{p}(\T^{d})))$ and $\rho u$ is in $L^{\infty}(\O,{L^{\infty}}([0,\tau]; L^{1}(\T^{d};\R^{d})))$.
 \item \label{itm_def_soln2} $(t,\o) \mapsto  \int_{\T^{d}} \rho(t,x;\o) \phi(x) dx$ has a modification that is a continuous semimartingale, for every $\phi \in C^{\infty}(\T^{d})$.
 \item \label{itm_def_soln3} For every $\phi \in C^{\infty}(\T^{d})$, for this continuous modification of $\int_{\T^{d}} \rho(t,x;\o) \phi(x) dx$, it holds that, $\P$-a.s., for all $t \in [0,\tau]$
 \begin{equation}
\label{eq_def_weak_sln}
\begin{aligned}
\int_{\T^d} \rho(t,x; \omega) \phi(x)dx = & \int_{\T^d} \rho_0(x; \omega) \phi(x)dx + \int_0^t \int_{\T^d} \rho(s,x; \omega) u(s,x) \cdot \nabla \phi(x) dx ds \\
& + \sum_{i=1}^d \int_0^t \left( \int_{\T^d}  \rho(s,x; \omega) \partial_{x_i} \phi(x) dx \right) d B^i(s) 
+ \frac{1}{2} \int_0^t \int_{\T^d} \rho(s,x; \omega) \D \phi(x) dx ds
\end{aligned}
\end{equation}
\end{enumerate}
\end{definition}

Note that this is the same definition as the one in \cite[Definition 3.1, Remark 3.2]{BFGM19}, localised in time and formulated for the torus instead of $\R^{d}$, except for the additional integrability condition on the product $\rho u$ in \ref{itm_def_soln1}, which generalises the condition $m \geq p'$ in \cite[Remark 3.2]{BFGM19}. 

Our main result is summarised in the following theorem and visualised in Fig. \ref{fig_main_result}. Note that a similar result holds for the transport-diffusion equation, cf. Section \ref{ssec_diffusion_thm}.
\begin{theorem}[Non-uniqueness for the stochastic transport equation]\label{thm_main_SPDE}
    Let $d \geq 2$, $\mathfrak{p} \in [0,1)$. Let $p \in [1, \infty)$, $\tilde{p} \in (1,\infty)$, $\theta \in [0,1]$ such that
    \begin{align}
    \label{eq_cond_exp}
        \frac{1}{p} + \frac{1}{\tilde{p}} & > 1 + \frac{\theta}{d}, \\
        \frac{d}{\tilde p} & > {1+\theta},     \label{eq_cond_exp_2}
    \end{align}    
    and let $B = (B(t))_{t \geq 0}$ be a $d$-dimensional Brownian motion defined on the stochastic basis $(\O,\F,\P,(\F_{t})_{t \geq 0})$ with a normal filtration.
    
    Then there exists a {stopping time $\tau_{\mfp}: \Omega \to [0,1]$ such that $\P(\tau_{\mfp} = {1}) > \mfp$} 
    as well as a deterministic, divergence-free vector field such that
    \begin{align*}
         u \in C^{0}([0,{1}];  W^{\theta, \tilde{p}}(\T^{d}) ).
    \end{align*}
    For this vector field, pathwise local-in-time uniqueness of distributional solutions to the stochastic transport equation
    \begin{equation}\label{eq_STE_proper}
        d \rho + (u \cdot \nabla) \rho = (\nabla \rho) \circ dB(t)
    \end{equation} 
    fails in the class of $L^{\infty}(\O; C^{0}([0,\tau_{\mfp}] ; L^{p}(\T^{d})))$ densities, in the sense that there exist two different, $(\F_{t})$-adapted processes
    \begin{align*}
        \rho_{1}, \rho_{2} \in L^{\infty}(\O; C^{0}([0,\tau_{\mfp}] ; L^{p}(\T^{d}))),
    \end{align*}
    differing on the set $\{ \tau_{\mfp} = {1} \}$, such that $\rho_{i} u \in L^{\infty}(\O; C^{0}([0,\tau_{\mfp}]; L^{1}(\T^{d};\R^{d})))$, $i = 1,2$, which both solve the equation \eqref{eq_STE_proper} with zero initial condition.
\end{theorem}
 
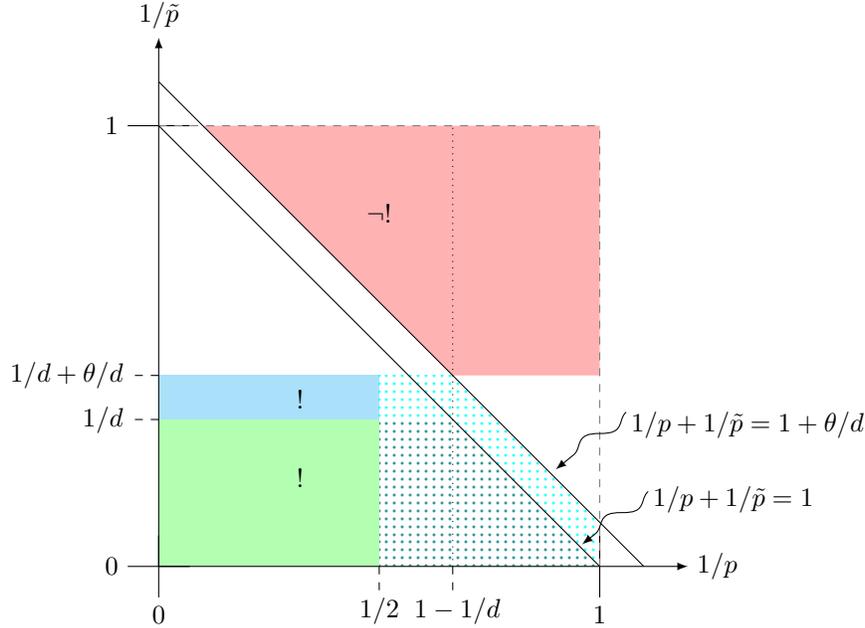
\begin{figure}[htbp]
  \centering
  \begin{tikzpicture}[scale=\textwidth/3cm]
    % Background layer
    \begin{scope}%[on background layer]
      % Uniqueness areas
      \fill[green!30] (0,1/3) -- (1/2,1/3) -- (1/2,0) -- (0,0) -- cycle;
      \fill[pattern=dots, pattern color=teal!80] (1/2,1/3) -- (1-1/3,1/3) -- (1,0) -- (1/2,0) --  cycle;
      \fill[cyan!30] (0,1/3) --(0,1/3+0.3/3) -- (1/2,1/3+0.3/3) -- (1/2,1/3)  -- cycle;
      \fill[pattern=dots, pattern color=cyan!80] (1/2,1/3) --(1/2,1/3+0.3/3) -- (1-1/3,1/3+0.3/3) -- (1-1/3,1/3)  -- cycle;
      \fill[pattern=dots, pattern color=cyan!80] (1-1/3,1/3+0.3/3) -- (1-1/3,1/3) -- (1,0) -- (1,0.3/3)  -- cycle;
%       \fill[pattern=dots, pattern color=teal!80] (1-1/3,1/3) -- (1,1/3) -- (1,0) -- cycle;
%       \fill[pattern=dots, pattern color=cyan!80] (1-1/3,1/3+0.3/3) -- (1-1/3,1/3) -- (1,1/3) -- (1,1) -- cycle;

      %\fill[pattern=dots, pattern color=cyan!90] (0,1/3) --(0,1/3+0.3/3) -- (1,1/3+0.3/3) -- (1,1/3)  -- cycle;
      % Nonuniqueness areas
      \fill[red!30] (0.3/3,1) -- (1-1/3,1/3+0.3/3) -- (1,1/3+0.3/3) -- (1,1) -- cycle;
    \end{scope}
    
    % Axis lines 
    \draw[-latex] (0,0) -- (1.2,0) node[right] {$1/p$};
    \draw[-latex] (0,0) -- (0,1.2) node[above] {$1/{\tilde{p}}$};
    
    % Ticks and labels
    \foreach \x/\xtext in {0,1}
      \draw (\x,2pt) -- (\x,-2pt) node[below] {\xtext};
      
    \foreach \y/\ytext in {0,1}
      \draw (2pt,\y) -- (-2pt,\y) node[left] {\ytext};
    
    % Label and horizontal line with dash
    \node[left] at (-0.06,1/3) {$1/d$};
    \draw[dashed] (0,1/3) -- (-2pt,1/3);
        
    % Dashed gray lines
    \draw[dashed, gray] (0,1) -- (1,1);
    \draw[dashed, gray] (1,0) -- (1,1);
    
    % Solid black lines
    \draw (0,1) -- (1,0);
    \draw (0,1+0.3/3) -- (1+0.3/3,0);
    
    % Dotted black vertical line
    \draw[dotted, black] (1-1/3,0) -- (1-1/3,1);
    
    % Label for the point
    \node[left] at (-0.06,1/3+0.3/3) {$1/d + \theta/d$};
    \draw[dashed] (0,1/3+0.3/3) -- (-2pt,1/3+0.3/3);
    
    % Label and squiggly arrow
    \draw[-latex, black, decorate, decoration={snake,amplitude=1mm,segment length=8mm,post length=1mm}] (1.11,0.17) -- (0.96,0.05);
    \node[above right] at (1.1,0.1) {$1/p + 1/\tilde{p} = 1$};
    
    \draw[-latex, black, decorate, decoration={snake,amplitude=1mm,segment length=8mm,post length=1mm}] (1.06,0.25+0.3/3) -- (0.9,0.12+0.3/3);
    \node[above right] at (1.05,0.17+0.3/3) {$1/p + 1/\tilde{p} = 1 + \theta/d$};
    
    % Dashed label
    \node[below] at (1-1/3,-0.05) {$~1-1/d$};
    \draw[dashed] (1-1/3,0) -- (1-1/3,-2pt);
    \node[below] at (1/2,-0.05) {$1/2$};
    \draw[dashed] (1/2,0) -- (1/2,-2pt);
    
    % Existence label
    \node[left] at (0.35,0.2) {$!$};
    \node[left] at (0.35,0.38) {$!$};
    \node[left] at (0.55,0.8) {$\neg !$};
  \end{tikzpicture}
  \caption{State of (local-in-time) uniqueness versus nonuniqueness of solutions $\rho \in L_{\o}^{\infty} L_{t}^{\infty} L_{x}^{p}$ to \eqref{eq_STE_proper} for deterministic vector fields $u \in L^{\infty}_{t}W^{\theta,\tilde{p}}_{x}$. We have displayed the case $d=3$, $\theta = 0.3$ for concreteness. \newline 
  \textit{Green areas:} The green area corresponds to the results of \cite{BFGM19} for which $\theta = 0$, with our modified definition of solution. Otherwise one needs that $\frac{1}{p} + \frac{1}{\tilde{p}} \leq 1$. Technically speaking, the result of \cite{BFGM19} only yields uniqueness in the part of the green area below the $\frac{1}{p} + \frac{1}{\tilde{p}} = 1$ line, which we indicated by the solid line. But we think it should not be difficult to extend their results to give uniqueness in the more general class of Definition \ref{eq_def_weak_sln}, which we indicated by the dotted green area (see Appendix \ref{sec_AppB} for a more detailed discussion). Note that the comparison with \cite{BFGM19} is limited to the case $r=\infty$, and only to a small subset of results of their results, i.e. uniqueness questions for \eqref{eq_STE_proper}. \newline
  \textit{Blue areas:} When $\theta > 0$, one can apply the Sobolev embedding theorem to get $W^{\theta,\tilde{p}}_{x} \subset L^{q}_{x}$, and only $q = \tilde{p}\left( \frac{1}{1 - \theta \tilde{p}/d} \right)$ needs to satisfy the LPS condition $\frac{d}{q} \leq 1$, thus extending the area of uniqueness upwards in the vertical direction. Similar reasoning applies to the dotted green area, which can be extended by to give the dotted blue area. \newline
  \textit{Uniqueness range:} Together, for parameters $p,\tilde{p}$ in the green or blue ranges, including the dotted ones, i.e. for all $(p,\tilde{p})$ with $\tilde{p}$ satisfying the LPS condition $\frac{1}{\tilde{p}} \leq \frac{1}{d}$, we get uniqueness, symbolised as ``$!$''. \newline \textit{Red area:} The main result of this paper, Theorem \ref{thm_main_SPDE}, gives nonuniqueness of \eqref{eq_STE_proper} in the red parameter range (top right part of the figure), symbolised as ``$\neg !$''. As $\theta$ increases, the blue area expands upwards, the red area retreats upwards, but the two always touch in one point $\frac{1}{\tilde{p}} = \frac{1}{d} + \frac{\theta}{d}$, $\frac{1}{p} = 1 - \frac{1}{d}$. 
  The \textit{white area} to the left of the nonuniqueness area is still open and should admit nonuniqueness, by Conjecture \ref{conj_LPS}, but it is not tractable with current convex integration techniques used in this paper. The white area in the bottom right of the figure is open as well.} \label{fig_main_result} 
\end{figure}

To this end, we employ the method of convex integration, in particular utilising the Mikados and blobs of \cite{MS18,MS19,MS20}. Similarly to \cite{KY22}, we apply this method to a ``shifted'' equation, a random PDE, and later shift it back using It\^{o}'s formula. However, as the scalar fields in question are quite rough in time, we cannot simply apply It\^{o}'s formula. We circumvent this obstacle by applying It\^{o}'s formula along the iteration, not just to the limit, and show that all the terms converge.

In order to achieve a deterministic vector field, we need to employ the convex integration technique \textit{with a constraint}. This constraint, roughly speaking, destroys a lot of the usual flexibility of the scheme and prevents us from getting a global-in-time solution, i.e. a solution up to a fixed time $T$, not just up to the stopping time $\tau_{\mfp}$, as all ways of extending the solution to later times that we tried seem to either not work or to introduce randomness into the drift, which is inacceptable in our setting. The current work seems to be the first stochastic convex integration paper that can only achieve local results. This is due to the above-mentioned constraint of having a \textit{deterministic} vector field.

This problem is perhaps comparable to solving the stochastic Navier--Stokes equations with convex integration techniques, while demanding that the first component of the vector field there is deterministic. As such a problem has not been studied yet, it is not surprising that the known techniques are not optimal in this setting. We are planning to address this issue in the future.

Compared to earlier works, we need to employ a mollification in time and take great care of which terms are endowed with mollified quantities to preserve the non-randomness of the vector field $u$. Moreover, we use an interpolation inequality to investigate the whole range of Sobolev regularities $W_{x}^{\theta,\tilde{p}}$ instead of just the end point $\theta = 1$. Note also that the functions constructed in Theorem \ref{thm_main_SPDE} are continuous in time and $(\F_{t})_{t}$-adapted, which implies that they are progressively measurable. In particular, our result implies nonuniqueness in the class of solutions of Definition \ref{def_local_weak_soln}. Moreover, our proof gives nonuniqueness also in the smaller class $L_{\o}^{\infty} C^{0}_{t} L_{x}^{p}$.

The problems caused by the constraint implies that, on the one hand, as in the deterministic scheme of \cite{MS18}, the sequence of approximate densities converges in some $L^s$ space, on the other hand, the sequence of approximate fields does not converge in the dual $L^{s'}$,
which prevents us from concluding that $\rho u \in L_{x}^{1}$ via H\"older's inequality. A more detailed explanation of this problem will be given in Remark \ref{rem_cutoff} below. Therefore, we needed to generalise the definition of weak solution by demanding that additionally $\rho u \in L_{x}^{1}$, an idea already used in the deterministic case in \cite{giri2021non}.

Regarding the parameter ranges for $p, \tilde{p}$ covered in this work, we note that in the current form, only $\tilde{p} > 1$ is covered by our results. This restriction stems entirely from the interpolation results of Appendix \ref{sec_AppA}, which are only for $\tilde{p} \in (1,\infty)$. Extension of these results to $\tilde{p} = 1$ should immediately yield that our main result holds in that case as well.

\subsection{Extension to transport-diffusion}\label{ssec_diffusion_thm}

We conclude this introduction by spending a few words on higher order transport models. We have already mentioned that convex integration yields non-uniqueness for the deterministic transport equation in the class $\rho \in L^p$, $u \in W^{\theta, \tilde p}$, provided condition \eqref{eq_cond_exp} holds (see \cite{MS18}, \cite{MS19}, \cite{MS20})\footnote{To be precise, only $\theta=1$ is considered in the mentioned papers, but extension to any $\theta \in [0,1]$ can be easily achieved using the interpolation inequality in Proposition \ref{prop_interpolation}.}. In these papers it is also  observed that the same scheme yields non-uniqueness even if diffusion is added to the equation as in 
\begin{equation}
\label{eq-tr-diff}
\partial_t \rho + \divv (\rho u) = \Delta \rho
% , \qquad \divv u = 0, \qquad \rho|_{t=0} = \rho_0
\end{equation}
provided condition \eqref{eq_cond_exp} is supplemented by additional conditions on the exponents $p, \tilde p$.  This is true also in the stochastic case, where, quite interestingly, no additional conditions besides \eqref{eq_cond_exp}, \eqref{eq_cond_exp_2} are needed. 

\begin{theorem}[Non-uniqueness for stochastic transport-diffusion]
\label{thm:spde-transp-diff}
Let $d, \mathfrak p, p, \tilde p, \theta, B$ be as in the statement of Theorem \ref{thm_main_SPDE}. In particular let \eqref{eq_cond_exp}, \eqref{eq_cond_exp_2} hold.  Then the same conclusion as in Theorem \ref{thm_main_SPDE} holds, with \eqref{eq_STE_proper} substituted by
\begin{equation}
\label{eq:SPDE-diff}
        d \rho + (u \cdot \nabla) \rho dt = (\nabla \rho) \circ dB(t) + \Delta \rho dt.
\end{equation}
\end{theorem}
\noindent The proof of Theorem \ref{thm:spde-transp-diff} is very similar to the one of Theorem \ref{thm_main_SPDE} and it will be sketched in Section \ref{s:diffusion}. 

It is interesting to compare the different conditions on the exponents needed in the deterministic/stochastic case for pure transport or transport/diffusion\footnote{ Actually, in \cite{MS20}, the additional condition besides \eqref{eq_cond_exp} needed in the deterministic case to deal with the diffusion term is not \eqref{eq_cond_exp_2}, as indicated in the table, but $p' < d$
%\begin{equation}
%\label{eq_cond_exp_diff}
%p' < d
%\end{equation}
which is equivalent to 
\begin{equation}
\label{eq_cond_exp_diff_2}
\frac{1}{p} < 1 - \frac{1}{d}.
\end{equation}
Conditions \eqref{eq_cond_exp} and \eqref{eq_cond_exp_diff_2} give together, in Figure \ref{fig_main_result}, the red \textit{triangle} to the \textit{left} of the dotted line. The reason why the red \textit{rectangle} to the \textit{right} of the dotted line could not be achieved in \cite{MS20} is that the result in \cite{MS20} actually provides a density $\rho \in L^p$ and a vector field $u \in W^{\theta, \tilde p}$ which in addition satisfies $u \in L^{p'}$. If one does not require explicitly that $u \in L^{p'}$ but only that $\rho u \in L^1$ (so that, in particular, $\rho u$ is well defined as a distribution), then it is not difficult to see that the non-uniqueness result of \cite{MS20} for deterministic transport-diffusion can be extended to the whole red area in Figure \ref{fig_main_result} (corresponding to the combination of conditions \eqref{eq_cond_exp}, \eqref{eq_cond_exp_2}). 
}:

\begin{spacing}{1.5}
\begin{center}
\begin{tabu}[h]{l|c|c}
& transport & transport-diffusion \\  \hline
deterministic PDE \cite{MS20} & \eqref{eq_cond_exp} & \eqref{eq_cond_exp} and \eqref{eq_cond_exp_2}
\\ \hline
SPDE [present paper] & \eqref{eq_cond_exp} and \eqref{eq_cond_exp_2} & \eqref{eq_cond_exp} and \eqref{eq_cond_exp_2} 
\end{tabu}
\end{center}
\end{spacing}
We observe that
\begin{enumerate}
\item the addition to the deterministic transport equation of a transport noise has the same effect (in terms of restriction on the exponents for which convex integration works) as the addition of a deterministic diffusion term; this is reminiscent of the fact that (formally) the expectation $\mathbb E [\rho(t,x)]$ of any solution to \eqref{eq_STE_proper} solves the deterministic transport-diffusion equation \eqref{eq-tr-diff}.

\item Since the very same conditions \eqref{eq_cond_exp} and \eqref{eq_cond_exp_2} are sufficient to make the construction work both in the deterministic case with diffusion \eqref{eq-tr-diff} and in the ``pure transport'' stochastic case \eqref{eq_STE_proper}, it is natural to expect that the same conditions are sufficient also in the SPDE \eqref{eq:SPDE-diff} where both transport noise and diffusion are simultaneously acting. This is exactly the content of Theorem \ref{thm:spde-transp-diff}.
\end{enumerate}

\textit{Organisation of the paper:} Section \ref{sec_preliminaries} contains the preliminaries used throughout the paper, including important technical lemmas and the definition of antidivergence operators. The main result needed to prove Theorem \ref{thm_main_SPDE} are given in Section \ref{sec_main_results}. It comes in the form of an iteration proposition, cf. Proposition \ref{prop_iteration_stage}. We prove that this proposition easily implies Theorem \ref{thm_main_SPDE}. The rest of the paper is then devoted to proving the proposition.
Section \ref{sec_Mikados} defines the main building blocks used in the convex integration scheme. Section \ref{sec_perturbations} uses these building blocks to construct perturbations of the density and vector fields to a relaxed, approximative continuity defect equation, Eq. \eqref{eq_cont_defect} and gives estimates on these. Section \ref{sec_defect_estimates} then moves on to define and estimate all the error terms involved in \eqref{eq_cont_defect}. Section \ref{sec_momentum} contains the estimates on the distance of momenta, ensuring that $\rho u \in L^1$ in the limit. The final section, Section \ref{sec_proof_of_prop}, combines the results from the previous sections to prove the iteration proposition, which concludes the paper.  Section \ref{s:diffusion} gives a sketch of the proof in the case of the stochastic transport-diffusion equation. \\ Appendix \ref{sec_AppA} contains a useful interpolation inequality used several times to get estimates for the $W^{\theta,\tilde{p}}$-norm from estimates for the $L^{\tilde{p}}$ and $W^{1,\tilde{p}}$-norms of the iterates. Finally, Appendix \ref{sec_AppB} gives a sketch of a proof of uniqueness for the stochastic transport equation, in the (dotted) green/blue  areas of Fig. \ref{fig_main_result}, following the method of \cite{BFGM19}.

%******************************************************************************

\section{Preliminaries}\label{sec_preliminaries}

\subsection{Periodic maps}

For $n \in \N$, $n \geq 1$, $C_{0}^{\infty}(\T^{d};\R^{n})$ denotes the space of smooth functions with mean zero for all components. If $n=1$, we simply write $C_{0}^{\infty}(\T^{d})$.
If not stated otherwise, for a periodic function $f \colon \T^{d} \to \R^{n}$ and $\lambda \in \N$, 
$f_{\lambda} \colon \T^{d} \to \R$ will denote the dilation $f_{\lambda}(x) := f(\lambda x)$. Note that 
\begin{align*}
    \| D^{k} f_{\lambda} \|_{L^{p}(\T^{d})} = \lambda^{k} \| D^{k} f \|_{L^{p}(\T^{d})}, \quad \text{and in particular} \quad \| f_{\lambda} \|_{L^{p}_{x}} = \| f \|_{L^{p}_{x}}.
\end{align*}
Furthermore, let $\tau_{y}$, $y \in \R^{d}$ denote spatial translations, acting on functions according to
\begin{equation}\label{eq_def_translation}
    (f \circ \tau_{y}) := f(x-y). 
\end{equation}

\subsection{Stochastic notions}

All the stochastic bases $(\O,\F,\P,(\F_{t})_{t \geq 0})$ will include normal filtrations, i.e. $\mathcal{F}_{0}$ contains all $\P$-nullsets and $(\F_{t})_{t \geq 0}$ is a right-continuous filtration, i.e. $\F_{t} = \bigcap_{s > t} \F_{s}$ for all $t \in \R_{+}$. We will often suppress $\o$-dependence and generally write $f(t,x;\o)$ or just $f(t,x)$ for the evaluation of $f \colon \O \times I \times \T^{d} \to \R^{n}$ for a compact interval $I \subset \R$.

We give a few of the stochastic definitions used in the following: A map $X \colon \O \times [0,{1}] \to \R^{n}$ is adapted to the filtration $(\F_{t})_{t \in [0,{1}]}$ (or $(\F_{t})_{t}$-adapted)  if $X(t) \colon \O \to \R^{n}$ is $\F_{t}$-measurable for all $t \in [0,{1}]$. It is called progressively measurable with respect to the filtration $(\F_{t})_{t \in [0,{1}]}$ if for all $t \in [0,{1}]$, the restricted map  $X \colon \O \times [0,t] \to \R^{n}$ is $\F_{t} \otimes \mcB([0,t])$-measurable for all $t \in I$. Here, $\mcB([0,t])$ denotes the Borel $\sigma$-algebra of $[0,t]$. If a process is continuous and $(\F_{t})_{t}$-adapted, then it is also progressively measurable with respect to the same filtration (cf. \cite[Ch. 1, 1.13 Proposition]{KS91}). An  $\F$-measurable map $\tau \colon \O \to [0,{1}]$ is called a stopping time with respect to the filtration $(\F_{t})_{t \in [0,{1}]}$ (or an $(\F_{t})_{t}$-stopping time) if the event $\{ \tau \leq t \}$ belongs to $\F_{t}$ for every $t \in  [0,{1}]$. 

\subsection{Maps depending on space and time as trajectories in Banach spaces}

If $X$ is a Banach space and $I \subseteq \R$ is a compact interval, we can consider the Banach space $C(I; X)$ with norm $\|f\|_{C(I; X)} = \|f\|_{C_t X} = \max_{t \in I} \|f(t)\|_X$.   Similarly, if $\tau: \Omega \to  [0,{1}]$ is an $\F$-measurable function (in particular, if it is a stopping time), we can consider the Banach space
\begin{equation*}
L_{\o}^{\infty} C_{\tau}X := L^\infty(\Omega; C([0, \tau]; X) 
\end{equation*}
with norm 
\begin{equation*}
\|f\|_{L^\infty (\Omega; C([0, \tau]; X))} = \|f\|_{L^\infty_\omega C_\tau X} = \esssup_{\omega \in \Omega} \|f(\o) \|_{C([0,\tau(\omega)]; X)}.
\end{equation*}
Notice the difference between $C_{\tau} X$ and $C_{t} X$. The former space only consists of functions up to a random time $\tau = \tau(\o)$, whereas the latter consists of functions on the whole time interval $[0,1]$.

\subsection{Continuous maps depending on space and time}

Let $I \subseteq \R$ be a compact interval. Let $D := I \times \T^d$. In what follows, we assume that all functions take values in $\R$. Analogous definitions and properties hold if the maps under consideration take values in $\R^n$ (or $\T^d$, in the sense specified above).
The space $C^0(D)$ is the space of real valued continuous functions on $D$ with the $\sup$-norm. 
The space $C^k(D)$, $k \in \N$, is the space of all real valued, continuously differentiable maps on $D$ up to the order $k$, with the corresponding norm
\begin{equation*}
\|f\|_{C^k} := \sum_{j \leq k} \|\nabla^j_{(t,x)} f\|_{C^0},
\end{equation*}
where $\nabla^j_{(t,x)}$ denotes any derivative in space and time of order $j$. We define also $C^\infty(D) := \bigcap_k C^k(D)$ as a vector space (but we do not specify norms). Differential operators such as Laplacians $\D = \D_{x}$ or divergences $\divv = \divv_{x}$ will generally always be taken with respect to the space variables, except for the time derivative operator $\partial_{t}$. 

For $f \in C^0(D)$ and $\theta \in (0,1)$ we define the seminorm
\begin{equation*}
[f]_{\theta} := \sup_{z_1 \neq z_2} \frac{|f(z_1) - f(z_2)|}{|z_1 - z_2|^\theta}.
\end{equation*}
We denote by $C^\theta(D)$ is the space of continuous functions for which $[f]_\theta< \infty$. It has the natural norm $\|f\|_{C^\theta} \leq \|f\|_{C^0} + [f]_\theta$.  As before, if $\tau : \Omega \to  [0,{1}]$ is measurable, we can consider the Banach space 
$$
L_{\o}^{\infty} C_{\tau x} := L^\infty (\Omega; C([0,\tau] \times \T^d)
$$ 
with the natural norm $\|f\|_{L^\infty_\omega C_{\tau x}} := \esssup_{\omega \in \Omega} \|f (\o) \|_{C([0, \tau(\omega)] \times \T^d)}$, and similar notations applies to $L^\infty C^k_{\tau x}$ and $L^\infty C^\theta_{\tau x}$. We would like to point out once more the distinction between $C_{\tau x}$ and $C_{tx}$, as we did in the previous subsection.

We recall also the following two elementary facts.

\begin{lemma}
\label{l_holder1}
For $f,g \in C^\theta(D)$, $\|fg\|_{C^\theta} \leq \|f\|_{C^\theta} \|g\|_{C^\theta}$. 
\end{lemma}

\begin{lemma}
\label{l_holder2}
Let $\Phi \in C^\theta(D; D)$ for some $\theta \in (0,1)$. Let $f \in C^\theta(D)$. 
\begin{equation*}
[f \circ \Phi]_\theta \leq \|f\|_{C^1} [\Phi]_\theta. 
\end{equation*}
\end{lemma}

\subsection{Mollifications}

Let $I \subseteq \R$ be a compact interval. We denote by $D: = I \times \T^d$. Let $\chi$ be a mollification kernel in space and time. 
We recall the following two elementary estimates.

\begin{lemma}
\label{l_mollification_1}
Let $r  \in [1,\infty]$, let $f \in C(I; L^r(\T^d))$. Let $f_\e := f * \chi_\e$ (where we assumed that $f$ is extended by continuity for times $t \in \R \setminus I$, in order to define the mollification). It holds that
\begin{equation*}
\|f_\e\|_{C^k(D)} \leq C_{k} \e^{-k-d} \|f\|_{C_t L^1_x},
\end{equation*}
where $C_{k}$ is a constant depending only on the kernel $\chi$.
\end{lemma}

\begin{lemma}
\label{l_mollification_2}
Let $\theta \in (0,1]$, $f \in C^\theta(D)$. Let $f_\e := f * \chi_\e$ (where again we assumed that $f$ is extended by continuity for times $t \in \R \setminus I$, in order to define the mollification). Then 
\begin{equation*}
\|f_\e - f\|_{C^0(D)} \leq \e^{\theta} [ f ]_{C^\theta}. 
\end{equation*}
\end{lemma}

\subsection{Fractional Sobolev spaces on \texorpdfstring{$\T^d$}{Td}}

For $n \in \N$, $\theta \in \R_{+}$, $r \in (1,\infty)$, we define the fractional Sobolev spaces as
\begin{align*}
    W^{\theta,r}_{x} := W^{\theta,r}(\T^{d};\R^{n}) := \left\{ f \colon \T^{d} \to \R^{n} \mid \| f \|_{W^{\theta,r}_{x}} := \sum_{i=1}^{n} \left\| \left( (1+4\pi^{2} |k|^{2})^{\theta/2} \hat{f}_{i}(k) \right)^{\vee} \right\|_{L^{r}_{x}} < \infty \right\},
\end{align*}
where $\hat{f}_{i}(k) := \int_{\T^{d}} f_{i}(x) e^{-2\pi i k \cdot x} dx$ denotes the Fourier coefficient of the $i$-th component of $f$, and $u^{\vee}$ is the inverse Fourier transform defined by $u^{\vee}(x) := \hat{u}(-x)$. For $\theta \in \N_{0}$, these spaces are just the usual Sobolev spaces with an equivalent norm.

\subsection{Maps with values in the torus}

We say that a map $f: \T^d \to \R^d$ is a map with values in $\T^d$ (and we write $f: \T^d \to \T^d$) if
\begin{equation*}
f(x+k) = f(x) + k \text{ for all } k \in \Z^d. 
\end{equation*}
Similarly, if $A$ is a nonempty set, we say that a map $f: A \times \T^d \to \R^d$  is  a map with values in $\T^d$ (and we write $f : A \times \T^d \to \T^d$) if $f(a, \cdot) : \T^d \to \T^d$ for all $a \in A$. 
In particular we will consider maps $f : I \times \T^d \to I \times \T^d$, meaning with this notation that $f = (f_0, f_1, \dots, f_d)$ and $f_0 : I \times \T^d \to I$ and $(f_1, \dots, f_n) : I \times \T^d \to \T^d$. Observe that if $f : \T^d \to \T^d$ is differentiable, its derivative is a map $\nabla f: \T^d \to \R^{d \times d}$ (i.e. a standard map with values in some $\R^n$ space).

\subsection{The role of fast oscillations}

In this section we will collect a series of lemmas from earlier works which will be used many times in our arguments. 

\begin{lemma}[Improved H\"older inequality]\label{lem_Holder}
Let $f, g \colon \T^{d} \to \R$ be smooth maps. Let $r \in [1,\infty]$. Then
\begin{align*}
    \| f g_{\lambda}\|_{L^{r}} \leq \| f \|_{L^{r}} \|g\|_{L^{r}} + C_{r} \lambda^{-\frac{1}{r}} \|f\|_{C^{1}} \|g\|_{L^{r}}.
\end{align*}
\end{lemma}
\begin{proof}
We refer to Lemma 2.1 in \cite{MS18}.
\end{proof}

\begin{lemma}[Standard antidivergence operator]\label{lem_std_antidiv}
There is an operator
\begin{equation*}
\divv^{-1} : C^\infty_0(\T^d) \to C^\infty(\T^d; \R^d)
\end{equation*}
defined by $\divv^{-1} f := \nabla \Delta^{-1} f$, such that
\begin{equation*}
\divv (\divv^{-1} f) = f \text{ for all } f \in C^\infty_0(\T^d)
\end{equation*}
and
\begin{equation*}
\|\divv^{-1} f\|_{L^r} \leq C_r \|f\|_{L^r} \text{ for all } r \in [1,\infty],
\end{equation*}
where $C_r$ is a constant depending only on $r$. 
\end{lemma}
\begin{proof}
We refer to Lemma 2.2 in \cite{MS18} or to Lemmas 3.3 and 3.4 in \cite{MS20}. 
\end{proof}

\begin{lemma}[Improved antidivergence operators]
	\label{lem:adiv-properties}
	Let $ N\in\N \setminus \{0\}$. There is an operator
	\begin{equation*}
	\mathcal{R}_N : C^\infty(\T^d) \times C^\infty_0(\T^d) \to C^\infty(\T^d; \R^d)
	\end{equation*}
	such that for all $ f\in C^\infty(\T^d) $ and $ g\in C_0^\infty(\T^d) $.
	\begin{enumerate}
		\item $ \mathcal R_N $ is an anitdivergence operator in the sense that
		\[ \divv \left(\mathcal R_N (f,g)\right) = fg-\fint_{\T^d}fg. \]
		\item $ \mathcal R_N $ satisfies the Leibniz rule:
		\begin{equation}\label{eq_Leibniz}
		 \partial_j \left(\mathcal R_N(f,g)\right) = \mathcal R_N(\partial_jf,g) + \mathcal R_N(f,\partial_jg).
		\end{equation}
		\item 	if $ r\in[1,\infty] $, $ \lambda,N\in\N $, then
	\begin{align}
		\left\| \adiv_N(f,g_\lambda) \right\|_{L^r} &\le C_{d,p,N} \|g\|_{L^r} \left( \sum_{k=0}^{N-1} \lambda^{-k-1} \|\nabla ^k f\|_{L^\infty} + \lambda^{-N} \|\nabla ^Nf\|_{L^\infty}\right),
		\label{eq:antidiv-in-infty} \\
		\left\| \adiv_N(f,g_\lambda) \right\|_{L^r} &\le C_{d,p,N} \|g\|_{L^\infty} \left( \sum_{k=0}^{N-1} \lambda^{-k-1} \|\nabla^k f\|_{L^r} + \lambda^{-N} \|\nabla^Nf\|_{L^r} \right).
		\label{eq:antidiv-in-p}
	\end{align}
	\end{enumerate}
\end{lemma}
\begin{proof}
We refer to Lemma 3.5. in \cite{MS20} and the subsequent remark in the same paper. 
\end{proof}

\newpage
\section{The main Proposition and the proof of the main Theorem}\label{sec_main_results}

We start here the proof of Theorem \ref{thm_main_SPDE}. Let $d, \mfp, p, \tilde p, \theta$ and the Brownian motion $B$ be as in the statement of Theorem \ref{thm_main_SPDE}. 

\subsection{Choice of the stopping time \texorpdfstring{$\tau_\mathfrak{p}$}{Tp}}
\label{ssec_stopping_times}

Let $s \in (1, \infty)$ be such that
\begin{equation*}
\frac{1}{p} + \frac{1}{\tilde p} > \frac{1}{s} + \frac{1}{\tilde p} > 1+ \frac{\theta}{d}
\end{equation*}
and
\begin{equation*}
s' < d,
\end{equation*}
where $s'$ is the dual H\"older exponent to $s$, $1/s + 1/s' = 1$. Notice also that $s>p \geq 1$. 
Such a choice is always possible because of \eqref{eq_cond_exp} and \eqref{eq_cond_exp_2}, as is illustrated in Fig. \ref{fig_choice_s}.

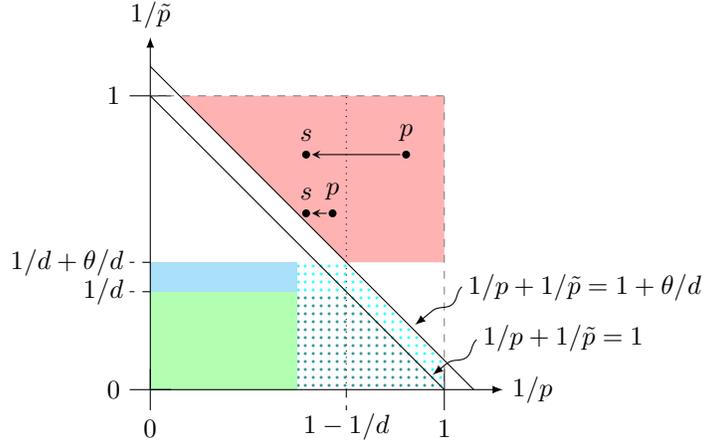
\begin{figure}[htbp]
  \centering
  \begin{tikzpicture}[scale=\textwidth/4.5cm]
    % Background layer
    \begin{scope}%[on background layer]
      \fill[green!30] (0,1/3) -- (1/2,1/3) -- (1/2,0) -- (0,0) -- cycle;
      \fill[pattern=dots, pattern color=teal!80] (1/2,1/3) -- (1-1/3,1/3) -- (1,0) -- (1/2,0) --  cycle;
      \fill[cyan!30] (0,1/3) --(0,1/3+0.3/3) -- (1/2,1/3+0.3/3) -- (1/2,1/3)  -- cycle;
      \fill[pattern=dots, pattern color=cyan!80] (1/2,1/3) --(1/2,1/3+0.3/3) -- (1-1/3,1/3+0.3/3) -- (1-1/3,1/3)  -- cycle;
      \fill[pattern=dots, pattern color=cyan!80] (1-1/3,1/3+0.3/3) -- (1-1/3,1/3) -- (1,0) -- (1,0.3/3)  -- cycle;

      %\fill[pattern=dots, pattern color=cyan!90] (0,1/3) --(0,1/3+0.3/3) -- (1,1/3+0.3/3) -- (1,1/3)  -- cycle;
      % Fill triangle
      \fill[red!30] (0.3/3,1) -- (1-1/3,1/3+0.3/3) -- (1,1/3+0.3/3) -- (1,1) -- cycle;
    \end{scope}
    
    % Axis lines 
    \draw[-latex] (0,0) -- (1.2,0) node[right] {$1/p$};
    \draw[-latex] (0,0) -- (0,1.2) node[above] {$1/{\tilde{p}}$};
    
    % Ticks and labels
    \foreach \x/\xtext in {0,1}
      \draw (\x,2pt) -- (\x,-2pt) node[below] {\xtext};
      
    \foreach \y/\ytext in {0,1}
      \draw (2pt,\y) -- (-2pt,\y) node[left] {\ytext};
    
    % Label and horizontal line with dash
    \node[left] at (-0.06,1/3) {$1/d$};
    \draw[dashed] (0,1/3) -- (-2pt,1/3);
        
    % Dashed gray lines
    \draw[dashed, gray] (0,1) -- (1,1);
    \draw[dashed, gray] (1,0) -- (1,1);
    
    % Solid black lines
    \draw (0,1) -- (1,0);
    \draw (0,1+0.3/3) -- (1+0.3/3,0);
    
    % Dotted black vertical line
    \draw[dotted, black] (1-1/3,0) -- (1-1/3,1);
    
    % Label for the point
    \node[left] at (-0.06,1/3+0.3/3) {$1/d + \theta/d$};
    \draw[dashed] (0,1/3+0.3/3) -- (-2pt,1/3+0.3/3);
    
    % Label and squiggly arrow
    \draw[-latex, black, decorate, decoration={snake,amplitude=1mm,segment length=8mm,post length=1mm}] (1.11,0.17) -- (0.96,0.05);
    \node[above right] at (1.1,0.1) {$1/p + 1/\tilde{p} = 1$};
    
    \draw[-latex, black, decorate, decoration={snake,amplitude=1mm,segment length=8mm,post length=1mm}] (1.06,0.25+0.3/3) -- (0.9,0.12+0.3/3);
    \node[above right] at (1.05,0.17+0.3/3) {$1/p + 1/\tilde{p} = 1 + \theta/d$};
    
    % Dashed label
    \node[below] at (1-1/3,-0.05) {$1-1/d$};
    \draw[dashed] (1-1/3,0) -- (1-1/3,-2pt);
    
    % Choice of s label
    \node[circle,fill=black, inner sep=0pt,minimum size=3pt, label = above:{$p$}] at (0.87,0.8) {};
    \node[circle,fill=black, inner sep=0pt,minimum size=3pt, label = above:{$s$}] at (0.53,0.8) {};
    \draw[black, -stealth] (0.85,0.8) -- (0.55,0.8);
    
    % Choice of s label
    \node[circle,fill=black, inner sep=0pt,minimum size=3pt, label = above:{$p$}] at (0.62,0.6) {};
    \node[circle,fill=black, inner sep=0pt,minimum size=3pt, label = above:{$s$}] at (0.53,0.6) {};
    \draw[black, -stealth] (0.60,0.6) -- (0.55,0.6);
    
    % Existence label
%     \node[left] at (0.66,0.2) {$!$};
%     \node[left] at (0.66,0.75) {$\neg !$};
  \end{tikzpicture}
  \caption{Illustration of the choice of the parameter $s$. If $p$ already corresponds to a point in the red triangle left of the $\frac{1}{p} = 1-\frac{1}{d}$ vertical line (dotted), we just need to choose $s$ within the triangle but to the left of $p$, see the lower pair of points. If $p$ corresponds to a point in the red area to the \textit{right} of the vertical line, we must choose $s$ to the left of $p$ in the red triangle to ensure $s' < d$, i.e. $\frac{1}{s} < 1 - \frac{1}{d}$, as illustrated for the upper pair of points.} \label{fig_choice_s} 
\end{figure}

Recall that $\kappa \in (0,1/2)$, describes the fixed deviation from $1/2$-H\"older continuity of Brownian motion. 
Let now $r : (0, 1/2) \to (1, \infty)$ be given by
\begin{align*}
    r(\kappa):= \frac{\frac{1}{2} + \kappa}{\frac{1}{2} - \kappa} > 1.
\end{align*}
Since $r$ is monotone increasing in $\kappa$ from 1 to $\infty$, we can choose $\kappa \in (0,1/2)$ such that
\begin{equation}
\label{eq_rkappa}
    r(\kappa) < \frac{d}{s'}.
\end{equation}
Let $N_{0} \subseteq \Omega$ be the set of paths of Brownian motion $t \mapsto B(t, \omega)$ which are $C^{1/2-\kappa}_{t}$-continuous. Clearly $\P(N_0) = 1$. Now, for any $L > 0$, we define the $(\mcF_{t})_{t}$-stopping time (cf. \cite[Lemma 3.5]{HZZ19})
\begin{equation}
\label{eq_stopping_time}
    \tau_{L}(\o) := \inf \{ t \geq 0 \mid [B(\o)]_{C^{1/2-\kappa}([0,t])} > L \} \wedge {1}, \quad \o \in \O,
\end{equation}
i.e. informally speaking, the first time $t \leq 1$ that the H\"older norm up to this time has surpassed $L$. By standard properties of the Brownian motion, it holds that $\tau_{L} > 0$, $\tau_L \to 1 $ almost surely as $L \to \infty$. Let thus $L = L(\mfp)$ be large enough such that
\begin{align*}
    \P \left( \tau_{L(\mfp)} = {1} \right) > \mfp.
\end{align*}
We thus define the stopping time $\tau : \Omega \to [0,1]$ by
\begin{equation*}
\tau = \tau_\mfp = \tau_{L(\mfp)}. 
\end{equation*}

\subsection{Space translations depending on the Brownian motion}
We introduce two $\omega$- and time dependent space translations of $\T^d$, due to translation originated by Brownian motion. More precisely, we define
\begin{equation}
\begin{aligned}
\Psi & : [0,1] \times \T^d \times \Omega \to [0,1] \times \T^d, &
\Psi(t,x; \omega) := (t, x + B(t; \omega))\\
\Psi^{-1} & :[0,1] \times \T^d \times \Omega \to [0,1] \times \T^d \times \Omega, &
\Psi^{-1}(t,x;\omega) := (t,x-B(t;\omega),\omega).
\end{aligned}
\end{equation}
Notice that $\Psi^{-1}$ is not really the inverse of $\Psi$, but nevertheless it holds
\begin{equation*}
\Psi \left( \Psi^{-1} (t,x;\omega) \right) = (t,x), \qquad \Psi^{-1} \left( \Psi(t,x; \omega); \omega\right) = (t,x;\omega) \quad \P-\text{a.s.}.
\end{equation*}
This motivates the (abuse of) notation. Notice also that a change of the values of $B$ on a null subset of $\Omega$ implies a change of the values of $\Psi$, $\Psi^{-1}$ on the same null subset of $\Omega$. Hence  $\Psi$ is well defined. Moreover, if $f \in C([0,1]; L^1(\T^d))$, the composition $f \circ \Psi$ (also denoted by $f(\Psi)$) is well defined. Similarly, if one takes $g \in L^\infty(\Omega; C([0,\tau]; L^1(\T^d))$, the composition $g \circ \Psi^{-1} = g(\Psi^{-1})$ is also well defined. 
\begin{remark}
It follows immediately from the definition that,
$\P$-a.s., for all $t \in [0,1]$, $\Psi$ and $\Psi^{-1}$ acts on the space variables as translations, in the sense that
\begin{equation*}
x \mapsto x+ B(t,\omega)  \text{ and } \ x \mapsto x - B(t,\omega) 
\end{equation*}
are translations.
\end{remark}

\begin{lemma}
\label{l_psi}
$\P$-a.s. the following hold.
\begin{enumerate}

\item If $b : [0,1] \times \T^d \to \R^d$ is a smooth vector field, then $\divv (b \circ \Psi) = (\divv b) \circ \Psi$.

\item If $f \in L^\infty([0,1]; L^r(\T^d))$ for some $r \in [1, \infty]$,  it holds that $\|(f \circ \Psi)(t)\|_{L^r(\T^d)}  = \|f(t)\|_{L^r(\T^d)}$ for all $t \in [0,1]$.

\item For the spatial derivatives it holds that, for all $t \in [0,1]$,  $\| \nabla \Psi(t, \cdot; \omega)\|_{C^k(\T^d)} \leq 1$ for all $k \in \N$, $k \geq 0$.

\item For the time-space H\"older norm on $[0, \tau(\omega)] \times \T^d$ it holds that
\begin{equation}
\label{eq_holder_psi}
[\Psi|_{[0, \tau(\omega)] \times \T^d}]_{1/2-\kappa} = \sup_{(t_1,x_1) \neq (t_2, x_2)} \frac{|\Psi(t_2, x_2; \omega) - \Psi(t_1, x_1; \omega)|}{ d_{\rm eucl} ((t_1, x_1), (t_2, x_2))^{ 1/2 - \kappa  } } 
 \leq L 
\end{equation}
where $d_{\rm eucl} ((t_1, x_1), (t_2, x_2)) = \sqrt{ |t_2 - t_1|^2 + |x_2 - x_1|^2 }$ is the Euclidean distance on $[0, \tau] \times  \R^d$ and  $L = L(\mfp)$ was fixed in Section \ref{ssec_stopping_times}. 
\end{enumerate}
\end{lemma}

\begin{proof}
The proof follows trivially from the previous remark and \eqref{eq_stopping_time}.
\end{proof}

\subsection{The continuity-defect equation}

Next we introduce the incompressible \textbf{continuity-defect equation}:
\begin{equation}\label{eq_cont_defect}
    \begin{cases}
                \partial_{t} {\rho}(t,x;\o) + {u}(t,x+B(t; \omega)) \cdot \nabla {\rho}(t,x;\o) & = - \divv  R(t, x; \omega), \\
                \divv {u} & = 0.
    \end{cases}
\end{equation}
Notice that the first equation can be also written using the map $\Psi$ above as
\begin{equation*}
\partial_t  \rho +  u( \Psi) \cdot \nabla  \rho = - \divv  R.
\end{equation*}
Observe that \eqref{eq_cont_defect} is not a stochastic PDE, but a \textit{random PDE}, in the sense that it is a deterministic PDE, where (some of) the unknowns (namely $\rho, R$) depend additionally on a parameter $\omega \in \Omega$. We will use the following notion of solution.

\begin{definition}
\label{def:cont-defect}
A solution to \eqref{eq_cont_defect} is a triple $(\rho, u, R)$ such that
\begin{equation*}
 \rho \in  L^\infty(\Omega; C^1([0, \tau] \times \T^d)), \quad
 u \in C^\infty([0,1] \times \T^d), \quad
 R \in L^\infty(\Omega; C^0([0, \tau]; L^1(\T^d)))
\end{equation*}
$\rho, R$ are $(\F_{t})$ adapted,  and, almost surely with respect to $\omega$, $(\rho, u, R)$ solves \eqref{eq_cont_defect} in the sense of distributions on $(0,1) \times \T^d$.
\end{definition}

\begin{remark}
The regularities we chose in Definition \ref{def:cont-defect} comes from the following observations:
\begin{enumerate}

\item for $\rho$, we will have to estimate the distance between $ \rho$ and its space-time mollification $ \rho_\e$ (see the proof of Lemma \ref{l_distance_rho} and of Lemma \ref{l_comm}), uniformly with respect to $\omega$: to do that, it is convenient to know that the $C^1_{\tau x}$-norm of $\rho$ is uniformly bounded with respect to $\omega$ (compare with Lemma \ref{l_mollification_2});
\item all vector fields constructed below will be deterministic and they will be always smooth by construction. Hence it is fine to require their smoothness in Definition \ref{def:cont-defect};
\item concerning the defect fields, we will always work with a (space-time) mollification of $ R$, and we will only need to know the the mollification can be estimated (in any $C^k$ norm) by a constant depending on $k$ times a norm of $ R$ that we can control uniformly with respect to $\omega$. Having thus a control of $ R$ in $C_t L^1_x$ is enough for this purpose (compare with Lemma \ref{l_mollification_1}). 

\end{enumerate}
\end{remark}

The following main proposition is the \textit{inductive step} in the proof of Theorem \ref{thm_main_SPDE}. 
  
\begin{proposition}\label{prop_iteration_stage}
    There exist an absolute constant $M > 0$ such that 
    the following holds. 
    For any $\delta> 0$ and any solution $({\rho}_{0}, u_{0}, R_{0})$ of the continuity-defect equation \eqref{eq_cont_defect} (in the sense of Definition \ref{def:cont-defect})
    there exists another solution $( \rho_{1},  u_{1}, R_{1})$ 
which satisfies the estimates
    \begin{align}
        \label{eq_rho1rho0Lp}\| {\rho}_{1} - {\rho}_{0} \|_{L^\infty_\omega C_t L^{p}_x} &\leq \delta, \\
        \label{eq_urho_L1} \| {\rho}_{1} {u}_{1} ( \Psi )- {\rho}_{0} {u}_{0} ( \Psi )\|_{L^\infty_\omega C_t L^{1}_x} &\leq M  \| R_{0} \|_{L^\infty_\omega C_t L^{1}_x} + \delta,  \\
        \label{eq_u1u0_Walpha} \| {u}_{1} - {u}_{0} \|_{L^\infty_\omega C_t W^{\theta,\tilde{p}}_x} &\leq \delta, \\
        \label{eq_R1_L1est} \| R_{1} \|_{L^\infty_\omega C_t L^{1}_x} &\leq \delta.  
    \end{align}
   Furthermore the following holds: 
    \begin{equation}
    \label{eq_fullsoln_prop}
    \begin{aligned}
        & \text{if there is $a \in (0,1]$ such that } 
        \rho_{0}, R_{0} = 0 \ \text{ $\P$-a.s.  on } [0,\min\{\tau(\omega),a\}] \times \T^d,  \\
     &   \text{then } \rho_{1}, R_{1} = 0 \ \text{ $\P$-a.s.  on } [0,\min\{\tau(\omega),a- \delta\}] \times \T^d  .
    \end{aligned}
    \end{equation}
\end{proposition} 

\subsection{Proof of Theorem \texorpdfstring{\ref{thm_main_SPDE}}{1.3}}
We will show that the main result of this work, Theorem \ref{thm_main_SPDE}, is easily implied by Proposition \ref{prop_iteration_stage}. As in this proof there will be both solutions to the random PDE \eqref{eq_cont_defect} as well as to the SPDE \eqref{eq_STE_proper}, we will denote the former by $\tilde{\rho}, \tilde{u}$ (resp. $\tilde{\rho}_n, \tilde{u}_n, \tilde{R}_n$) etc. and the latter by $\rho, u$ (resp. ${\rho}_n, {u}_n, {R}_n$).

Heuristically, what one wants to do after having solved equation \eqref{eq_cont_defect} would be to translate the equation to a stochastic transport equation \eqref{eq_STE} by applying It\^{o}'s formula as follows: Let $\tilde{\rho}, \tilde{u}$ be a suitable limit of solutions to the random PDE \eqref{eq_cont_defect} as $\tilde{R}_{n} \to 0$. Define the shifted solutions
\begin{equation*}
\rho := \tilde \rho(\Psi^{-1}), \quad u := \tilde u,
\end{equation*}
i.e.
\begin{align*}
    \rho(t,x,\o) := \tilde{\rho}(t,x-B(t,\o),\o), \quad u(t,x) := \tilde{u}(t,x).
\end{align*}
Note that $u$ is a deterministic function, whereas $\rho$ is random. Then by It\^{o}'s formula
\begin{align*}
    d\rho(t,x) &= d\tilde{\rho}(t,x-B(t)) \\
    &= \partial_{t} \tilde{\rho}(t,x-B(t))dt + \sum_{i=1}^{d} \partial_{x_{i}} \tilde{\rho}(t,x-B(t)) d(-B^{i}(t)) + \frac{1}{2} \sum_{i,j=1}^{d} \partial_{x_{i}}\partial_{x_{j}} \tilde{\rho}(t,x-B(t)) d[B^{i}, B^{j}](t) \\
    &\overset{\eqref{eq_cont_defect}}{=} -\tilde{u}(t,(x + B(t)) -B(t)) \cdot \nabla \tilde{\rho}(t,x-B(t)) dt - \sum_{i=1}^{d} \partial_{x_{i}} \tilde{\rho}(t,x-B(t)) dB^{i}(t) + \frac{1}{2} (\D \tilde{\rho})(t,x-B(t)) dt \\
    &= -u(t,x) \cdot (\nabla \tilde{\rho})(t,x-B(t)) dt - \nabla \tilde{\rho}(t,x-B(t)) \circ dB(t) \\
    &= - u(t,x) \cdot \nabla \rho(t,x) dt - \nabla \rho(t,x) \circ dB(t),
\end{align*}
and we see that $\rho$ solves Equation \eqref{eq_STE}.

There are, however, several problems with this heuristic approach:
\begin{enumerate}
 \item Most importantly, $\tilde{\rho}$ is not a $C_{t,x}^{1,2}$ function, as is required by It\^{o}'s formula. This is due to lacking spatial regularity, as the convergence guaranteed by Proposition \ref{prop_iteration_stage} is only in $C_t L^p_x$.
 \item $\tilde{\rho}, \tilde{u}$ are solutions to \eqref{eq_cont_defect} only \textit{in the weak sense}, not pointwise.
\end{enumerate}
To solve these problems, we do not apply It\^{o}'s formula to the limiting object $\tilde{\rho}$, but to the (smooth) iterates of the convex integration scheme, $\tilde{\rho}_{n}$. Furthermore, we don't apply it directly to the function $\rho_{n}(t,x) = \tilde{\rho}_{n}(t,x-B(t))$, but to the function tested against a suitable test function $\phi$.

\begin{proof}[Proof of Theorem \ref{thm_main_SPDE} assuming Proposition \ref{prop_iteration_stage}]

\textit{Initial stage.}
Let $\Phi \in C_{0}^{\infty}$ be a smooth, mean-zero function with $\| \Phi \|_{L^{p}} = 1$ and let $\chi \colon [0,{1}] \to \R_{+}$ be a smooth function such that
\begin{align*}
    \chi(t) \begin{cases}
                = 0, \quad t \in [0,{1}/3], \\
                \geq 0, \quad t \in ({1}/3, {2}/3), \\
                =1, \quad t \in [{2}/3,{1}].
            \end{cases}
\end{align*}
Define for $t \in [0,{1}], x \in \T^{d}$
\begin{equation}
    \tilde{\rho}_{0}(t,x; \omega) := \chi(t) \Phi(x), \quad \tilde{u}_{0}(t,x) := 0, \quad \tilde{R}_{0}(t,x; \omega) := \mathrm{div}^{-1} (\partial_{t} \rho_{0})(t,x; \omega) = \dot{\chi}(t) (\mathrm{div}^{-1} \Phi)(x).
\end{equation}
Then it is easy to check that the triple $(\tilde{\rho}_0, \tilde{u}_0, \tilde{R}_0)$ solves the continuity-defect equation \eqref{eq_cont_defect}, and that
\begin{equation}\label{eq_rho0_Lp}
    \tilde{\rho}_{0}(t,x) = \Phi(x) \text{ for } t \in [{2}/3,{1}], \text{ hence } \| \tilde{\rho}_{0}(t) \|_{L^{p}} = 1 \text{ for } t \in [{2}/3,{1}].
\end{equation}

\smallskip
\textit{Construction of the sequence by recursion.} 
We now let $(\delta_n)_n$ be a sequence of positive real numbers such that
\begin{equation*}
\sum_{n=1}^\infty \delta_n < \frac{1}{6}.
\end{equation*}
For any $n \geq 1$, we construct  $(\tilde{\rho}_{n}, \tilde{u}_{n}, \tilde{R}_{n})$ applying Proposition \ref{prop_iteration_stage} with parameter $\delta_n$ to $(\tilde{\rho}_{n-1}, \tilde{u}_{n-1}, \tilde{R}_{n-1})$.  
As $\delta_{n} \to 0$ and it is summable, we immediately find, using \eqref{eq_rho1rho0Lp}--\eqref{eq_fullsoln_prop} that the following convergences hold, for some $\tilde \rho, \tilde u$,
\begin{align}
    \label{eq_rhon_conv} \tilde{\rho}_{n} \to \tilde{\rho} \quad &\text{in } L^{\infty}\left(\O; C([0,\tau_{L}];L^{p}(\T^{d})) \right), \\
    \label{eq_un_conv} \tilde{u}_{n} \to \tilde{u} \quad &\text{in }C([0,{1}];W^{\alpha,\tilde{p}}(\T^{d};\R^{d})), \quad \text{and $\tilde{u}$ is deterministic,} \\
    \label{eq_rhon_un_conv} \tilde{\rho}_{n} \tilde{u}_{n}(\Psi) \to \tilde{\rho} \tilde{u}(\Psi) \quad &\text{in } L^{\infty}\left(\O; C([0,\tau_{L}];L^{1}(\T^{d};\R^{d})) \right), \\
    \label{eq_Rn_conv} \tilde{R}_{n} \to 0 \quad &\text{in } L^{\infty}\left(\O; C([0,\tau_{L}];L^{1}(\T^{d};\R^{d})) \right).
\end{align}
In addition, by \eqref{eq_rho1rho0Lp}, we observe explicitely that
\begin{equation}
\label{eq_distance_lp}
\|\tilde \rho_0 - \tilde \rho\|_{L^\infty_\omega C_t L^p_x} \leq \sum_{n=1}^\infty \delta_n < \frac{1}{6}.
\end{equation}
Moreover, observe that $\tilde \rho_0, \tilde R_0 = 0$ $\P$-a.s. on $[0, \min\{\tau(\omega), 1/3\}]$. Hence $\tilde \rho_1, \tilde R_1 = 0$ $\P$-a.s. on $[0, \min\{ \tau(\omega), 1/3 - \delta_1\}]$ and thus, recursively,
\begin{equation*}
\tilde \rho_n, \tilde R_n = 0 \quad \P-\text{a.s. on } \left[0, \min\left\{ \tau(\omega), \frac{1}{3} - \sum_{j=1}^n \delta_n \right\} \right] \supseteq \left[0, \min\left\{ \tau(\omega), \frac{1}{6}\right\} \right].
\end{equation*}
In particular 
\begin{equation*}
\tilde \rho|_{t=0} = 0 \quad  \P-\text{a.s.}. 
\end{equation*}
On the other hand, by \eqref{eq_distance_lp}, $\P$-a.s. on $\{ \tau =1\}$,
\begin{align*}
    \| \tilde \rho|_{t=1} \|_{L^{p}_{x}} \geq \| \tilde \rho_{0}|_{t=1} \|_{L^{p}_{x}} - \left\| (\tilde \rho_0 - \tilde \rho)|_{t=1} \right\|_{L^{p}_{x}} \geq 1 - \frac{1}{6} > 0.
\end{align*}
In particular (cf. the definition of $\tau$ in Sect. \ref{ssec_stopping_times}),
\begin{equation*}
\tilde \rho \neq 0 \text{ with probability } \geq \mfp.
\end{equation*}

\smallskip
\textit{Translated sequence.} 
We now define, for every $n = 0,1,2, \dots$
\begin{equation*}
\begin{aligned}
\rho_n & : = \tilde \rho_n (\Psi^{-1}), & u_n & := \tilde u_n, & R_n & := \tilde R_n(\Psi^{-1}), \\
\rho & : = \tilde \rho (\Psi^{-1}), & u & := \tilde u.
\end{aligned}
\end{equation*}
Note also that \eqref{eq_rhon_conv}-\eqref{eq_Rn_conv} imply that
\begin{align}
    \label{eq_rhon_conv_1} {\rho}_{n} \to {\rho} \quad &\text{in } L^{\infty}\left(\O; C([0,\tau_{L}];L^{p}(\T^{d})) \right), \\
    \label{eq_un_conv_1} {u}_{n} \to {u} \quad &\text{in }C([0,{1}];W^{\alpha,\tilde{p}}(\T^{d};\R^{d})), \quad \text{and $\tilde{u}$ is deterministic} \\
    \label{eq_rhon_un_conv_1} {\rho}_{n} {u}_{n} \to {\rho} {u} \quad &\text{in } L^{\infty}\left(\O; C([0,\tau_{L}];L^{1}(\T^{d};\R^{d})) \right), \\
    \label{eq_Rn_conv_1} {R}_{n} \to 0 \quad &\text{in } L^{\infty}\left(\O; C([0,\tau_{L}];L^{1}(\T^{d};\R^{d})) \right),
\end{align}
since the maps $\Psi$ and $\Psi^{-1}$ act only as a (random and time dependent) translation in the $x$-variable. In addition
\begin{equation*}
\rho|_{t=0} = 0 \quad  \P-\text{a.s.}
\end{equation*} 
and 
\begin{equation*}
\rho \neq 0 \text{ with probability } \geq \mfp.
\end{equation*}
If we can show that $\rho$ solves \eqref{eq_STE_proper} with velocity field $u$, the proof is concluded, as $\rho$ would be a nonzero solution to a linear equation, starting from zero initial data.

\smallskip
\textit{Weak solutions to the stochastic transport equation \eqref{eq_STE_proper}.} We now justify that our procedure produces a solution to the SPDE \eqref{eq_STE_proper} by applying It\^{o}'s formula along the iterations and using the convergences \eqref{eq_rhon_conv_1}--\eqref{eq_Rn_conv_1}.

Let $\phi \in C^{\infty}(\T^{d})$ be a test function and define
\begin{align*}
    F_{\phi,n}(t,y;\o) := \int_{\T^{d}} \tilde{\rho}_{n}(t,x+y;\o) \phi(x) dx.
\end{align*}
Note that the function $(t,x) \mapsto F_{\phi,n}(t,x)$ is in $C^{\infty}_{t,x}$, since $\tilde{\rho}_{n}$ is smooth. We then apply It\^{o}'s formula (cf. \cite[IV.$(3.12)$ Exercise, p. 152 f.]{RY99}) to find for each $n \in \N$ a $\P$-nullset $N_{n}$ such that (suppressing the $\o$-dependence again in most of what follows), for all $\o \in (N_{n})^{c}$ and $t \in [0,\tau(\o)]$
\begin{align*}
  &I_{1} + I_{2} := \int_{\T^{d}} {\rho}_{n}(t,x;\o) \phi(x) dx - \int_{\T^{d}} {\rho}_{n}(0,x;\o) \phi(x) dx = \int_{\T^{d}} \tilde{\rho}_{n}(t,x-B(t);\o) \phi(x) dx - \int_{\T^{d}} \tilde{\rho}_{n}(0,x;\o) \phi(x) dx  \\
  &\overset{\text{def.}}{=} F_{\phi,n}(t,-B(t);\o) - F_{\phi,n}(0,-B(0);\o) \\
  &\overset{\text{It\^{o}}}{=} \int_{0}^{t} (\partial_{t} F_{\phi,n})(s,-B(s))ds + \int_{0}^{t} \sum_{i=1}^{d} (\partial_{y^{i}} F_{\phi,n})(s,-B(s))d(-B^{i})(s) \\
  &\quad + \frac{1}{2} \sum_{i,j=1}^{d} \left( \partial_{y^{i}}\partial_{y^{j}} F_{\phi,n} \right)(s,-B(s)) d[-B^{i},-B^{j}](s) \\
  &\overset{\text{def.}}{=} \int_{0}^{t} \int_{\T^{d}} (\partial_{t} \tilde{\rho}_{n}(s,x-B(s)) \phi(x) dx ds - \int_{0}^{t} \int_{\T^{d}} \sum_{i=1}^{d} (\partial_{x^{i}} \tilde{\rho}_{n})(s,x-B(s)) \phi(x) dx ~dB^{i}(s) \\
  &\quad + \frac{1}{2} \int_{0}^{t} \int_{\T^{d}} (\D \tilde{\rho}_{n}(s,x-B(s)) \phi(x) dx ds \\
  &\overset{\eqref{eq_cont_defect}}{=} \int_{0}^{t} \int_{\T^{d}} \divv \left( \tilde{u}_{n}(s,(x+B(s))-B(s)) \tilde{\rho}_{n}(s,x-B(s)) \right) \phi(x) dx ds + \int_{0}^{t} \int_{\T^{d}} \divv \tilde{R}_{n}(s,x-B(s)) \phi(x) dx ds \\
  &\quad - \int_{0}^{t} \int_{\T^{d}} \sum_{i=1}^{d} (\partial_{x^{i}} \tilde{\rho}_{n})(s,x-B(s)) \phi(x) dx ~dB^{i}(s)  + \frac{1}{2} \int_{0}^{t} \int_{\T^{d}} ( \D \tilde{\rho}_{n})(s,x-B(s)) \phi(x) dx ds \\
  &\underset{\text{def.}}{\overset{\text{IBP,}}{=}} - \int_{0}^{t} \int_{\T^{d}}   \tilde{u}_{n}(s,x) \tilde{\rho}_{n}(s,x-B(s)) \cdot \nabla \phi(x) dx ds - \int_{0}^{t} \int_{\T^{d}} \tilde{R}_{n}(s,x-B(s)) \cdot \nabla \phi(x) dx ds \\
  &\quad + \int_{0}^{t} \int_{\T^{d}}   \tilde{\rho}_{n}(s,x-B(s)) \sum_{i=1}^{d} \partial_{x^{i}} \phi(x) dx ~dB^{i}(s)  + \frac{1}{2} \int_{0}^{t} \int_{\T^{d}} \tilde{\rho}_{n}(s,x-B(s)) \D \phi(x) dx ds \\
  &= - \int_{0}^{t} \int_{\T^{d}}   {u}_{n}(s,x) {\rho}_{n}(s,x) \cdot \nabla \phi(x) dx ds - \int_{0}^{t} \int_{\T^{d}} {R}_{n}(s,x) \cdot \nabla \phi(x) dx ds \\
  &\quad + \int_{0}^{t} \int_{\T^{d}}   {\rho}_{n}(s,x) \sum_{i=1}^{d} \partial_{x^{i}} \phi(x) dx ~dB^{i}(s)  + \frac{1}{2} \int_{0}^{t} \int_{\T^{d}} {\rho}_{n}(s,x) \D \phi(x) dx ds \\
  &=: I_{3} + I_{4} + I_{5} + I_{6}.
\end{align*}
Here, IBP means integration by parts. We deal with each term separately.

First, $I_{1},I_{2}$ are easy to deal with: using  \eqref{eq_rhon_conv_1}, 
\begin{align*}
    I_{1} &\to \int_{\T^{d}} \rho(t,x; \omega) \phi(x) dx, \\
    I_{2} &\to  \int_{\T^{d}} \rho(0,x; \omega) \phi(x) dx
\end{align*}
$\P$-a.s. In the same way, we find that, $\P$-a.s.,
\begin{align*}
    I_{6} \to 
    \frac{1}{2} \int_{0}^{t} \int_{\T^{d}} \rho(s,x) \D \phi(x) dx ds.
\end{align*}
For $I_{3}$, we use \eqref{eq_rhon_un_conv_1} to get, $\P$-a.s.
\begin{align*}
    I_{3} \to 
    - \int_{0}^{t} \int_{\T^{d}}   u(s,x) \rho(s,x) \cdot \nabla \phi(x) dx ds
\end{align*}
and, similarly, using \eqref{eq_Rn_conv_1} 
we find
\begin{align*}
    I_{4} \to 0,
\end{align*}
$\P$-a.s. The most interesting term is $I_{5}$. From Burkholder's inequality (e.g. \cite[Proposition D.0.1]{LR15}), we get
\begin{align*}
    &\E \left[ \sup_{t \in [0,\tau_{L}]} \left| \int_{0}^{t} \int_{\T^{d}} \left( {\rho}_{n}(s,x) - {\rho}(s,x) \right) \sum_{i=1}^{d} \partial_{x^{i}} \phi(x) dx ~dB^{i}(s) \right| \right] \\
    &\leq C\| \phi \|_{C^{1}} \E\left[ \int_{0}^{\tau_{L}}  \| {\rho}_{n}(s,\cdot) - {\rho}(s,\cdot) \|_{L^{1}}^{2} ds \right]^{1/2} \overset{n \to \infty}{\longrightarrow} 0.
\end{align*}
This means that we can find a subsequence $n_{k} \to \infty$ 
such that, $\P$-a.s.,
\begin{equation}\label{eq_BorelCantelli}
    I_{5}(n_{k}) \overset{k \to \infty}{\longrightarrow} \int_{0}^{t} \int_{\T^{d}} {\rho}(s,x) \sum_{i=1}^{d} \partial_{x^{i}} \phi(x) dx ~dB^{i}(s). % \quad \forall \o \in (N_{0} \cup \bar{N})^{c}.
\end{equation}
Combining the above convergence statements along that subsequence shows that $\rho, u$ solve the stochastic transport equation \eqref{eq_STE_proper} in the weak sense, and, as observed before, this concludes the proof of the theorem.
\end{proof}

The rest of the paper will now be devoted to the proof of the main iteration, Proposition \ref{prop_iteration_stage}.

\section{Definition of the Mikado densities and fields}\label{sec_Mikados}

In this section, we will define the (deterministic) building blocks used in the convex integration scheme. At the beginning of Section \ref{sec_main_results}, we have already fixed the exponent $s$ such that 
\begin{equation}\label{eq_main_scaling_cond_s}
    \frac{1}{s} + \frac{1}{\tilde{p}} > \frac{1}{s} + \frac{1}{\tilde{p}} > 1 + \frac{\theta}{d}.
\end{equation}
and
\begin{equation*}
s'<d. 
\end{equation*}
This will ensure that
\begin{align*}
    \rho \in L_{x}^{s} \subset L_{x}^{p}, \quad u \in L_{x}^{s'} \cap W_{x}^{\theta,\tilde{p}}.
\end{align*}
The choice has been illustrated in Fig. \ref{fig_choice_s} above.

\begin{remark}
    Let us remark that, differently than in the deterministic scheme of \cite{MS20}, the concentration of $\rho$ in $L^{p}_{x}$ seems to be not sufficient to make the convex integration scheme work in this paper. We will briefly comment on this. As the definition of the density perturbation $\vt^{\mathrm{det}}$ in the deterministic work \cite{MS20} contained the exactly right power of the modulus of $R_{0}$ (cf. Eq. \eqref{eq_def_vt_MS20} below), raising it to the $p$-th power did not give something that was arbitrarily small but roughly speaking, $\| \vt^{\mathrm{det}} \|_{L^{p}} \lesssim \| R_{0} \|_{L^{1}}$, and the latter quantity decreases \textit{along the iteration}. In our case, however, this will not be the case as we need to put all the $R_{0}$-dependence into $\vt$, therefore we would only get $\| \vt \|_{L^{p}} \lesssim \| R_{0} \|_{L^{\infty}}$, and we cannot control the latter quantity. 
    Therefore we needed to resort to demanding more integrability for $\rho$ at the expense of less integrability for $u$. In particular, it does not follow from H\"older's inequality anymore that $\rho u \in L_{x}^{1}$, a condition needed for the transport equation, which is in divergence form, to be well-defined. However, we will prove directly that $\rho u \in L_{x}^{1}$, cf. Eq. \eqref{eq_rhon_un_conv}.
\end{remark}
We take the blobs of \cite[Lemma 4.1, Lemma 4.3]{MS20}, i.e. we take suitable $\zeta_{j} \in \R^{d}$ to define periodic functions
\begin{align*}
    \varphi_{\mu}^{j} = \vphi_{\mu}(\cdot - \zeta_{j}) \colon \T^{d} \to \R, \quad \tilde{\varphi}_{\mu}^{j} = \tilde{\vphi}_{\mu}(\cdot - \zeta_{j}) \colon \T^{d} \to \R.
\end{align*}
To this end, we take $\vphi \colon \R^{d} \to \R$ smooth with 
\begin{equation}\label{eq_intphi}
    \supp \vphi \subset B(P,r) \subset (0,1)^{d}, \quad P = (1/2, \ldots, 1/2) \in (0,1)^{d}, \quad \int_{\R^{d}} \vphi^{2} dx = 1,
\end{equation}
and the concentrated functions 
\begin{align*}
    \vphi_{\mu}(x) := \mu^{d/s} \vphi(\mu x), \quad \tilde{\vphi}_{\mu}(x) := \mu^{d/s'} \vphi(\mu x).
\end{align*}
Periodising and shifting the functions then yields the above periodic functions, and we have the (obvious) scaling behaviour 
\begin{align}\label{eq_Mikado_scaling}
    \| \nabla^{k} \vphi^{j}_{\mu} \|_{L^{r}} = \mu^{\frac{d}{s} - \frac{d}{r} + k} \| \nabla^{k} \vphi \|_{L^{r}}, \quad \| \nabla^{k} \tilde{\vphi}^{j}_{\mu} \|_{L^{r}} = \mu^{\frac{d}{s'} - \frac{d}{r} + k} \| \nabla^{k} \vphi \|_{L^{r}},
\end{align}
and, moreover, 
\begin{equation}
\label{eq_varphi_supp_disj}
\varphi^i_\mu ( x - t e_i) \tilde \varphi^j_\mu ( x - t e_j) = 0
\end{equation}
for all $i \neq j$. The fact that such shifting is possible is proven in \cite[Lemma 4.3]{MS20}. Finally, we will need a stationary Mikado flow \cite{DS17, MS18,MS19,MS20}, i.e. a smooth, periodic function $\psi \colon \T^{d-1} \to \R$ satisfying 
\begin{equation}\label{eq_intpsi}
    \fint_{\T^{d-1}} \psi = 0, \quad \fint_{\T^{d-1}} \psi^{2} = 1.
\end{equation}
We then define
\begin{align*}
    \psi^{j} \colon  \T^{d} \to \R, \ \psi^{j}(x)  := \psi^{j}(x_{1},\ldots,x_{d}) := \psi(x_{1},\ldots, \hat{x}_{j}, \ldots, x_{d}), 
\end{align*}
where by $\hat{x}_{i}$ we denote the omission of the $i$-th component of $x$. Since $\psi^j$ does not depend on the $j$-th coordinate, we can also find smooth vector fields $b^j : \T^d \to \R^d$ such that
\begin{equation}\label{eq_def_potential}
\divv b^j = \psi^j \qquad \text{ and } \qquad \partial_j b^j = 0,
\end{equation}
i.e. $b^j$ does not depend on the $j$-th coordinate as well. 

As in \cite[p. 15]{MS20}, we let the functions oscillate with (high) frequency $\lambda \in \N$, $\lambda \ll \mu$, shift the blobs $\vphi^{j}_{\mu}(\lambda \cdot)$, $\tilde{\vphi}^{j}_{\mu}(\lambda \cdot)$ in time with phase speed $\sigma$ and slice them with the stationary Mikado oscillating at very high spatial frequency $\nu \in \lambda \N$, $\nu \gg \lambda$ to arrive at
\begin{alignat*}{4}
    &\text{Mikado density} \quad &\Theta^{j}(t,x) &:= \Theta^{j}_{\lambda, \mu,\sigma, \nu}(t,x) &&:= \vphi^{j}_{\mu}(\lambda(x - \sigma t e_{j})) \psi^{j}(\nu x) = (\varphi_\mu^j)_\lambda \circ \tau_{\sigma t e_j} \psi^j_\nu , \\
    &\text{Quadratic density corrector} \quad & Q^{j}(t,x) &:= Q^{j}_{\lambda,\mu,\sigma,\nu}(t,x) &&:= \frac{1}{\sigma} \left( \vphi^{j}_{\mu} \tilde{\vphi}^{j}_{\mu} \right)(\lambda(x - \sigma t e_{j}) \left( \psi^{j}(\nu x) \right)^{2} \\
    & & & &&= (\varphi_\mu^j \tilde \varphi_\mu^j)_\lambda \circ \tau_{\sigma t e_j} (\psi^j_\nu)^2 \\
    &\text{Mikado field} \quad &W^{j}(t,x) &:= W^{j}_{\lambda,\mu,\sigma,\nu}(t,x) &&:= \tilde{\vphi}^{j}_{\mu}(\lambda(x - \sigma t e_{j})) \psi^{j}(\nu x) e_{j} = (\tilde \varphi_\mu^j)_\lambda \circ \tau_{\sigma t e_j} \psi^j_\nu e_j, \\
    &\text{Mikado field corrector} \quad & W^{j, \textrm{corr}}(t,x) &:= W^{j, \textrm{corr}}_{\lambda,\mu,\sigma,\nu}(t,x) &&:=
    \frac{\lambda}{\nu} \left[ e^j \otimes b^j - b^j \otimes e^j \right](\nu x) \cdot (\nabla \tilde \varphi_\mu)^T_{\lambda} \circ \tau_{\sigma t e_j}.
\end{alignat*}
\newpage
Note that the ``Mikado density'' and the ``Mikado field'' are not stationary solutions of the incompressible transport equation, but one can use the time-dependence of the corrector $Q^{j}$ to cancel the quadratic term coming from multiplication of $u_{n}$ and $\rho_{n}$, cf. \eqref{eq_cont_eq_Q} in the following proposition. The exact relationships between the parameters $\lambda, \mu, \nu, \sigma$ will be given later in Section \ref{ssec_parameter_choice}. We also define the family of vector fields (depending on $N \in \N$ and on $\lambda, \mu, \sigma, \nu$)
\begin{alignat*}{4}
&\text{Potential density} \quad & A^{j}_N(t,x) &:= A^{j}_{N, \lambda, \mu,\sigma, \nu}(t,x) &&:=
 \lambda \sigma \mathcal{R}_N \left( (\partial_j \varphi_\mu^j)_\lambda \circ \tau_{\sigma t e_j}, \psi_\nu^j \right).
\end{alignat*}
Note that, at this stage, all of the building blocks are deterministic objects.

\begin{proposition}[Equations]
\label{prop_mikado_eqn}
The Mikado functions ``solve the continuity equation'' in the sense that for every $j \in \{1, \ldots, d\}$ and $N \in \N$,
\begin{equation}
\label{eq_cont_eq_Q}
\begin{aligned}
    \partial_{t} Q^{j}_{\lambda, \mu,\sigma, \nu} + \divv \left( \Theta^{j}_{\lambda, \mu,\sigma, \nu} W^{j}_{\lambda, \mu,\sigma, \nu} 
   \right) & = 0, \\
\partial_t \Theta^j & = - \divv A^j_N, \\
\divv \left(W^j + W^{j, \textrm{corr}} \right) & = 0.
\end{aligned}
\end{equation}
on $[0,1]\times \T^{d}$. 
\end{proposition}
\begin{proof}
The first equality follows as in \cite{MS20}. Concerning the second equality, we have
\begin{equation*}
\partial_t \Theta^j (t,x) = - \lambda \sigma \partial_j \varphi_\mu^j (\lambda (x - \sigma t e_j)) \psi^j(\nu x),
\end{equation*}
whereas by definition of $\mathcal{R}_N$ we have
\begin{equation*}
\divv A_N (t,x) =  \lambda \sigma \partial_j \varphi_\mu^j( \lambda (x - \sigma t e_j)) \psi^j(\nu x) - \fint \lambda \sigma \partial_j \varphi_\mu^j( \lambda (x - \sigma t e_j)) \psi^j(\nu x) dx = - \partial_t  \Theta^j + \underbrace{\fint \partial_t \Theta^j}_{=0} = - \partial_t \Theta^j, 
\end{equation*}
where the last equality follows from the fact that $\Theta^j(t,x) = f(x - \sigma t e_j)$ for some smooth function $f$ and thus its spatial mean value is constant in time. Finally, the third equality in \eqref{eq_cont_eq_Q} follows from the observation that, by Eq. \eqref{eq_def_potential},
\begin{equation*}
(W^j + W^{j, \textrm{corr}})(t,x) = \divv \left(  \frac{1}{\nu} \tilde \varphi_\mu (\lambda (x - \sigma t e_j)) \left( e^j \otimes b^j(\nu x) - b^j(\nu x) \otimes e^j   \right) \right),
\end{equation*}
and the fact that $\divv \divv S = 0$ if $S$ is a matrix field taking values in the set of antisymmetric matrices.
\end{proof}

The following estimates will be used repeatedly in the sequel. 
\begin{proposition}[Estimates]
\label{prop_Mikado_est}
Assume that
\begin{equation}
\label{eq_lambdamunu}
\frac{\lambda \mu}{\nu} \leq \frac{1}{2}.
\end{equation}
For every $k \in \N$ and $r \in [1,\infty]$ there is a constant $M_{k,r}$ depending only on $k,r,\varphi, \tilde \varphi, \psi$ such that
\begin{align}
\label{eq_Theta}
\|\nabla^k  \Theta^j_{\lambda, \mu, \sigma, \nu}\|_{L^\infty_t L^r_x} & \leq M_{k,r} \mu^{\frac{d}{s} - \frac{d}{r}} \nu^k, \\
\label{eq_Q}
\|\nabla^k Q^j_{\lambda, \mu, \sigma, \nu}\|_{L^\infty_t L^r_x} & \leq M_{k,r} \frac{\mu^{d - \frac{d}{r}} \nu^k}{\sigma}, \\
\label{eq_W}
\|\nabla^k W^j_{\lambda, \mu, \sigma, \nu}\|_{L^\infty_t L^r_x} & \leq M_{k,r} \mu^{\frac{d}{s'} - \frac{d}{r}} \nu^k, \\
\label{eq_Wcorr}
\|\nabla^k W^{j, \textrm{corr}}_{\lambda, \mu, \sigma, \nu}\|_{L^\infty_t L^r_x} & \leq M_{k,r} \frac{\lambda \mu}{\nu} \mu^{\frac{d}{s'} - \frac{d}{r}} \nu^k.
\end{align}
Similarly, for every $k \in \N$, $r \in [1, \infty]$ and $N \in \N$, there is a constant $M_{k,r,N}$ depending only on $k,r,N,\varphi, \tilde \varphi, \psi$ such that
\begin{align}
\label{eq_A}
\|\nabla^k  A^j_{N,\lambda, \mu, \sigma, \nu}\|_{L^\infty_t L^r_x} & \leq M_{k,r,N} \frac{\lambda \mu \sigma}{\nu} \mu^{\frac{d}{s} - \frac{d}{r}} \nu^k \left( 1 + \frac{(\lambda \mu)^N}{\nu^{N-1}} \right).
\end{align}
\end{proposition}
\begin{proof}
We prove \eqref{eq_Theta}. If $\nabla^k$ denotes a derivative of order $k$, we have, by the Leibniz rule,
\begin{equation*}
\begin{aligned}
\nabla^k \Theta^j(t,x) & = \sum_{\substack{k_1, k_2 \\ 0 \leq k_1 + k_2 \leq k}} \nabla^{k_1} \left( \varphi^j_\mu ( \lambda (x - \sigma t e_j)) \right) \nabla^{k_2} \left( \psi (\nu x) \right) \\ 
& = \sum_{\substack{k_1, k_2 \\ 0 \leq k_1 + k_2 \leq k}} \lambda^{k_1} \nabla^{k_1} (\varphi^j_\mu) ( \lambda (x - \sigma t e_j)) \nu^{k_2} (\nabla^{k_2} \psi)(\nu x).
\end{aligned}
\end{equation*}
Hence, computing the $L^r_x$ norm, 
\begin{equation*}
\begin{aligned}
\| \nabla^k \Theta^j(t, \cdot)\|_{L^r_x} 
& \leq \sum_{\substack{k_1, k_2 \\ 0 \leq k_1 + k_2 \leq k}} \lambda^{k_1} \| \nabla^{k_1} \varphi^j_\mu\|_{L^r_x} \nu^{k_2} \|\nabla^{k_2} \psi\|_{L^\infty} \\
& \leq M_{k,r} \sum_{\substack{k_1, k_2 \\ 0 \leq k_1 + k_2 \leq k}} (\lambda \mu)^{k_1} \mu^{d/s - d/r} \nu^{k_2}
 \\
\text{(since $\lambda \mu \leq \nu/2$)} & \leq M_{k,r} \nu^k \mu^{d/s-d/r}.
\end{aligned}
\end{equation*}
Estimates \eqref{eq_Q}-\eqref{eq_Wcorr} can be treated similarly. Concerning \eqref{eq_A}, for $k=0$, we apply \eqref{eq:antidiv-in-p} and we get
\begin{equation*}
\begin{aligned}
\|A_N^j(t, \cdot)\|_{L^r} 
& \leq M_{N} \lambda \sigma \left( \sum_{k=0}^{N-1} \frac{1}{\nu^{k+1}} \|\nabla^k (\partial_j \varphi^j_\mu)_\lambda\|_{L^r} \|\psi\|_{L^\infty} + \frac{1}{\nu^N} \|\nabla^N \partial_j \varphi_\mu^j\|_{L^r} \|\psi\|_{L^\infty} \right) \\
& \leq M_{r,N} \lambda \sigma \left( \sum_{k=0}^{N-1} \frac{(\lambda \mu)^{k}}{\nu^{k+1}}  \mu^{1 + d/s - d/r} + \frac{(\lambda \mu)^N}{\nu^N} \mu^{1 + d/s - d/r}  \right) \\
& \leq M_{r,N} \frac{\lambda \mu \sigma}{\nu} \mu^{d/s - d/r} \left( \underbrace{\sum_{k=0}^{N-1} \left(\frac{\lambda \mu}{\nu} \right)^k}_{\leq 2} + \frac{(\lambda \mu)^N}{\nu^{N-1}} \right) \\
\text{(by \eqref{eq_lambdamunu})} & \leq M_{r,N} \frac{\lambda \mu \sigma }{\nu} \mu^{d/s - d/r} \left( 1 + \frac{(\lambda \mu)^N}{\nu^{N-1}} \right).
\end{aligned}
\end{equation*}
Estimate \eqref{eq_Q} for higher order derivatives can be obtained similarly, using the fact that the operator $\mathcal{R}_N$ satisfies the Leibniz rule \eqref{eq_Leibniz}.
\end{proof}

\begin{proposition}[Interactions between different Mikados]
\label{prop_mikado_interactions}
For every $i \in \{1, \dots, d\}$ and $t \in [0,{1}]$
\begin{equation*}
\fint_{\T^d} \Theta^j(t,x) W^j(t,x) dx = e_j,
\end{equation*}
and for every $i \neq j$, $t \in [0,{1}]$, $x \in \T^{d}$
 \begin{equation}\label{eq_Mikado_supp_disjoint}
    \Theta^{i}_{\lambda, \mu,\sigma, \nu}(t,x) W^{j}_{\lambda, \mu,\sigma, \nu}(t,x) = 0.
 \end{equation}
\end{proposition}
\begin{proof}
The mean value property follows from the definition of $\Theta^j, W^j$, whereas \eqref{eq_Mikado_supp_disjoint} follows from \eqref{eq_varphi_supp_disj}.
\end{proof}

\newpage

\section{Defining the perturbations}\label{sec_perturbations}

In this section we start the proof of the main Proposition \ref{prop_iteration_stage}, defining the new density $\rho_1$ and the new velocity field $u_1$. We will employ a convex integration scheme similar to the one of \cite{MS18,MS19,MS20}. We stress again that we want our velocity field ${u}$ to be deterministic, since this will imply that we have nonuniqueness for the stochastic transport equation with a \textit{deterministic} drift, which is the setting where the strongest uniqueness results have been obtained to date. To this end, we need to make sure that all the parameters going into the definition of ${u}$ do not depend on $\o$.

\subsection{Differences from the deterministic scheme}

Our starting point is the auxiliary equation
\begin{align*}
    \partial_{t} {\rho} + {u}(\Psi) \cdot \nabla {\rho} &= -\divv R.
%    \rho(0,x,\o) &= \rho_{0}(x,\o).
\end{align*}
Recall that $\Psi(t,x; \omega) = (t, x+ B(t; \omega))$. 
Here it is crucial that the vector field ${u}$ is deterministic and the only $\o$-dependence of it comes via the uniform translation by $B(t;\o)$ in the spatial argument.

We start our discussion from the ``usual'' convex integration scheme of \cite{MS20}, where $\rho_1 \approx \rho_0 + \vartheta$, $u_1 \approx u_0 + w$ and the error term is distributed between the perturbations according to
\begin{equation}
     \label{eq_def_vt_MS20} 
    \begin{split}
      \vt(t,x) &:= \sum_{j=1}^{d} \sign\left( R_{0}^{j}(t,x) \right) \left| R_{0}^{j}(t,x)\right|^{\frac{1}{p}} \Theta^{j}_{\lambda, \mu,\sigma, \nu}(t,x), \\
     w(t,x) &:= \sum_{j=1}^{d} \left| R_{0}^{j}(t,x)\right|^{\frac{1}{p'}} W^{j}_{\lambda,\mu,\sigma,\nu}(t,x),
    \end{split}
\end{equation}
where $
\lambda, \mu, \sigma, \nu
$
are parameters to be fixed later, which guarantee that estimates \eqref{eq_rho1rho0Lp}-\eqref{eq_R1_L1est} can be achieved. 
The sign function introduces a discontinuity which needs to be amended using suitable cutoff functions $\chi_{j}$ which we will omit at this stage for ease of presentation (and also because we won't need them in the end, cf. Remark \ref{rem_cutoff}).

We have three requirements on our perturbations, namely that
\begin{enumerate}[label = (\roman*)]
 \item $w$ needs be deterministic.
 \item\label{itm_req_pert_2} The density perturbation needs to be regular enough in time to take a derivative.
 \item\label{itm_req_pert_1} The product $\vt w$ needs to cancel the previous error term $R_{0}$.
\end{enumerate}
A definition like the one in \eqref{eq_def_vt_MS20} is thus not suitable because
\begin{enumerate}[label = (\roman*)]
\item $w$ depends on $R_0$, and $R_0$ is not deterministic (it depends on $\omega$).
\item $\vt$ is not differentiable in time, not only because of the $\sign$ and the modulus, but mostly because $R_0$ is assumed to be only  \textit{continuous in time with values in $L^1_x$} (and, indeed, the defect fields constructed along the iteration are not differentiable in time, because of the presence of the Browinian motion $B(t)$).  
Moreover, we only have (inductively) an estimate for $R_0$ in $L^\infty_\omega C_t L^1_x$, but we know nothing about e.g. estimates of $R_0$ in different integrability (e.g. $L^\infty_x$). In particular, even though by the assumed space-time continuity of $R_0$, we know that, $\P$-a.s., $\|R_0(\cdot; \omega)\|_{C_t L^\infty_x} < \infty$, we don't know if $\esssup_\omega \|R_0(\cdot; \omega)\|_{C_t L^\infty_x}$  is still bounded, and this prevents us from estimating e.g. the $L^\infty_\omega C_t L^1_x$-norm of $\rho_1$.

\item In the quadratic term of the continuity-defect equation we are considering, $u$ appears only composed with the translation $\Psi$. Hence, in order for $\vt w$ to cancel the previous error term,  the Mikado density $\Theta$ used in the definition of $\vt$ must also be computed  in $(t,x + B(t;\omega); \omega)$ and not in = $(t,x,\omega)$. This however would again make $\vt$ not differentiable in time. 
\end{enumerate}
These three issues are solved as follows.
\begin{enumerate}[label = (\roman*)]
\item All the $R_0$-dependence is put in $\vt$, and none in $w$.
\item We do not use $R_0$, but a space-time mollification of it, with mollification parameter $\e>0$. In this way the mollification of $R_0$ is smooth and every norm of the mollified object can be estimated in terms of  $\|R_0\|_{L^\infty_\omega C_t L^1_x}$ and (a negative power of) $\e$.

\item The Mikado density $\Theta$ is not computed in $(t,x + B(t;\omega); \omega)$, but in $(t,x + B_\ell(t;\omega); \omega)$, where $B_\ell$ is a mollification of the Browinian motion, with mollification parameter $\ell$. 
\end{enumerate}

We stress strongly that there are two mollification parameters: $\e$ and $\ell$. The first one, namely $\e$, is used in the mollification of $R_0$ and will be fixed essentially requiring that the distance in $L^1_x$ of $R_0$ and its mollification is small enough. In particular it will be fixed \textit{before} fixing $\ell$ and the parameters $\lambda, \mu, \sigma, \nu$ (see Section \ref{ssec_parameter_choice}).

The second one, namely $\ell$, is used in the mollification of the Brownian motion and will be chosen in such a way that the five parameters $\ell, \lambda, \mu, \sigma, \nu$ satisfy certain balances (see Section \ref{ssec_parameter_choice}). 

\subsection{\texorpdfstring{$\e$}{epsilon}-mollification of \texorpdfstring{$\rho_0$}{r0} and \texorpdfstring{$R_0$}{R0}}

Before finally defining $\rho_1$ and $u_1$, we mollify the various objects as announced. Since $\rho_0, R_0$ are defined only from time $0$  up to the stopping time $\tau = \tau(\omega)$, we extend them before time $0$ and after time $\tau(\omega)$ by simply setting 
\begin{equation*}
\rho_0(t, \cdot; \omega) =
\begin{cases}
\rho_0(\tau(\omega), \cdot; \omega) & \text{if } t \geq \tau(\omega), \\
\rho_0(0, \cdot; \omega) & \text{if } t \leq 0,
\end{cases}
\quad
R_0(t, \cdot; \omega) =
\begin{cases}
R_0(\tau(\omega), \cdot; \omega) & \text{if } t \geq \tau(\omega), \\
R(0, \cdot; \omega) & \text{if } t \leq 0.
\end{cases}
\end{equation*}
Clearly the extensions still belong respectively to
\begin{equation*}
\rho_0 \in  L^\infty(\Omega; C^1(\R \times \T^d)), \quad
R_0 \in L^\infty(\Omega; C(\R; L^1(\T^d)))
\end{equation*}
and the two fundamental estimates are preserved:
\begin{equation*}
\|\rho_0\|_{L^\infty(\Omega; C^1(\R \times \T^d))} \leq \|\rho_0\|_{L^\infty(\Omega; C^1([0, \tau] \times \T^d))}, \quad
\|R_0\|_{L^\infty(\Omega; C(\R; L^1(\T^d)))} \leq \|R_0\|_{L^\infty(\Omega; C([0, \tau]; L^1(\T^d)))}.
\end{equation*}
Let $(\eta_\e)_{\e}$ be a space-time mollifier that is a product of a space-mollifier $(\eta^{(1)}_{\e})_{\e}$ and a time-mollifier $(\eta^{(2)}_{\e})_{\e}$ with $\supp \eta^{(2)} \subseteq (0,1)$. This implies  
with $\supp \eta_\e \subseteq (0,1) \times \R^d$ (to ensure adaptedness). Let us now define
\begin{equation*}
\rho_\e := \rho_0 * \eta_\e, \qquad R_\e := R_0 * \eta_\e. 
\end{equation*}

\begin{lemma}
\label{l_mollified_density}
It holds that
\begin{equation*}
\|\rho_\e\|_{L^\infty_\omega C^k_{\tau x}} \leq \e^{-k} \|\rho_0\|_{L^\infty_\omega C_{\tau x}},
\end{equation*}
and
\begin{equation*}
\|\rho_\e - \rho_0\|_{L^\infty_\omega C_{\tau x}} \leq \e \| \rho\|_{L^\infty_\omega C^1_{\tau x}}.
\end{equation*}
\end{lemma}

\begin{lemma}
\label{l_mollified_error}
It holds that
\begin{equation*}
\|R_\e\|_{L^\infty(\Omega; C^k([0,\tau] \times \T^d))} \leq C_k \e^{-k - d} \|R_0\|_{L^\infty(\Omega; C([0,\tau]; L^1(\T^d)))}. 
\end{equation*}
Furthermore, we have the $L_{x}^{1}$-contraction property
\begin{equation*}
    \| R_{\eps} \|_{L^{\infty}(\O; C([0,\tau]; L^{1}(\T) ))} \leq \| R_{0} \|_{L^{\infty}(\O; C([0,\tau]; L^{1}(\T) ))}.
\end{equation*}
\end{lemma}
\begin{proof}
The proof of the first estimate follows immediately from Lemma \ref{l_mollification_1}. The second estimate is proven using Young's convolution inequality: $\P$-a.s., for every $t \in [0,\tau(\o)]$ it follows that
\begin{align*}
    |R_{\eps}(t,x;\o)| &= \left| \int\int R_{0}(s,y;\o) \eta^{(1)}_{\eps}(x-y) \eta^{(2)}_{\eps}(t-s) dy ds \right| \leq \int \| R_{0}(s;\o) \|_{L^{1}_{x}} \| \eta^{(1)}_{\eps} \|_{L^{1}_{x}} \eta^{(2)}_{\eps}(t-s) ds \\
    (\eta^{(i)} \text{ mollifiers})&\leq \sup_{r \in [0,\tau(\o)]} \| R_{0}(r;\o) \|_{L_{x}^{1}} \leq \| R_{0} \|_{L_{\o}^{\infty} C_{\tau} L^{1}_{x}}.
\end{align*}
Integrating over $x$, taking the supremum over $t \in [0,\tau(\o)]$ and the essential supremum over $\o \in \O$ implies the claim.
\end{proof}

\subsection{\texorpdfstring{$\ell$}{l}-mollification of the Brownian motion}
\label{ssec_BM_mollif}

We extend the Brownian motion $B$ to the left by setting $B(-t) := B(0) = 0$ for all $t > 0$.
Let $\chi \colon \R \to \R$ be a time-mollifier with $\supp \chi \subset (0,{1})$ and define for $\ell > 0$, $\chi_{\ell}(s) :=  \frac{1}{\ell} \chi \left( \frac{s}{\ell} \right)$. Consider the mollified Brownian motion:
\begin{align*}
    B_{\ell}(t) := (B *_{t} \chi_{\ell})(t) := \int_{\R} B(t-s) \chi_{\ell}(s) ds = \int_{0}^{\ell} B(t-s) \chi_{\ell}(s) ds.
\end{align*}

\begin{lemma}
\label{l_brownian}
It holds that
\begin{align}
    \label{eq_BM_C0} \|B_{\ell} - B\|_{L^\infty(\Omega; C^0([0,\tau]))} &\leq L \ell^{1/2 - \kappa},
\end{align}
and, for $k \in \{0,1,\dots\}$,
\begin{align}
    \label{eq_BM_C1} \| \dot B_{\ell} \|_{L^\infty(\Omega; C^{k}([0,\tau]))} &\leq C_k L \ell^{(1/2 - \kappa) - (k+1)}, 
\end{align}
where $C_k$ is a constant depending only on the $C^k$ norm of the mollifier $\chi$ and $L$ is the bound on the H\"older norm of almost each trajectory of the Brownian motion up to the stopping time, as fixed at the beginning of Section \ref{sec_main_results}. 
\end{lemma}
\begin{proof}
The proof is standard, and thus it is omitted, but it heavily relies on the fact that we consider trajectories of the Brownian motion only up to the stopping time $\tau$ and we can thus control the $C^{1/2-\kappa}$ norm of almost each $t \mapsto B(t;\omega)$ by means of the constant $L$. 
\end{proof}

In analogy with the definition of $\Psi$, we define also
\begin{equation*}
\Psi_\ell : [0,{1}] \times \T^d \times \Omega \to [0,{1}] \times \T^d, \qquad 
\Psi_\ell(t,x; \omega) := (t, x+ B_\ell(t; \omega)).
\end{equation*}

A lemma analog to Lemma \ref{l_psi} holds for $\Psi_\ell$ as well. The only difference with respect to Lemma \ref{l_psi} is in the last point, where - thanks to the mollification - higher order time derivatives can be estimated. 
\begin{lemma}
\label{l_psi_ell}
$\P$-a.s. the following hold.
\begin{enumerate}
\item For all $t \in [0,{1}]$, $\Psi_{\ell}$ and $\Psi_{\ell}^{-1}$ act on the space variables as translations, in the sense that
\begin{equation*}
x \mapsto x+ B_{\ell}(t,\omega)  \text{ and } \ x \mapsto x - B_{\ell}(t,\omega) 
\end{equation*}
are translations.

\item If $b : [0,{1}] \times \T^d \to \R^d$ is a smooth vector field, then $\divv (b \circ \Psi_\ell) = (\divv b) \circ \Psi_\ell$.

\item If $f \in C([0,{1}]; L^r(\T^d))$ for some $r \in [1, \infty]$,  it holds $\|(f \circ \Psi_\ell)(t)\|_{L^r(\T^d)}  = \|f(t)\|_{L^r(\T^d)}$ for all $t \in [0,{1}]$.

\item For the spatial derivatives it holds, for all $t \in [0,{1}]$,  $\| \nabla \Psi_\ell(t, \cdot; \omega)\|_{C^k(\T^d)} \leq 1$ for all $k \in \N$, $k \geq 0$.

\item For the time-space derivatives it holds
\begin{equation}
\label{eq_Psi_ell}
\|\nabla_{(t,x)}^{k} \Psi_\ell\|_{L^\infty(\Omega; C([0,\tau] \times \T^d))} \leq C(k, \ell),
\end{equation}
\end{enumerate}
where $\nabla^k_{(t,x)} \Psi_\ell$ denotes any derivative with respect to time and/or space variables of order $k$. 
\end{lemma}
\begin{proof}
The proof of the Lemma is trivial. Notice, in particular, that \eqref{eq_Psi_ell} follows immediately from Lemma \ref{l_brownian}, and \eqref{eq_Psi_ell} holds only up the stopping time. 
\end{proof}

\subsection{Definition of the perturbations}

Consider the decomposition of $R_\e$ along the coordinate basis
\begin{equation*}
R_\e(t,x;\omega) = \sum_{j=1}^d R_\e^j(t,x; \omega) e_j.
\end{equation*}
We can finally define
\begin{align}
    \label{eq_def_vt}\vt(t,x; \omega) &:= \sum_{j=1}^{d} R_{\e}^{j}(t,x; \omega) \Theta^{j}_{\lambda, \mu,\sigma, \nu}(t,x + B_{\ell}(t; \omega)), \\
    \nonumber w(t,x) &:= \sum_{j=1}^{d} W^{j}_{\lambda,\mu,\sigma,\nu}(t,x).
\end{align}
The extra shift $B_{\ell}(t)$ is introduced into $\vt$ to have a chance of using the disjoint support property \eqref{eq_Mikado_supp_disjoint} of $\Theta, W$ in the product $\rho(t,x) \cdot u(t,x+B(t))$, but it will, of course, create its own complications manifest in an additional defect term $R^{\mathrm{sto,1}}$ below, cf. Section \ref{ssec_quadratic}.

Furthermore, we will define the quadratic corrector term as
\begin{align*}
    q(t,x; \omega) := \sum_{j=1}^{d} R_{\e}^{j}(t,x; \omega) Q^{j}_{\lambda, \mu,\sigma, \nu}(t,x + B_{\ell}(t; \omega)).
\end{align*}

In short 
\begin{equation*}
\vt = \sum_{j=1}^d R_\e^j \Theta^j(\Psi_\ell), \qquad w = \sum_{j=1}^d W^j, \qquad q = \sum_{j=1}^d R_\e^j Q^j (\Psi_\ell). 
\end{equation*}
\begin{remark}\label{rem_cutoff}
    \begin{enumerate}[label = (\alph*)]
    \item Let us comment briefly on the problems with extending our result to all times. First, our velocity field $u$ is given and converges for \textit{all} times, as it is a purely deterministic object that does not depend on the Brownian motion and its stopping time, so this is not a problem. Once we have solved the equation up to $t = \tau(\o)$, what prevents us from using standard existence results to go beyond that time? The problem is that due to the constraint of having deterministic $u$ we needed to put all $R_{\e}$-dependence into the density field, which causes convergence problems that were then avoided by concentration of the building block $\Theta$ in a higher-order space $L_{x}^{s}$, $s > p$. However, the convergence of $\rho$ is only in $L^{p}$ (or, with a small extension, using a sequence $s_{n} \downarrow p$, in $L^{p+}$). Similarly, as the building blocks $\Theta, W$ need to be in conjugate H\"older spaces, this implies that $W \in L_{x}^{s'}$, or again with a bit more work, $W \in L_{x}^{p'-}$. Note that we cannot conclude that the product $\rho u \in L_{x}^{1}$ from H\"older's inequality, as the integrabilities "just miss" each other. For times $t \in [0,\tau(\o)]$, this was not a problem, as by construction $\rho u \in C_{\tau} L_{x}^{1}$. When we want to extend the solution past $\tau(\o)$, this does not hold anymore. This is the main reason why we only get local-in-time solutions and hence only local nonuniqueness with our approach.
    
    \item A nice side-effect of this approach is that we do not get a sign function anymore in the definition of the perturbations and hence do not need the cutoff functions with respect to the stress term, cf. \cite[p. 17]{MS20}, or \cite[p. 18]{KY22} which simplifies the following discussions. Note, however, that this leads to many new problems, e.g., since we need to put all the $R_{\e}$-dependence into the density terms, we do not have the property that $R_{0}(t) = 0$ for some $t$ implies that $u_{0}(t) = u_{1}(t)$ and that also $R_{1}(t) = 0$. This is contrary to other works such as \cite{MS20, KY22}.

    \item Note also that $w$ itself is smooth in time and space. Non-smoothness is only introduced on the level of the equation where we have to consider $ u(\Psi) \rho$, and $\Psi$ is not smooth (in time). 
    \end{enumerate}
\end{remark}

Each of the scalar terms gets a corrector to ensure they have mean zero:
\begin{align*}
    \vt_{c}(t; \omega) &:= - \fint_{\T^{d}} \vt(t,x; \omega) dx, \\
    q_{c}(t; \omega) &:= - \fint_{\T^{d}} q(t,x; \omega) dx,
\end{align*}
whereas the vector-valued velocity will be corrected to be incompressible by the incompressibility corrector
\begin{align}
    w_{c}(t,x) := \sum_{j=1}^{d}  W^{j, \textrm{corr}} (t,x).
\end{align}
Observe that, by Proposition \ref{prop_mikado_eqn}, 
\begin{equation*}
\divv (w + w^c) = 0. 
\end{equation*}
We then set
\begin{align}\label{eq_def_rho1}
    \rho_{1} &:= \rho_\e + \vt + \vt_{c} + q + q_{c}, \\
    \label{eq_def_u1} u_{1} &:= u_{0} + w + w_{c}.
\end{align}
In particular it holds that
\begin{equation*}
\divv u_1 = 0. 
\end{equation*}
We conclude this section by observing that we will have to fix the following six parameters:
\begin{equation*}
\e, \ell, \lambda, \mu, \sigma, \nu, N. 
\end{equation*}
This will be done in Section \ref{ssec_parameter_choice}.

\subsection{Estimates for the perturbation}

We estimate in this section $\vartheta, q,w$ in the various norms we need.  We want to note that, as the previous stage $R_{\ell}$ in the application of the iteration proposition depends also on $B(t)$ (though in a mollified way), we can only properly control it up to time $[0,\tau_{L}(\o)]$. The fact that we can absorb norms of the old error term $R_{\ell}$ into the constant $C$ only works up to time $\tau_{L}(\o)$, and the constant $C$ along the iteration will therefore depend on $L$, but it will still be a deterministic, finite constant.

\begin{remark}
In this and in the following sections, we will have several times claims of the form
\begin{equation*}
\|f\|_{L^\infty_\omega C_\tau L^r_x} \leq C
\end{equation*}
for various functions $f(t, x; \omega)$, integrability exponents $r$ and constant $C$. We observe explicitly that this amounts to show 
\begin{itemize}
\item not only that the spatial $L^r$ norm of $f$ is uniformly bounded with respect to $\omega$ and $t$ by $C$,
\item but also that the map $[0,\tau(\omega)] \ni t \mapsto f(t, \cdot; \omega) \in L^r(\T^d)$ is continuous (as a function from time into $L^r$), $\P$-a.s. 
\end{itemize}
We will, however, never show explicitly the time continuity, because this follows easily from the fact that all objects we are considering are continuous (in the classical sense) in space and time ($\P$-a.s.). We will thus focus just on estimating $\esssup_\omega \max_t \|f(t, \cdot; \omega)\|_{L^r}$. 
\end{remark}

\begin{lemma}[$\vt$ small in $L^p$ and $L^{1}$]
\label{l_vt_small_lp}
It holds that
    \begin{align}
    \label{eq_vt_Lp}
        \| \vt\|_{L^\infty_\omega C_\tau L^{p}_{x}} & \leq C( \|R_0\|_{L^\infty_\omega C_{\tau} L^{1}_{x}}, \e) \mu^{\frac{d}{s} - \frac{d}{p}}, \\
    \label{eq_vt_L1}
        \| \vt \|_{L^\infty_\omega C_\tau L^{1}_{x}} & \leq C(  \|R_0\|_{L^\infty_\omega C_{\tau} L^{1}_{x}}, \e) \mu^{-\frac{d}{s'}}, \\
    \label{eq_vt_Ls}
        \| \vt \|_{L^\infty_\omega C_\tau L^{s}_{x}} & \leq C(  \|R_0\|_{L^\infty_\omega C_{\tau} L^{1}_{x}}, \e).
    \end{align}
    Moreover
    \begin{equation}
    \label{eq_vt_smooth}
            \| \vt \|_{L^\infty_\omega C^1_{\tau x}} \leq C(\|R_0\|_{L^\infty_\omega C_{\tau} L^{1}_{x}}, \e, \ell, \lambda, \mu, \sigma, \nu) < \infty. 
    \end{equation}
\end{lemma}

\begin{remark}
The reason we want estimates like \eqref{eq_vt_smooth}  is that we want to be sure that $\vartheta$ is bounded in $C^1_{\tau x}$ \textit{uniformly (in the sense of $L^\infty$ bound) in $\omega$}. Hence, we do not care how big this $C^1_{\tau x}$ norm is, but only that the $\esssup$ over all $\omega \in \Omega$ of the $C^1_{\tau x}$ norm is still bounded. We need this information, because the definition of solutions to the continuity-defect equation \eqref{eq_cont_defect} we are using (Definition \ref{def:cont-defect}) requires such bound. 
\end{remark}

\begin{proof}
The fact that, $\P$-a.s, the map $[0,\tau(\omega)] \ni t \mapsto \vartheta(t, \cdot; \omega) \in L^p(\T^d)$ is continuous follows immediately from the continuity of the objects involved in the definition of $\vartheta$. Concerning the estimate \eqref{eq_vt_Lp}, we find 
\begin{align*}
    \| \vt \|_{L^\infty_\omega C_\tau L^{p}_{x}} \leq \sum_{j=1}^{d} \| R_{\e} \|_{L^\infty_\omega C_{\tau x}} \| \Theta^{j}(\Psi_\ell) \|_{L^{\infty}_{\o} C_\tau L^{p}_{x}}
    \leq \sum_{j=1}^{d} \| R_{\e} \|_{L^\infty_\omega C_{\tau x}} \| \Theta^{j} \|_{C_t L^{p}_{x}}
    {\leq} C(  \|R_0\|_{L^\infty_\omega C_{\tau} L^{1}_{x}}, \e)  \mu^{\frac{d}{s} - \frac{d}{p}},
\end{align*}
where, in the last inequality, we used Lemma \ref{l_mollified_error}, \eqref{eq_Theta} and the fact that $\Psi_\ell$ acts on the space variable just as a translation. 
A similar argument holds for the continuity in $L^1$ (resp. $L^s$) and the corresponding estimate \eqref{eq_vt_L1} (resp. \eqref{eq_vt_Ls}). Concerning \eqref{eq_vt_smooth} we have, $\P$-a.s.,
\begin{equation*}
\|\vartheta\|_{L^\infty_\omega C^1_{\tau  x}} \leq \sum_{j=1}^{d} \|R_\e\|_{L^\infty_\omega C^1_{\tau x}} \|\Theta^{j} (\Psi_\ell)\|_{L^\infty_\omega C^1_{\tau x}}.
% \leq \|R_\e\|_{L^\infty_\omega C^k_{tx}} \|\Theta (\Psi_\ell)\|_{C^k_{tx}}.
\end{equation*}
Now, by Lemma \ref{l_mollified_error}, 
\begin{equation*}
\|R_\e\|_{L^\infty_\omega C^1_{\tau x}} \leq C(\|R_0\|_{L^\infty_\omega C_{\tau} L^{1}_{x}}, \e) < \infty,
\end{equation*}
whereas, combining the definition of $\Theta$ with \eqref{eq_Psi_ell},
\begin{equation*}
\begin{aligned}
\|\Theta (\Psi_\ell)\|_{L^\infty_\omega C^1_{\tau x}} \leq \|\Theta\|_{C^0_{tx}} + \|\Theta\|_{C^1_{tx}} \|\nabla \Psi\|_{L^{\infty}_{\o} C^0_{\tau x}} \leq  C( \ell, \lambda, \mu, \sigma, \nu) < \infty,
\end{aligned}
\end{equation*}
so that, after all, \eqref{eq_vt_smooth} holds true. 
\end{proof}

\begin{lemma}[$q$ small in $L^{1}$ and $L^s$] \label{l_q_small_lp}
It holds that
	    \begin{equation}\label{eq_q_L1}
         \| q \|_{L^\infty_\omega C_\tau L^{1}_{x}} \leq C( \|R_0\|_{L^\infty_\omega C_{\tau} L^{1}_{x}}, \e)\frac{1}{\sigma},
    \end{equation}
    and
        \begin{equation}\label{eq_q_Ls}
         \| q \|_{L^\infty_\omega C_\tau L^{s}_{x}} \leq C(  \|R_0\|_{L^\infty_\omega C_{\tau} L^{1}_{x}}, \e) \frac{ \mu^{\frac{d}{s'}} }{\sigma}.
    \end{equation}
        Moreover
    \begin{equation}
    \label{eq_q_smooth}
            \| q \|_{L^\infty_\omega C^1_{\tau x}} \leq C( \|R_0\|_{L^\infty_\omega C_{\tau} L^{1}_{x}}, \e, \ell, \lambda, \mu, \sigma, \nu) < \infty. 
    \end{equation}
\end{lemma}
\begin{proof}
The proof works analogously to the one of Lemma \ref{l_vt_small_lp}, using \eqref{eq_Q} instead of \eqref{eq_Theta}. 
\end{proof}

\begin{lemma}[$\vt_{c}, q_{c}$ small as numbers]
It holds that
    \begin{align}
        \label{eq_vtc} \| \vt_{c}\|_{L^\infty_\omega C_\tau} &\leq C(  \|R_0\|_{L^\infty_\omega C_{\tau} L^{1}_{x}}, \e)  \mu^{-d/s'}, \\
        \label{eq_qc} \| q_{c}\|_{L^\infty_\omega C_\tau} &\leq \frac{C(  \|R_0\|_{L^\infty_\omega C_{\tau} L^{1}_{x}}, \e) }{\sigma}.
    \end{align}
        Moreover
    \begin{equation}
    \label{eq_vc_smooth}
    \begin{aligned}
     \| \vartheta_c \|_{L^\infty_\omega C^1_{\tau}} & \leq C( \|R_0\|_{L^\infty_\omega C_{\tau} L^{1}_{x}}, \e, \ell, \lambda, \mu, \sigma, \nu) < \infty, \\
	\|q_c\| _{L^\infty_\omega C^1_{\tau}} & \leq C( \|R_0\|_{L^\infty_\omega C_{\tau} L^{1}_{x}}, \e, \ell, \lambda, \mu, \sigma, \nu) < \infty.
    \end{aligned}
    \end{equation}
\end{lemma}
\begin{proof}
The proof of \eqref{eq_vtc} and \eqref{eq_qc} follows easily from \eqref{eq_vt_L1}, \eqref{eq_q_L1} and the definitions of $\vt_c, q_c$, whereas the proof of \eqref{eq_vc_smooth} is analogous to the corresponding part of the proof of Lemma \ref{l_vt_small_lp}. 
\end{proof}

Combining the previous estimates, we get the following
\begin{lemma}
\label{l_distance_rho}
It holds that
\begin{equation}
\label{eq_distance_rho}
\|\rho_1 - \rho_0\|_{L^\infty_\omega C_\tau L^p_x} \leq C(\|\rho_0\|_{L^\infty_\omega C^1_{\tau x}}) \e + C(\|\rho_0\|_{L^\infty_\omega C^1_{\tau x}}, \|R_0\|_{L^\infty_\omega C_\tau L^1_x},\e) \left(      \mu^{\frac{d}{s} - \frac{d}{p}} + \frac{\mu^{\frac{d}{s'}}}{\sigma} + \mu^{-\frac{d}{s'}} + \frac{1}{\sigma} \right).
\end{equation}
\end{lemma}
\begin{proof}
We combine the previous estimates with  Lemma \ref{l_mollified_density} to get
\begin{equation*}
\begin{aligned}
\|\rho_1 - \rho_0\|_{L^\infty_\omega C_\tau L^p_x} 
& \leq \|\rho_\e - \rho_0\|_{L^\infty_\omega C_\tau L^p_x} 
 + \|\vartheta\|_{L^\infty_\omega C_\tau L^p_x}  + \|q\|_{L^\infty_\omega C_\tau L^p_x}  + \|\vartheta_c\|_{L^\infty_\omega C_\tau L^p_x}  + \|q_c\|_{L^\infty_\omega C_\tau L^p_x} \\
& \leq \|\rho_\e - \rho_0\|_{L^\infty_\omega C_{\tau x}}  + \|\vartheta\|_{L^\infty_\omega C_\tau L^p_x}  + \|q\|_{L^\infty_\omega C_\tau L^s_x}  + \|\vartheta_c\|_{L^\infty_\omega C_\tau L^p_x}  + \|q_c\|_{L^\infty_\omega C_\tau L^p_x} \\
&\leq C(\|\rho_0\|_{L^\infty_\omega C^1_{\tau x}}) \e + C(\|\rho_0\|_{L^\infty_\omega C^1_{\tau x}}, \|R_0\|_{L^\infty_\omega C_\tau L^1_x},\e) \left(      \mu^{\frac{d}{s} - \frac{d}{p}} + \frac{\mu^{\frac{d}{s'}}}{\sigma} + \mu^{-\frac{d}{s'}} + \frac{1}{\sigma} \right),
\end{aligned}
\end{equation*}
which concludes the proof.
\end{proof}

\begin{lemma}[$w$ small in $L^1$ and in $W^{\theta,\tilde{p}}$] 
\label{l_w_sob_theta}
It holds that
    \begin{equation}\label{eq_wperturb_Sobolev_alpha}
        \| w \|_{C([0,{1}]; L^1(\T^d))} \leq C \mu^{-\frac{d}{s}}, \qquad 
        \| w \|_{C([0,{1}]; L^{s'}(\T^d))} \leq C, \qquad 
        \| w \|_{C([0,{1}]; W^{\theta,\tilde{p}}(\T^d))} \leq C \mu^{\frac{d}{s'} -\frac{d}{ \tilde p}} \nu^{\theta}.
    \end{equation}
where the constant $C$ depends only on the constants $M_{k,r}$ (with $k=0,1$ and $r = 1, \tilde p, s'$) appearing in Proposition \ref{prop_Mikado_est} and on the constant appearing in the statement of Proposition \ref{prop_interpolation}. 
\end{lemma}

\begin{proof}
The first estimate is an immediate consequence of \eqref{eq_W} with $k=0$ and $r =1$, whereas the second follows from \eqref{eq_W} with $k=0$ and $r = s'$. Concerning the $W^{\theta, \tilde p}$ estimate, we observe that, again by \eqref{eq_W}, for $t \in [0,1]$,
\begin{equation*}
\|w(t, \cdot)\|_{L^{\tilde p}} \leq M_{0,\tilde p} \mu^{d/s' - d/\tilde p}, \qquad
\|w(t, \cdot)\|_{W^{1, \tilde p}} \leq M_{1, \tilde p} \mu^{d/s' - d/\tilde p} \nu.
\end{equation*}
This, combined with Proposition \ref{prop_interpolation}, yields the conclusion.
\end{proof}

\begin{lemma}[$w_{c}$ small in $W^{\theta,\tilde{p}}$] 
\label{l_wc_sob_theta}
It holds that
    \begin{equation}\label{eq_wcorr_Sobolev_alpha}
            \| w_{c} \|_{C([0,{1}]; L^{s'}(\T^d))} \leq C \frac{\lambda \mu}{\nu}, \qquad 
     \| w_{c} \|_{C([0,{1}];W^{\theta, \tilde{p}}_{x})} \leq C \frac{\lambda \mu}{\nu} \mu^{\frac{d}{s'} - \frac{d}{\tilde p}} \nu^{\theta},
    \end{equation}
where the constant $C$ depends only on the constants $M_{k,r}$ (with $k=0,1$ and $r = 1, \tilde p, s'$) appearing in Proposition \ref{prop_Mikado_est} and on the constant appearing in the statement of Proposition \ref{prop_interpolation}. 
\end{lemma}
\begin{proof}
The proof works analogously to the proof of Lemma \ref{l_w_sob_theta}, using \eqref{eq_Wcorr} instead of \eqref{eq_W}. 
\end{proof}

\section{Definition and estimates of the defect terms}\label{sec_defect_estimates}

In this section we define and estimate the new defect field $R_1$. We have already defined $\rho_1, u_1$. Let us thus see what equation they satisfy. We get
\begin{equation}\label{eq_error_decomp}
 \begin{split}
  \partial_{t} {\rho}_{1} + \divv ({u}_{1}(\Psi) {\rho}_{1}) 
	& =  \divv \left( \underbrace{(u_0(\Psi) \rho_\e - (u_0(\Psi) \rho)_\e}_{=: - R^{\mathrm{com}}} \right) \\
	& \quad + \underbrace{\partial_t \rho_\e + \divv((u_0(\Psi) \rho)_\e)}_{=-\divv R_\e} + \partial_{t}(q + q_{c}) + \divv(w(\Psi) \vartheta) \\
    &\quad + \partial_{t}(\vt + \vt_{c}) + \divv \left( {\rho}_{\e} w(\Psi) + \vt {u}_{0}(\Psi) \right) \\
    &\quad + \underbrace{\divv \left( q({u}_{0}(\Psi) + w(\Psi)) \right)}_{=: -\divv R^{q}} \\
    &\quad + \underbrace{\divv \left( ({\rho}_{\e} + \vt + q) w_{c}(\Psi) \right)}_{=: -\divv R^{\mathrm{corr}}} \\
    &\quad + (\vt_{c} + q_{c}) \underbrace{\divv \left( {u}_{0} + w + w_{c} \right)}_{=0}(\Psi).
\end{split}
\end{equation}
We now analyse each line of the right-hand side separately.  The second and the third line will be decomposed further in the next sections, whereas lines 1, 4 and 5 can be estimated directly. 

\subsection{Analysis of the first line in \texorpdfstring{\eqref{eq_error_decomp}}{(75)}}
\label{ssec_comm_error}

\begin{lemma}
\label{l_comm}
It holds that
\begin{equation}
\label{eq_comm}
\|R^{\mathrm{com}}\|_{L^\infty_\omega C_{\tau} L^{1}_{x}} \leq \|R^{\mathrm{com}}\|_{L^\infty_\omega C_{\tau x}} \leq C(\|u_0\|_{C^1_{tx}}, \|\rho_0\|_{L^\infty_\omega C^1_{\tau x}}) \e^{1/2 - \kappa}.
\end{equation}
\end{lemma}
\begin{proof}
We have
\begin{equation*}
R^{\rm com} = u_0(\Psi) \rho_\e - (u_0(\Psi) \rho)_\e = \big( u_0(\Psi) (\rho_\e - \rho) \big) + \big( u_0(\Psi) \rho - (u_0(\Psi) \rho)_\e \big). 
\end{equation*}
We now estimate the two terms separately. For the first one we have
\begin{equation*}
\| u_0(\Psi) (\rho_\e - \rho)\|_{L^\infty_\omega C_{\tau x}} \leq \|u_0\|_{C_{tx}} \|\rho_\e - \rho\|_{L^\infty_\omega C_{\tau x}} 
\leq \|u_0\|_{C_{tx}} \|\rho\|_{L^\infty_\omega C^1_{\tau x}} \e,
\end{equation*}
where we used Lemma \ref{l_mollification_2}. For the second term, using again the same lemma together with Lemmas \ref{l_holder1} and \ref{l_holder2} and estimate \eqref{eq_holder_psi},  we have
\begin{equation*}
\begin{aligned}
\| u_0(\Psi) \rho - (u_0(\Psi) \rho)_\e\|_{L^\infty_\omega C_{\tau x}} 
& \leq [u_0(\Psi) \rho]_{1/2 - \kappa} \e^{1/2-\kappa} \\
& \leq \|u_0 \circ \Psi\|_{L^\infty_\omega C^{1/2-\kappa}_{\tau x}} \|\rho\|_{L^\infty_\omega C^{1/2-\kappa}_{\tau x}} \e^{1/2-\kappa} \\
& \leq \Big(\|u_0\|_{C^0_{tx}} + \|u_0\|_{C^1_{tx}}\esssup_{\omega \in \Omega} [\Psi|_{[0, \tau(\omega)] \times \T^d}]_{1/2-\kappa} \Big ) \|\rho\|_{L^\infty_\omega C^{1/2-\kappa}_{\tau x}} \e^{1/2-\kappa} \\
& \leq C(\|u_0\|_{C^1_{tx}}, \|\rho\|_{L^\infty_\omega C^{1/2-\kappa}_{\tau x}}) \e^{1/2 - \kappa}.
\end{aligned}
\end{equation*}
Putting together the two estimates we get the conclusion. 
\end{proof}

\subsection{Quadratic defects} \label{ssec_quadratic}
We analyse now the second line in \eqref{eq_error_decomp}, namely what is usually called the quadratic part of the error. It turns out that one additional term with respect to the deterministic scheme appears due to the fact that we are dealing with a random PDE, but want to produce a purely deterministic vector field. 

We thus compute (note that unless the two factors of the product are evaluated in the same point, we cannot make sure that they have disjoint support!) 
\begin{align*}
    \divv (\vt w(\Psi) - R_\e) &= \divv \left( \sum_{j,k=1}^{d} R_{\e}^{j} \Theta^{j}(\Psi_{\ell}) W^{k}(\Psi) - \sum_{j=1}^{d} R_{\e}^{j} e_{j} \right) \\
    &= \divv  \sum_{j=1}^{d} R_{\e}^{j} \left( \Theta^{j}(\Psi_{\ell}) W^{j}(\Psi_{\ell}) - e_{j} \right) \\
    &\quad + \divv  \sum_{j,k=1}^{d} R_{\e}^{j} \Theta^{j}(\Psi_{\ell}) \left[ W^{k}(\Psi) - W^{k}(\Psi_{\ell}) \right]  \\
    &= \sum_{j=1}^{d} (\nabla R_{\e}^{j}) \cdot \left[ \Theta^{j}(\Psi_{\ell}) W^{j}(\Psi_{\ell}) - e_{j} \right] + \sum_{j=1}^{d} R_{\e}^{j} \divv \left( \Theta^{j}(\Psi_{\ell}) W^{j}(\Psi_{\ell}) \right) \\
    &\quad + \divv  \sum_{j,k=1}^{d} R_{\e}^{j} \Theta^{j}(\Psi_{\ell}) \left[ W^{k}(\Psi) - W^{k}(\Psi_{\ell}) \right] .
\end{align*}
On the other hand, 
\begin{align*}
    \partial_{t} (q + q_{c}) 
    &= \sum_{j=1}^{d} (\partial_{t} R_{\e}^{j}) Q^{j}(\Psi_{\ell}) + \sum_{j=1}^{d}  R_{\e}^{j} \partial_{t} \left(Q^{j} \circ \Psi_{\ell} \right)  + \dot{q}_{c}\\
    &= \sum_{j=1}^{d} (\partial_{t} R_{\e}^{j}) Q^{j}(\Psi_{\ell}) + \sum_{j=1}^{d}  R_{\e}^{j} \partial_{t} Q^j(\Psi_\ell) + R_\e^j \nabla Q_j(\Psi_\ell) \dot B_\ell + \dot{q}_{c}\\
   &= \sum_{j=1}^{d} (\partial_{t} R_{\e}^{j}) Q^{j}(\Psi_{\ell}) + \sum_{j=1}^{d}  R_{\e}^{j} \partial_{t} Q^j(\Psi_\ell) + \divv (R_\e^j Q_j(\Psi_\ell) \dot B_\ell) - \nabla R_\e^j \cdot \dot B_\ell Q_j(\Psi_\ell) + \dot{q}_{c}\\ 
   &= \sum_{j=1}^{d} (\partial_{t} R_{\e}^{j}) Q^{j}(\Psi_{\ell}) - \sum_{j=1}^{d}  R_{\e}^{j} \divv (\Theta^j (\Psi_\ell) W^j(\Psi_\ell)) + \divv (R_\e^j Q_j(\Psi_\ell) \dot B_\ell) - \nabla R_\e^j \cdot \dot B_\ell Q_j(\Psi_\ell) + \dot{q}_{c},
\end{align*}
where in the last line we used \eqref{eq_cont_eq_Q}. Addition of both equations finally yields
\begin{align*}
    &\divv (\vt w(\Psi_\ell) - R_{\e}) + \partial_{t} (q+ q_{c}) \\
    &= \sum_{j=1}^{d} \left( (\nabla R_{\e}^{j}) \cdot \left[ \Theta^{j}(\Psi_{\ell}) W^{j}(\Psi_{\ell}) - e_{j} \right] - \fint_{\T^{d}} (\nabla R_{\e}^{j}) \cdot [\Theta^{j}(\Psi_{\ell}) W^{j}(\Psi_{\ell}) - e_{j}] dx \right) \\
    &\quad + \divv \left(  \sum_{j,k=1}^{d} R_{\e}^{j} \Theta^{j}(\Psi_{\ell}) \left[ W^{k}(\Psi) - W^{k}(\Psi_{\ell}) \right] \right) \\
    &\quad + \sum_{j=1}^{d} \left( (\partial_{t} R_{\e}^{j}) Q^{j}(\Psi_{\ell}) - \fint_{\T^{d}} (\partial_{t} R_{\e}^{j}) Q^{j}(\Psi_{\ell}) dx  \right) \\
    & \quad + \divv \left( \sum_{j=1}^d R_\e^j Q_j(\Psi_\ell) \dot B_\ell \right) \\
    & \quad - \sum_{j=1}^d \left( \nabla R_\e \cdot \dot B_\ell Q_j(\Psi_\ell) - \fint_{\T^d}  \nabla R_\e \cdot \dot B_\ell Q_j(\Psi_\ell) \right) \\
    & \quad + \underbrace{\dot q_c +  \fint_{\T^{d}} (\nabla R_{\e}^{j}) \cdot [\Theta^{j}(\Psi_{\ell}) W^{j}(\Psi_{\ell}) - e_{j}] dx  
    + \fint_{\T^{d}} (\partial_{t} R_{\e}^{j}) Q^{j}(\Psi_{\ell}) dx  
    + \fint_{\T^d}  \nabla R_\e \cdot \dot B_\ell Q_j(\Psi_\ell)}_{=0}. 
\end{align*}
The last line equals zero since all the other lines of the equation have mean zero. We treat and name each line separately as follows:
\begin{align*}
    -\divv R^{\mathrm{quadr}} &:= \sum_{j=1}^{d}\left( (\nabla R_{\e}^{j}) \cdot \left[ \Theta^{j}(\Psi_{\ell}) W^{j}(\Psi_{\ell}) - e_{j} \right] - \fint_{\T^{d}} (\nabla R_{\e}^{j}) \cdot [\Theta^{j}(\Psi_{\ell}) W^{j}(\Psi_{\ell}) - e_{j}] dx \right), \\
    -R^{\mathrm{sto,1}} &:= \sum_{j,k=1}^{d} R_{\e}^{j} \Theta^{j}(\Psi_{\ell}) \left[ W^{k}(\Psi) - W^{k}(\Psi_{\ell}) \right], \\
    -R^{\mathrm{time,1}} &:= \sum_{j=1}^{d} \left\{ \divv^{-1} 
     \left( (\partial_{t} R_{\e}^{j}) Q^{j}(\Psi_{\ell}) - \fint_{\T^{d}} (\partial_{t} R_{\e}^{j}) Q^{j}(\Psi_{\ell}) dx  \right) \right\}, \\
    -R^{\mathrm{sto,2}} &:= \sum_{j=1}^d R_\e^j Q_j(\Psi_\ell) \dot B_\ell, \\
    -R^{\mathrm{sto,3}} &:= -\sum_{j=1}^d \left\{ \divv^{-1} \left(  \nabla R_\e \cdot \dot B_\ell Q_j(\Psi_\ell) - \fint_{\T^d}  \nabla R_\e \cdot \dot B_\ell Q_j(\Psi_\ell)    \right) \right\}. 
\end{align*}
The defects $R^{\mathrm{sto,1}}, R^{\mathrm{sto,2}}, R^{\mathrm{sto,3}}$ are particular to the stochastic case. We will focus on estimating these terms while referring to \cite{MS20} for estimates on the terms $R^{\mathrm{quadr}}, R^{\mathrm{time,1}}$.

The quadratic error is further decomposed as in \cite{MS20}:
\begin{align*}
    (\nabla R_{\e}^{j}) \cdot \left[ \Theta^{j}(\Psi_{\ell}) W^{j}(\Psi_{\ell}) - e_{j} \right] &=  (\nabla R_{\e}^{j}) \cdot \left[ \left( \vphi_{\mu}^{j} \tilde{\vphi}_{\mu}^{j} \right)_{\lambda} \circ \tau_{\sigma t e_{j} - B_{\ell}(t)} \left( \psi^{j}_{\nu} \right)^{2} - 1 \right]e_{j} \\
    &=  (\nabla R_{\e}^{j}) \cdot \left[ \left( \vphi_{\mu}^{j} \tilde{\vphi}_{\mu}^{j} \right)_{\lambda} \circ \tau_{\sigma t e_{j} - B_{\ell}(t)} \left( \left( \psi^{j} \right)^{2} - 1 \right)_{\nu} + \left( \vphi_{\mu}^{j} \tilde{\vphi}_{\mu}^{j} - 1 \right)_{\lambda} \circ \tau_{\sigma t e_{j}- B_\ell(t)} \right]e_{j} \\
    &= (\partial_{x^{j}} R_{\e}^{j})  \left[ \left( \vphi_{\mu}^{j} \tilde{\vphi}_{\mu}^{j} \right)_{\lambda} \circ \tau_{\sigma t e_{j}- B_{\ell}(t)} \left( \left( \psi^{j} \right)^{2} - 1 \right)_{\nu} + \left( \vphi_{\mu}^{j} \tilde{\vphi}_{\mu}^{j} - 1 \right)_{\lambda} \circ \tau_{\sigma t e_{j}- B_{\ell}(t)} \right]e_{j},
\end{align*}
which gives rise to the two terms
\begin{align*}
    -R^{\mathrm{quadr,1}} &:= \sum_{j=1}^{d} \mathcal{R}_{1} \left( (\partial_{x^{j}} R_{\e}^{j})  \left( \vphi_{\mu}^{j} \tilde{\vphi}_{\mu}^{j} \right)_{\lambda} \circ \tau_{\sigma t e_{j}- B_{\ell}(t)}, \left( \left( \psi^{j} \right)^{2} - 1 \right)_{\nu} \right), \\
    -R^{\mathrm{quadr,2}} &:=  \sum_{j=1}^{d} \mathcal{R}_{1} \left( (\partial_{x^{j}} R_{\e}^{j}) , \left( \vphi_{\mu}^{j} \tilde{\vphi}_{\mu}^{j} - 1 \right)_{\lambda} \circ \tau_{\sigma t e_{j} - B_{\ell}(t)} \right),
\end{align*}
and to the decomposition
\begin{align*}
    R^{\mathrm{quadr}} = R^{\mathrm{quadr,1}} + R^{\mathrm{quadr,2}}.
\end{align*}
We remark that both antidivergence operators are well-defined since by Eq. \eqref{eq_intpsi} and Eq. \eqref{eq_intphi}
\begin{align*}
    \fint_{\T^{d}} \left( \left( \psi^{j} \right)^{2} - 1 \right)_{\nu} dx = 0, \quad \fint_{\T^{d}} \left( \vphi_{\mu}^{j} \tilde{\vphi}_{\mu}^{j} - 1 \right)_{\lambda} \circ \tau_{\sigma t e_{j} + B_\ell(t)} dx = 0.
\end{align*}

\begin{lemma}[Bound on $R^{\mathrm{quadr}}$]\label{lem_Rquadr}
It holds that
    \begin{align*}
        \left\| R^{\mathrm{quadr}} \right\|_{L^\infty_\omega C_\tau L^{1}_{x}} \leq C (\e, \|R_0\|_{L^\infty_\omega C_\tau L^1_x}   ) \left( \frac{\lambda \mu}{\nu} + \frac{1}{\lambda} \right).
    \end{align*}
\end{lemma}
\begin{proof}
    The proof is very similar to that of \cite[Lemma 5.3]{MS20} and is hence omitted. Cf. \cite[Lemma 5.1]{MS18} for the precise dependence of the constant on $R_{0}$.
\end{proof}

\begin{lemma}[Bound on $R^{\mathrm{time,1}}$]
\label{lem_Rtime1}
It holds that
\begin{equation*}
            \left\| R^{\mathrm{time,1}} \right\|_{L^\infty_\omega C_\tau L^{1}_x} \leq \frac{C( \e, \|R_0\|_{L^\infty_\omega C_\tau L^1_x})}{\sigma}. 
\end{equation*}
\end{lemma}
\begin{proof}
    We use Lemma \ref{lem_std_antidiv} in the case $r=1$: $\P$-almost surely, for every $t \in [0, \tau]$, we use H\"older's inequality to find 
    \begin{align*}
        \left\| R^{\mathrm{time,1}}(t; \omega) \right\|_{L^{1}_x} &\leq 2\sum_{j=1}^{d}  \| (\partial_{t} R_{\e}^{j}) (t; \omega) Q^{j}(\Psi_{\ell})\|_{L^{1}_{x}} \leq 2 \sum_{j=1}^{d}  \|  R_{\e}^{j}  \|_{L^\infty_\omega C^1_{\tau x} } \| Q^{j} \|_{L^{1}_{x}} \\
        \text{(by Lemma \ref{l_mollified_error})}
        &\leq C(\e, \|R_0\|_{L^\infty_\omega C_{\tau} L^1_x}) \| Q^{j} \|_{L^{1}_{x}}  \overset{\eqref{eq_Q}}{\leq} \frac{C(\e, \|R_0\|_{L^\infty_\omega C_{\tau} L^1_x})}{\sigma}. 
         \end{align*}
        Taking the maximum for $t \in [0,\tau(\omega)]$ and the essential supremum with respect to $\omega$, we get the conclusion.
\end{proof}
\begin{lemma}[Bound on $R^{\mathrm{sto,1}}$]\label{lem_Rsto1} 
It holds that
    \begin{equation}
            \left\| R^{\mathrm{sto,1}} \right\|_{L^\infty_\omega C_\tau L^{1}_{x}} \leq C (\e, \|R_0\|_{L^\infty_\omega C_{\tau} L^1_x}) \nu  \ell^{\frac{1}{2} - \kappa}.
    \end{equation}
\end{lemma}
\begin{proof}
$\P$-a.s, for all $t \in [0, \tau]$, using the H\"older inequality, we find 
    \begin{align*}
        \left\| R^{\mathrm{sto,1}}(t; \omega) \right\|_{L^{1}_{x}} &\leq \sum_{j,k} \| R^{j}_{\e}(t; \omega) \|_{L^{\infty}_x} \| \Theta^{j}(\Psi_{\ell}) \left[ W^{k}(\Psi) - W^{k}(\Psi_{\ell}) \right] \|_{L^{1}_x} \\
        &\leq \sum_{j,k} \| R^{j}_{\e} \|_{L^\infty_\omega C_{\tau x}} \| \Theta^{j} \|_{L^{s}_x} \| \nabla W^{k} \|_{L^{s'}_x} |B(t, \omega) - B_{\ell}(t, \omega)| \\
\text{(by Lemma \ref{l_mollified_error}, Proposition \ref{prop_Mikado_est} and \eqref{eq_BM_C0})}        
        &\leq C(\e, \| R_0 \|_{L^\infty_\omega C_{\tau} L^1_x}) \nu \ell^{\frac{1}{2} - \kappa}.
    \end{align*}
        Taking the maximum for $t \in [0,\tau(\omega)]$ and the essential supremum with respect to $\omega$, we get the conclusion.
\end{proof}

\begin{lemma}[Bound on $R^{\mathrm{sto,2}}$]\label{lem_sto2} 
It holds that
    \begin{equation}
        \| R^{\mathrm{sto,2}} \|_{L^\infty_\omega C_\tau L^{1}_{x}} \leq C (\e, \|R_0\|_{L^\infty_\omega C_{\tau} L^1_x})   \frac{\ell^{-1/2 - \kappa}}{\sigma}. 
    \end{equation}
\end{lemma}
\begin{proof}
$\P$-a.s., for all $t \in [0,\tau]$, we have
\begin{equation*}
\begin{aligned}
\|R^{\mathrm{sto,2}}(t; \omega)\|_{L^1_x}
& \leq \|R_\e(t; \omega)\|_{L^\infty_x} \|Q_j (\Psi_\ell (t; \omega)) \|_{L^1_x} |\dot B_\ell(t; \omega)| \\
\text{(by Lemma \ref{l_mollified_error}, \eqref{eq_Q} and \eqref{eq_BM_C1})}
& \leq C(\e, \|R_0\|_{L^\infty_\omega C_{\tau} L^1_x}) \frac{ \ell^{-1/2-\kappa}}{\sigma}.
\end{aligned}
\end{equation*}
        Taking the maximum for $t \in [0,\tau(\omega)]$ and the essential supremum with respect to $\omega$, we get the conclusion.
\end{proof}

\begin{lemma}[Bound on $R^{\mathrm{sto,3}}$] 
\label{lem_sto3}
It holds that
    \begin{equation}
        \| R^{\mathrm{sto,3}} \|_{L^\infty_\omega C_\tau L^{1}_{x}} \leq C(\e,  \|R_0\|_{L^\infty_\omega C_{\tau} L^1_x}) \frac{ \ell^{-1/2-\kappa}}{\sigma}. 
    \end{equation}
\end{lemma}
\begin{proof}
$\P$-a.s., for all $t \in [0,\tau]$ we have, by Lemma \ref{lem_std_antidiv} with $r=1$,
\begin{equation*}
\begin{aligned}
\|     R^{\mathrm{sto,3}} (t; \omega) \|_{L^1_x}
& \leq 2 \sum_{j=1}^d \|    \nabla R_\e(t; \omega)\|_{L^\infty_x} |B_\ell(t; \omega)| \| Q_j(\Psi_\ell(t; \omega))\|_{L^1_x} \\
\text{(by Lemma \ref{l_mollified_error}, \eqref{eq_Q} and \eqref{eq_BM_C1})}
& \leq C(\e,  \|R_0\|_{L^\infty_\omega C_{\tau} L^1_x}) \frac{ \ell^{-1/2-\kappa}}{\sigma}.
\end{aligned}
\end{equation*}
Taking the maximum for $t \in [0,\tau(\omega)]$ and the essential supremum with respect to $\omega$, we get the conclusion.
\end{proof}

\newpage
\subsection{The linear error terms}
\label{ssec_linear_error}
We now consider the third line of the error decomposition \eqref{eq_error_decomp}. Recall that, by Proposition \ref{prop_mikado_eqn}, $\partial_t \Theta^j = -\divv A_N^j$  (where $N$ is a large natural number to be fixed in Section \ref{ssec_parameter_choice}), and thus also $(\partial_t \Theta^j) (\Psi_\ell) = -(\divv A_N) (\Psi_\ell) = - \divv (A_N(\Psi_\ell))$, since $\Psi_\ell$ acts on space variable just as a translation. We thus compute, using the definition of $\vartheta$:
\begin{equation}
\label{eq_linear_error}
\begin{aligned}
    &\partial_{t}(\vt + \vt_{c}) + \divv \left( {\rho}_{\e} w(\Psi) + \vt {u}_{0}(\Psi) \right) \\
    &= \sum_{j=1}^{d}  \partial_{t} R_{\e}^{j}  \Theta^{j}(\Psi_\ell) + R_{\e}^{j} (\partial_{t} \Theta^{j})(\Psi_\ell) + R_\e^j \nabla \Theta^j(\Psi_\ell) \cdot \dot B_\ell + \dot{\vt}_{c} + \divv \left( {\rho}_{\e} w(\Psi) + \vt {u}_{0}(\Psi)  \right) \\
    &= \sum_{j=1}^{d}  \partial_{t} R_{\e}^{j}  \Theta^{j}(\Psi_\ell) - R_{\e}^{j} \divv A_N(\Psi_\ell) + R_\e^j \nabla \Theta^j(\Psi_\ell) \cdot \dot B_\ell + \dot{\vt}_{c} + \divv \left( {\rho}_{\e} w(\Psi) + \vt {u}_{0}(\Psi)  \right) \\
    &= \sum_{j=1}^{d}  \partial_{t} R_{\e}^{j}  \Theta^{j}(\Psi_\ell) - \divv (R_{\e}^{j} A_N(\Psi_\ell)) + \nabla R_{\e}^{j} \cdot A_N(\Psi_\ell) + \divv( R_\e^j \Theta^j(\Psi_\ell) \cdot \dot B_\ell ) - \nabla R_\e^j \cdot \dot B_\ell \Theta^j(\Psi_\ell) + \dot{\vt}_{c} \\
    & \quad + \divv \left( {\rho}_{\e} w(\Psi) + \vt {u}_{0}(\Psi)  \right) \\
    &= \sum_{j} \left[ \left( \partial_{t} R_{\e}^{j} \right) \Theta^{j}(\Psi_\ell) - \fint_{\T^{d}} \left( \partial_{t} R_{\e}^{j} \right) \Theta^{j} (\Psi_\ell) dx \right]
    + \divv \left( {\rho}_{\e} w(\Psi) + \vt {u}_{0}(\Psi) \right) \\
	& \quad - \sum_j \divv( R_\e^j A_N(\Psi_\ell)) \\
    &\quad + \sum_{j} \left[ \nabla R_{\e}^{j} \cdot A_N(\Psi_\ell)  - \fint_{\T^{d}} \nabla R_{\e}^{j} \cdot A_N(\Psi_\ell) dx \right] \\
    & \quad + \sum_j \divv \left( R_\e^j \Theta^j(\Psi_\ell) \cdot \dot B_\ell \right) \\
    & \quad - \sum_j \left[ \nabla R_\e^j \cdot \dot B_\ell \Theta^j(\Psi_\ell) - \fint_{T^d}   \nabla R_\e^j \cdot \dot B_\ell \Theta^j(\Psi_\ell) dx  \right]   \\
    &\quad +\underbrace{\sum_j \left( \fint_{\T^{d}} \left( \partial_{t} R_{\e}^{j} \right) \Theta^{j} (\Psi_\ell) dx + \fint_{\T^{d}} \nabla R_{\e}^{j} \cdot A_N (\Psi_\ell)  dx 
    - \fint_{T^d}   \nabla R_\e^j \cdot \dot B_\ell \Theta^j(\Psi_\ell) dx \right)
    + \dot{\vt}_{c}}_{= 0, ~ \text{as l.h.s. and all other terms have mean zero}}. \\
\end{aligned}
\end{equation}
We can now treat and name each term on the right-hand side of \eqref{eq_linear_error} separately as follows:
\begin{equation*}
\begin{aligned}
-R^{\mathrm{lin}} 
& := \sum_{j=1}^{d} 
\divv^{-1} \left( \left( \partial_{t} R_{\e}^{j} \right) \Theta^{j}(\Psi_\ell) - \fint_{\T^{d}} \left( \partial_{t} R_{\e}^{j} \right) \Theta^{j} (\Psi_\ell) dx \right)
+{\rho}_{\e} w(\Psi) + \vt {u}_{0}(\Psi), \\
-R^{\mathrm{time},2} & := -\sum_{j=1}^{d}  R_\e^j A_N(\Psi_\ell), \\
-R^{\mathrm{time},3} & := \sum_{j=1}^{d}  \divv^{-1} \left(    \nabla R_{\e}^{j} \cdot A_N(\Psi_\ell)  - \fint_{\T^{d}} \nabla R_{\e}^{j} \cdot A_N(\Psi_\ell) dx \right), \\ 
-R^{\mathrm{sto},4} 
& := \sum_{j=1}^{d}  R_\e^j \Theta^j(\Psi_\ell) \cdot \dot B_\ell, \\
-R^{\mathrm{sto},5}
& := \sum_{j=1}^{d}  \divv^{-1} \left(   \nabla R_\e^j \cdot \dot B_\ell \Theta^j(\Psi_\ell) - \fint_{T^d}   \nabla R_\e^j \cdot \dot B_\ell \Theta^j(\Psi_\ell) dx  \right).
\end{aligned}
\end{equation*}
The defects $R^{\mathrm{sto},4}, R^{\mathrm{sto},5}$ are again particular to the stochastic case.

\begin{lemma}[Bound on $R^{\mathrm{lin}}$]\label{lem_Rlin} 
It holds that
    \begin{equation}
         \| R^{\mathrm{lin}} \|_{L^\infty_\omega C_\tau L^{1}_{x}} \leq 
C(\e, \|R_0\|_{L^\infty_\omega C_{\tau} L^1_x},  \|\rho_0\|_{L^\infty_\omega C_{\tau x}}, \|u_0\|_{C_{tx}}) \left( \mu^{-d/s} + \mu^{-d/s'} \right).
    \end{equation}
\end{lemma}
\begin{proof} 
$\P$-a.s. for every $t \in [0,\tau]$ we have, by Lemma \ref{lem_std_antidiv} with $r=1$,
\begin{equation*}
\begin{aligned}
\left\|\divv^{-1} \left( \left( \partial_{t} R_{\e}^{j} \right) \Theta^{j}(\Psi_\ell) - \fint_{\T^{d}} \left( \partial_{t} R_{\e}^{j} \right) \Theta^{j} (\Psi_\ell) dx \right) \right\|_{L^1_x}
& \leq 2 \sum_j \|\partial_t R_\e^j\|_{L^\infty_x} \|\Theta^j(\Psi_\ell)\|_{L^1_x} \\
\text{(by Lemma \ref{l_mollified_error} and \eqref{eq_Theta})}
& \leq C(\e, \|R_0\|_{L^\infty_\omega C_{\tau} L^1_x}) \mu^{-d/s'}.
\end{aligned}
\end{equation*}
We also have, $\P$-a.s. for every $t \in [0,\tau]$,
\begin{equation*}
\begin{aligned}
\|\rho_\e w(\Psi) + \vartheta u_0(\Psi)\|_{L^1_x}
& \leq \|\rho_\e\|_{L^\infty_x} \|w(\Psi)\|_{L^1_x} + \|\vartheta\|_{L^1_x} \|u_0(\Psi)\|_{L^\infty_x} \\
\text{(by Lemmas \ref{l_mollified_density}, \ref{l_vt_small_lp} and \ref{l_w_sob_theta})}
& \leq C(\|\rho_0\|_{L^\infty_\omega C_{\tau x}}, \|u_0\|_{C_{tx}}) \left( \mu^{-d/s'} + \mu^{-d/s} \right). 
\end{aligned}
\end{equation*}
Taking the maximum for $t \in [0,\tau(\omega)]$ and the essential supremum with respect to $\omega$, we get the desired estimate.
\end{proof}

\begin{lemma}[Bound on $R^{\mathrm{time,2}}$]\label{lem_Rtime2} 
It holds that
    \begin{equation}
        \| R^{\mathrm{time,2}} \|_{L^\infty_\omega C_\tau L^{1}_{x}} \leq C(\e, \|R_0\|_{L^\infty_\omega C_\tau L^1_x})   \frac{\lambda \mu \sigma}{\nu} \mu^{-\frac{d}{s'}}  \left( 1 + \frac{(\lambda \mu)^N}{\nu^{N-1}} \right).
    \end{equation}
\end{lemma}
\begin{proof}
    The proof follows immediately from the definition of $R^{\mathrm{time},2}$, Lemma \ref{l_mollified_error}, \eqref{eq_A} and the fact that $\Psi_\ell$ acts only as a spatial translation. 
\end{proof}

\begin{lemma}[Bound on $R^{\mathrm{time,3}}$]\label{lem_Rtime3} 
It holds that
    \begin{equation}
        \| R^{\mathrm{time,3}} \|_{L^\infty_\omega C_\tau L^{1}_{x}} \leq C(\e, \|R_0\|_{L^\infty_\omega C_\tau L^1_x})   \frac{\lambda \mu \sigma}{\nu} \mu^{-\frac{d}{s'}}  \left( 1 + \frac{(\lambda \mu)^N}{\nu^{N-1}} \right).
    \end{equation}
\end{lemma}
\begin{proof}
One first estimates $\|R^{\mathrm{time}, 3}(t, \cdot; \omega)\|_{L^1_x}$ in terms of $\sum_j \| \nabla R_\e(t, \cdot; \omega)\|_{L^\infty_x} \|A_N(t, \cdot)\|_{L^1_x}$ using Lemma \ref{lem_std_antidiv} and then one argues as in the proof of the previous lemma. 
\end{proof}

\begin{lemma}[Bound on $R^{\mathrm{sto,4}}$]\label{lem_Rsto4}
It holds that
    \begin{equation}
        \| R^{\mathrm{sto,4}}\|_{ L^\infty_\omega C_\tau L^{1}_{x}} \leq C(\|R_0\|_{L^\infty_\omega C_{\tau} L^1_x}) \mu^{-d/s'} \ell^{-1/2 - \kappa}. 
    \end{equation}
\end{lemma}
\begin{proof}
$\P$-a.s. for all $t \in [0,\tau]$ we have
\begin{equation*}
\begin{aligned}
\left\|\sum_j R_\e^j \Theta^j(\Psi_\ell) \dot B_\ell \right\|_{L^1_x}
& \leq \sum_j \|R_\e^j\|_{L^\infty_x} \|\Theta^j(\Psi_\ell)\|_{L^1_x} |\dot B_\ell| \\
& \leq C(\e, \|R_0\|_{L^\infty_\omega C_{\tau} L^1_x}) \mu^{-d/s'} \ell^{-1/2 - \kappa}. 
\end{aligned}
\end{equation*}
Taking the maximum for $t \in [0,\tau(\omega)]$ and the essential supremum with respect to $\omega$, we get the desired estimate.
\end{proof}

\begin{lemma}[Bound on $R^{\mathrm{sto,5}}$]\label{lem_Rsto5}
It holds that
    \begin{equation}
        \| R^{\mathrm{sto,5}} \|_{L^\infty_\omega C_\tau L^{1}_{x}} \leq C(\e, \|R_0\|_{L^\infty_\omega C_{\tau} L^1_x})  \mu^{-d/s'} \ell^{-1/2 - \kappa}.
    \end{equation}
\end{lemma}
\begin{proof}
We have, $\P$-a.s., for all $t \in [0,\tau]$, by Lemma \ref{lem_std_antidiv},
\begin{equation*}
\begin{aligned}
\left\| \sum_j \divv^{-1} \left(   \nabla R_\e^j \cdot \dot B_\ell \Theta^j(\Psi_\ell) - \fint_{T^d}   \nabla R_\e^j \cdot \dot B_\ell \Theta^j(\Psi_\ell) dx  \right) \right \|_{L^1_x} 
& \leq C(\e, \|R_0\|_{L^\infty_\omega C_{\tau} L^1_x}) |\dot B_\ell| \|\Theta^j(\Psi_\ell)\|_{L^1_x} \\
& \leq C(\e, \|R_0\|_{L^\infty_\omega C_{\tau} L^1_x}) \ell^{-1/2 - \kappa} \mu^{-d/s'}.
\end{aligned}
\end{equation*}
Taking the maximum for $t \in [0,\tau(\omega)]$ and the essential supremum with respect to $\omega$, we conclude the proof.
\end{proof}

\subsection{Analysis of the fourth line}
\label{ssec_q_error}
Here we will estimate the fourth line of \eqref{eq_error_decomp}.
\begin{lemma}[${R}^{q}$ is small]\label{lem_Rq}
It holds that
    \begin{equation}
     \| {R}^{q} \|_{L^\infty_\omega C_\tau L^{1}_{x}} \leq C(\|R_0\|_{L^\infty_\omega C_{\tau} L^1_x}, \|u_0\|_{C_{tx}}) \frac{\mu^{d/s'}}{\sigma}.
    \end{equation}
\end{lemma}
\begin{proof}
    This proof works very similarly to \cite[Lemma 5.6]{MS20} and thus it is omitted.
\end{proof}

\subsection{The corrector error}
\label{ssec_corr_error}
The final error we must control is the corrector error of the fifth line of \eqref{eq_error_decomp}.
\begin{lemma}[Bound on $R^{\mathrm{corr}}$]\label{lem_Rcorr} 
It holds that
    \begin{equation}
         \| R^{\mathrm{corr}}\|_{ L^\infty_\omega C_\tau L^{1}_{x}} \leq 
         C( \|\rho_0\|_{L^\infty_\omega C_{\tau x}}, \|R_0\|_{L^\infty_\omega C_\tau L^1_x},\e) \left[ 1 + \frac{\mu^{d/s'}}{\sigma} \right] \frac{\lambda \mu}{\nu}.
    \end{equation}
\end{lemma}
\begin{proof} $\P$-a.s., for all $t \in [0,\tau]$, using \eqref{eq_vt_Ls}, \eqref{eq_q_Ls} and \eqref{eq_wcorr_Sobolev_alpha}, 
    \begin{align*}
        \| R^{\mathrm{corr}} \|_{L^{1}_{x}} &\leq \| ({\rho}_{\e} + \vt + q) w_{c}(\Psi) \|_{L^{1}_{x}} \\
        &\leq \left( \| {\rho}_{\e} \|_{L^{s}_{x}} + \| \vt \|_{L^{s}_{x}}+ \| q \|_{L^{s}_{x}} \right) \| w_{c}(\Psi) \|_{L^{s'}_{x}} \\
        &\leq \left( \| {\rho}_{\e} \|_{C_{tx} } + \| \vt \|_{L^{s}_{x}}+ \| q \|_{L^{s}_{x}} \right) \| w_{c} \|_{L^{s'}_{x}} \\
		& \leq C( \|\rho_0\|_{L^\infty_\omega C_{\tau x}}, \|R_0\|_{L^\infty_\omega C_\tau L^1_x},\e) \left[ 1 + \frac{\mu^{d/s'}}{\sigma} \right] \frac{\lambda \mu}{\nu}. 
    \end{align*}
Taking the maximum for $t \in [0,\tau(\omega)]$ and the essential supremum with respect to $\omega$, we conclude the proof.
\end{proof}

\section{Estimates on the momentum difference}
\label{sec_momentum}

\begin{lemma}
\label{l_momentum}
It holds that
\begin{equation}
\label{eq_momentum}
\begin{aligned}
&\|\rho_1 u_1(\Psi)  - \rho_0 u_0(\Psi)\|_{L^\infty_\omega C_\tau L^1_x} \\
&  \leq M \|R_0\|_{L^\infty_\omega C_\tau L^1_x} + C(\|\rho_0\|_{L^\infty_\omega C^1_{\tau x}}, \|u_0\|_{C_{tx}}) \e \\
& \quad + C ( \e, \|\rho_0\|_{L^\infty_\omega C_{\tau x}}, \|u_0\|_{C_{tx}}, \|R_0\|_{L^\infty_\omega C_{\tau} L^1_x})
\left[\frac{1}{\lambda} + \nu  \ell^{\frac{1}{2} - \kappa} +  \mu^{-\frac{d}{s}} + \mu^{-\frac{d}{s'}} + \frac{1}{\sigma} + \frac{\mu^{\frac{d}{s'} }}{\sigma} 
+ \left( 1 + \frac{\mu^{d/s'}}{\sigma} \right)\frac{\lambda \mu}{\nu}   \right],
\end{aligned}
\end{equation}
where $M$ is a constant depending only on the constants $M_{0,r}$ appearing in Proposition \ref{prop_Mikado_est}, with $r=s,s'$.
\end{lemma}
\begin{proof}
$\P$-a.s., for all $t \in [0, \tau]$, we get
\begin{equation}
\label{eq_rho1u1}
\begin{aligned}
\|\rho_1 u_1(\Psi) - \rho_0 u_0(\Psi)\|_{L^1_x}
& = \|(\rho_\e + \vartheta + q + \vartheta_c + q_c) (u_0(\Psi) + w(\Psi) + w^c(\Psi)) - \rho_0 u_0(\Psi)\|_{L^1_x} \\
& \leq \|\vartheta w(\Psi)\|_{L^1_x} + \|(\rho_\e - \rho_0) u_0(\Psi)\|_{L^1_x}  +
\|(\vartheta + q + \vartheta_c + q_c) u_0(\Psi)\|_{L^1_x} \\
& \qquad + \|(\rho_\e + q + \vartheta_c + q_c) w(\Psi)\|_{L^1_x} + \|(\rho_\e + \vartheta + q + \vartheta_c + q_c) w_c(\Psi)\|_{L^1_x}.
\end{aligned}
\end{equation}
We now estimate the various terms separately. We have
\begin{equation*}
\begin{aligned}
\vartheta w(\Psi) & = \sum_{j,k} R_\e^j \Theta^j(\Psi_\ell) W^k (\Psi) \\
& = \sum_{j} R_\e \Theta^j (\Psi_\ell) W^j(\Psi_\ell) + \sum_{j,k} R_\e \Theta^j(\Psi_\ell) (W^k (\Psi) - W^k(\Psi_\ell)) \\
& = \sum_{j} R_\e \Theta^j (\Psi_\ell) W^j(\Psi_\ell) + R^{\mathrm{sto}, 1}. 
\end{aligned}
\end{equation*}
Now for the first term on the right-hand side after the last equal sign, we use the improved H\"older inequality of Lemma \ref{lem_Holder}:
\begin{equation*}
\begin{aligned}
\sum_{j} \|R_\e^j \Theta^j(\Psi_\ell) W^j(\Psi_\ell)\|_{L^1_x} 
& \leq \left( \|R_\e\|_{L^1_x} + \frac{ \|R_\e \|_{C_{x}^{1}}}{\lambda}  \right) \|\Theta^j(\Psi_{\ell}) W^j(\Psi_{\ell})\|_{L^1_x}  \\
& \leq \left( \|R_\e\|_{L^\infty_\omega C_\tau L^1_x} + \frac{C(\e, \|R_0\|_{L^\infty_\omega C_\tau L^1_x})}{\lambda}  \right) \|\Theta^j \|_{L^s_x} \|W^j\|_{L^{s'}_x} \\
\text{(by Lemma \ref{l_mollified_error} and Proposition \ref{prop_Mikado_est})}
& \leq M \|R_0\|_{L^\infty_\omega C_\tau L^1_x} + \frac{C(\e, \|R_0\|_{L^\infty_\omega C_\tau L^1_x})}{\lambda},
\end{aligned}
\end{equation*}
where $M$ is a constant depending only on the constants $M_{0,r}$ appearing in Proposition \ref{prop_Mikado_est}, with $r=s,s'$ respectively. 
For $R^{\mathrm{sto},1}$ we use the estimate in Lemma \ref{lem_Rsto1}.  For the second term in \eqref{eq_rho1u1} we have, by Lemma \ref{l_mollified_density},
\begin{equation*}
 \|(\rho_\e - \rho_0) u_0(\Psi)\|_{C_\tau L^1_x} \leq \|(\rho_\e - \rho_0)\|_{C_{\tau x}} \|u_0\|_{C_{tx}} \leq C(\|\rho_0\|_{L^\infty_\omega C^1_{\tau x}}, \|u_0\|_{C_{tx}}) \e. 
\end{equation*}
For the third term we have, using the $L^1_x$-estimates for the perturbations proven in Section \ref{sec_perturbations},
\begin{equation*}
\|(\vartheta + q + \vartheta_c + q_c) u_0(\Psi)\|_{L^1_x} \leq \|u_0\|_{L^\infty_x} \| \vartheta + q + \vartheta_c + q_c\|_{L^1_x} \leq 
C(\e, \|u_0\|_{C_{tx}}, \|R_0\|_{L^\infty_\omega C_\tau L^1_x}) \left( \mu^{- \frac{d}{s'}} + \frac{1}{\sigma} \right).
\end{equation*}
Similarly, for the fourth term in \eqref{eq_rho1u1}, 
\begin{equation*}
\begin{aligned}
\| (\rho_\e + q + \vartheta_c + q_c) w(\Psi)\|_{L^1_x}
& \leq \| \rho_\e \|_{L^{\infty}_{x}} \| w \|_{L^{1}_{x}} + \| q + \vartheta_c + q_c\|_{L^s_x} \|w\|_{L^{s'}_x} \\
&
\leq C ( \e, \|\rho_0\|_{L^\infty_\omega C_{\tau x}}, \|R_0\|_{L^\infty_\omega C_{\tau} L^1_x}) \left( \mu^{-\frac{d}{s}} +  \frac{\mu^{  \frac{d}{s'} }}{\sigma} + \mu^{- \frac{d}{s'} } + \frac{1}{\sigma} \right).
\end{aligned}
\end{equation*}
Finally, for the fifth term in \eqref{eq_rho1u1}, the same estimates as for $R^{\mathrm{corr}}$ hold, yielding
\begin{equation*}
 \|(\rho_\e + \vartheta + q + \vartheta_c + q_c) w_c(\Psi)\|_{L^1_x} \leq C( \|\rho_0\|_{L^\infty_\omega C_{\tau x}}, \|R_0\|_{L^\infty_\omega C_\tau L^1_x}) \left[ 1 + \frac{\mu^{d/s'}}{\sigma} \right] \frac{\lambda \mu}{\nu}. 
\end{equation*}
Taking the maximum for $t \in [0,\tau(\omega)]$ and the essential supremum with respect to $\omega$, we obtain the desired result
\end{proof}

\section{Proof of Proposition \ref{prop_iteration_stage}}\texorpdfstring{\label{sec_proof_of_prop}}{3.5}

In this section, we conclude the proof of Proposition \ref{prop_iteration_stage}. Let $\delta>0$ be given in the statement of the Proposition. 

\subsection{Definition of the new solution}
The universal constant $M$ is the constant appearing in the statement of Lemma \ref{l_momentum}. We have already defined $(\rho_{1},u_{1})$ in Section \ref{sec_perturbations}, Eq. \eqref{eq_def_rho1}, \eqref{eq_def_u1}. We define $R_{1}$ as the sum 
of all error terms defined in Section \ref{sec_defect_estimates}:
\begin{equation}
\label{eq_new_error}
\begin{aligned}
R_1 & := \underbrace{R^{\rm com}}_{\text{Section \ref{ssec_comm_error}}} \\
& \quad + \underbrace{R^{\rm quadr} + R^{\mathrm{sto},1} + R^{\mathrm{time},1} + R^{\mathrm{sto},2} + R^{\mathrm{sto},3}}_{\text{Section \ref{ssec_quadratic}}} \\
& \quad + \underbrace{R^{\rm lin} + R^{\mathrm{time},2} + R^{\mathrm{time},3} + R^{\mathrm{sto},4} + R^{\mathrm{sto},5}}_{\text{Section \ref{ssec_linear_error}}} \\
& \quad + \underbrace{R^q}_{\text{Section \ref{ssec_q_error}} } + \underbrace{R^{\rm corr}}_{\text{Section \ref{ssec_corr_error}}}.
\end{aligned}
\end{equation}
By construction $(\rho_1, u_1, R_1)$ solves the continuity-defect equation \eqref{eq_cont_defect} in the sense of distributions. Moreover, by Lemma \ref{l_mollified_density} and estimates \eqref{eq_vt_smooth}, \eqref{eq_q_smooth}, \eqref{eq_vc_smooth}, $\rho_1 \in L^\infty(\Omega; C^1([0,\tau] \times \T^d))$. The velocity field $u_1$ is smooth by construction. By the analysis in Section \ref{sec_defect_estimates}, the defect field $R_1 \in L^\infty(\Omega; C([0,\tau]; L^1(\T^d)))$. Hence $(\rho_1, u_1, R_1)$ is a solution to the continuity-defect equation in the sense of Definition \ref{def:cont-defect}. Note that $u_1$ is deterministic
and $\rho_{1}, R_{1}$ are each $(\mcF_{t})_{t}$-adapted, as we mollified in time with mollifiers whose support is contained in $\R_{+}$.

\subsection{Proof of \texorpdfstring{\eqref{eq_fullsoln_prop}}{(27)}}
\label{ssec_fullson}

Let us now show the property \eqref{eq_fullsoln_prop}. Assume that there is $a>0$ such that, $\P$-a.s., $\rho_0, R_0 = 0$ on $[0,\min\{\tau(\omega), a\}] \times \T^d$. Then, choosing the mollification parameter $\e>0$ small enough, we can ensure that, $\P$-a.s.,
\begin{equation*}
\rho_\e, R_\e = 0 \text{ on } [0, \min\{\tau(\omega), a - \delta\}] \times \T^d
\end{equation*}

By definition of the perturbation $\vt$ \eqref{eq_def_vt} as well as the defect terms $R_{1}$, which are all homogeneous in $R_{\e}$ or $\rho_{\e}$, we see that, $\P$-a.s.,
\begin{align*}
    \rho_{1}(t), R_{1}(t) = 0  \text{ on } [0, \min\{\tau(\omega), a - \delta\}] \times \T^d.
\end{align*}

\subsection{Choice of parameters}\label{ssec_parameter_choice}
In order to conclude the proof of Proposition \ref{prop_iteration_stage}, we need to show \eqref{eq_rho1rho0Lp}--\eqref{eq_R1_L1est}. This will be done using the estimate proven in Sections \ref{sec_perturbations}, \ref{sec_defect_estimates}, \ref{sec_momentum}. All those estimates depend on the seven parameters
\begin{equation*}
\e, \lambda, \mu, \sigma, \nu, \ell, N.
\end{equation*}
We now show that it is possible to choose these parameters in such a way that \eqref{eq_rho1rho0Lp}--\eqref{eq_R1_L1est} are satisfied. We first fix $\e>0$ such that (a first condition on the smallness of $\e$ was already introduced in Section \ref{ssec_fullson}):
\begin{enumerate}
\item in Lemma \ref{l_distance_rho}, the term $C \e$ appearing in \eqref{eq_distance_rho} is $\leq \delta/2$,
\item in Lemma \ref{l_comm}, the term $C \e$ appearing in \eqref{eq_comm} is $\leq \delta/2$,
\item in Lemma \ref{l_momentum}, the term $C\e$ appearing in \eqref{eq_momentum} is $\leq \delta/2$.
\end{enumerate}
Now  $\e>0$ is fixed. Hence all constants appearing in Sections \ref{sec_perturbations}, \ref{sec_defect_estimates}, \ref{sec_momentum} and depending on $\e$ can be very large, but they will be ``compensated'' by smallness of negative powers of $\lambda, \mu, \sigma$, or positive powers of $\ell$.

A quick look into the estimates obtained in Sections \ref{sec_perturbations}, \ref{sec_defect_estimates}, \ref{sec_momentum} shows that some terms are estimates in terms of negative powers of one of the parameters $\lambda, \mu, \sigma, \nu, \ell$ (and these terms will be then easily made arbitrarily small by picking the parameters large enough), whereas other parameters are estimated in terms of products of positive and negative powers of the parameters $\lambda, \mu, \sigma, \nu, \ell$. In this case we need to seek a balance among these five parameters (and the additional parameter $N$ appearing in some estimates). 

Similar to \cite{MS20}, we thus let
\begin{align*}
    \mu = \lambda^{\alpha}, \quad \nu = \lambda^{\gamma}, \quad \sigma = \lambda^{\beta}, \quad \ell = \lambda^{-\zeta},
\end{align*}
and we need to find the conditions on $\alpha, \beta, \gamma, \zeta$. 

% New Choice
We choose the constants in the following order:
\begin{enumerate}
 \item\label{itm_par_choice1} Choose $\alpha$ large enough such that
 \begin{equation*}
 \frac{d}{\theta} \underbrace{\left[ \frac{1}{s} + \frac{1}{\tilde p} - 1 - \frac{\theta}{d}\right]}_{>0 \text{ by \eqref{eq_cond_exp}}} \alpha  > 2
 \quad \text{and} \quad
 \underbrace{\left[ \frac{\frac{1}{2} - \kappa }{ \frac{1}{2} + \kappa} \frac{d}{s'} -1 \right]}_{>0 \text{ by \eqref{eq_rkappa}}} \alpha > 2.
 \end{equation*}
 This is possible thanks to \eqref{eq_cond_exp} and \eqref{eq_rkappa}. As a consequence 
 \begin{equation*}
 \frac{d}{\theta} \left( \frac{1}{s} + \frac{1}{\tilde p} - 1 \right) \alpha - (1 + \alpha) > 1
 \quad \text{and} \quad
\frac{\frac{1}{2} - \kappa }{ \frac{1}{2} + \kappa} \frac{d}{s'} \alpha - (1+\alpha) > 1.
  \end{equation*}
 \item\label{itm_par_choice2} We can thus find $\gamma \in \N$ such that
 \begin{align*}
    \alpha + 1 < \gamma < \min \left\{ \frac{d}{\theta} \left( \frac{1}{s} + \frac{1}{\tilde p} -1 \right),  \frac{\frac{1}{2} - \kappa }{ \frac{1}{2} + \kappa} \frac{d}{s'} \right\}\alpha,
 \end{align*}
 which we can do by the choice of $\alpha$ as the difference between the leftmost number and the rightmost number is strictly greater than $1$.

 \item\label{itm_par_choice4} Choose $\beta > 1$ such that
 \begin{align*}
\frac{d}{s'} \alpha < \beta < \frac{d}{s'}\alpha + \underbrace{( \gamma - \alpha - 1)}_{>0 \text{ by  \eqref{itm_par_choice2}}}
 \end{align*}
    enabled by condition \eqref{itm_par_choice2}.

 \item\label{itm_par_choice3} Choose a natural number $N \in \N$ large enough such that
 \begin{align*}
    \frac{N}{N-1} < \frac{\gamma}{1+\alpha},
 \end{align*}
 possible again by condition \eqref{itm_par_choice2}.

 \item\label{itm_par_choice5} Finally choose $\zeta > 1$ such that 
 \begin{equation}\label{eq_zeta}
\frac{\gamma}{\frac{1}{2} - \kappa} < \zeta < \frac{\frac{d}{s'} \alpha}{\frac{1}{2} + \kappa},
\end{equation} 
 enabled again by condition \eqref{itm_par_choice2}.
\end{enumerate}

\subsection{Estimates on the perturbations}\label{ssec_proofprop_est_pert}

We first prove \eqref{eq_rho1rho0Lp}. By Lemma \ref{l_distance_rho}, 
\begin{equation*}
\begin{aligned}
\|\rho_1 - \rho_0\|_{L^\infty_\omega C_\tau L^p_x} 
& \leq C(\|\rho_0\|_{L^\infty_\omega C^1_{\tau x}}) \eps + C(\|\rho_0\|_{L^\infty_\omega C^1_{\tau x}}, \|R_0\|_{L^\infty_\omega C_\tau L^1_x},\e) \left(      \mu^{\frac{d}{s} - \frac{d}{p}} + \frac{\mu^{\frac{d}{s'}}}{\sigma} + \mu^{-\frac{d}{s'}} + \frac{1}{\sigma} \right) \\
& \leq \frac{\delta}{2} + C(\|\rho_0\|_{L^\infty_\omega C^1_{\tau x}}, \|R_0\|_{L^\infty_\omega C_\tau L^1_x}) \left(\lambda^{\alpha \left(\frac{d}{s} - \frac{d}{p} \right)} + \lambda^{\frac{d}{s'} \alpha - \beta}  + \lambda^{-\frac{d}{s'} \alpha} + \lambda^{-\beta} \right).
\end{aligned}
\end{equation*}
The term $C\e$ is estimated by $\delta/2$ by the choice of $\e$ made at the beginning of Section \ref{ssec_parameter_choice}. Now, all powers of $\lambda$ are strictly smaller than zero (in particular the first power of $\lambda$ is $<0$ since $s>p$ and the second one because of the parameter choice \eqref{itm_par_choice4} in Section \ref{ssec_parameter_choice}). Hence, for $\lambda$ large enough, \eqref{eq_rho1rho0Lp} is achieved. 

Let us now show \eqref{eq_urho_L1}. By Lemma \ref{l_momentum}, we have
\begin{equation*}
\begin{aligned}
&\|\rho_1 u_1(\Psi)  - \rho_0 u_0(\Psi)\|_{L^\infty_\omega C_\tau L^1_x} \\
&  \leq M \|R_0\|_{L^\infty_\omega C_\tau L^1_x} + C(\|\rho_0\|_{L^\infty_\omega C^1_{\tau x}}, \|u_0\|_{C_{tx}}) \e \\
& \qquad + C ( \e, \|\rho_0\|_{L^\infty_\omega C_{\tau x}}, \|u_0\|_{C_{tx}}, \|R_0\|_{L^\infty_\omega C_{\tau} L^1_x})
\left[\frac{1}{\lambda} + \nu  \ell^{\frac{1}{2} - \kappa} +  \mu^{-\frac{d}{s'}} + \frac{1}{\sigma} + \frac{\mu^{\frac{d}{s'} }}{\sigma} 
+ \left( 1 + \frac{\mu^{d/s'}}{\sigma} \right)\frac{\lambda \mu}{\nu}   \right] \\
& \text{(by the choice of $\e$ in Section \ref{ssec_parameter_choice})} \\
&  \leq M \|R_0\|_{L^\infty_\omega C_\tau L^1_x} + \frac{\delta}{2} \\
&  + C ( \e, \|\rho_0\|_{L^\infty_\omega C_{\tau x}}, \|u_0\|_{C_{tx}}, \|R_0\|_{L^\infty_\omega C_{\tau} L^1_x})
\left[\frac{1}{\lambda} + \lambda^{\gamma - \left(\frac{1}{2} - \kappa \right) \zeta} +  \lambda^{-\frac{d}{s'} \alpha} + \lambda^{-\beta} + \lambda^{\frac{d}{s'} \alpha - \beta} 
+ \left( 1 + \lambda^{\frac{d}{s'} \alpha - \beta} \right)  \lambda^{1 + \alpha - \gamma}   \right],\end{aligned}.
\end{equation*}
As before, all powers of $\lambda$ are strictly inegative: in particular, the second power due to the parameter choice \eqref{itm_par_choice5}, the fifth one because of parameter choice \eqref{itm_par_choice4} and the last one because of parameter choices \eqref{itm_par_choice2} and \eqref{itm_par_choice4}. Hence, again picking $\lambda$ sufficiently large (depending on $\e>0$ because the constant in front of the powers of $\lambda$ depends on $\e$, but this is not a problem, as $\e$ has already been fixed),  \eqref{eq_urho_L1} is proven.

\smallskip

Let us show \eqref{eq_u1u0_Walpha} next. Combining Lemmas \ref{l_w_sob_theta} and \ref{l_wc_sob_theta}, we get
\begin{equation*}
\begin{aligned}
\|u_1 - u_0\|_{C_t W^{\theta, \tilde p}_x} 
& \leq \|w\|_{C_t W^{\theta, \tilde p}_x}+ \|w_c\|_{C_t W^{\theta, \tilde p}_x} \\
& \leq C \left( 1 + \frac{\lambda \mu}{\nu} \right) \mu^{\frac{d}{s'} -\frac{d}{ \tilde p}} \nu^{\theta} \\
\text{(using that, by assumption, $\lambda \mu/ \nu < 1$)}
& \leq C \lambda^{\left( \frac{d}{s'} - \frac{d}{\tilde p} \right) \alpha + \theta \gamma},
\end{aligned}
\end{equation*}
and, again, the power of $\lambda$ is $<0$ because of parameter choice \eqref{itm_par_choice2} in Section \ref{ssec_parameter_choice}, since
\begin{equation*}
\left( \frac{d}{s'} - \frac{d}{\tilde p} \right) \alpha + \theta \gamma < 0 \quad \Longleftrightarrow \quad
\gamma - \frac{d}{\theta} \left( \frac{1}{s} + \frac{1}{\tilde p} - 1 \right) \alpha < 0.
\end{equation*}

\newpage
\subsection{Estimates on the error terms}\label{ssec_proofprop_est_error}

The new error term $R_1$ was defined in \eqref{eq_new_error}. The term $R^{\rm com}$ satisfies, by Lemma \ref{l_comm} and the choice of $\e$ in Section \ref{ssec_parameter_choice},
\begin{equation*}
\|R^{\mathrm{com}}\|_{L^\infty_\omega C_{tx}} \leq C(\|u_0\|_{C^1}, \|\rho_0\|_{L^\infty_\omega C^1_{tx}}) \e^{1/2 - \kappa} \leq \frac{\delta}{2}.
\end{equation*}
All other terms in the definition \eqref{eq_new_error} of $R_1$ are estimated by a constant (possibly depending on $\e$) times a power of $\lambda$. If we can show that all those powers of $\lambda$ are $<0$, then, by picking $\lambda$ large enough, we achieve \eqref{eq_R1_L1est} and the proof of Proposition \ref{prop_iteration_stage} (and thus also the proof of the main Theorem \ref{thm_main_SPDE}) is concluded. 

We summarise in the following table the estimates for the various terms and, for each of them, the parameter choice in Section \ref{ssec_parameter_choice} which guarantees that the corresponding power of $\lambda$ is $<0$, thus proving \eqref{eq_R1_L1est}. 

\bigskip

{\tabulinesep=1.2mm
   \begin{tabu}[h]{l|l|l|l|l}
Term & Lemma & Estimate & Powers of $\lambda$ & Parameter choice \\
\hline
$R^{\rm quadr}$   &   \ref{lem_Rquadr}  			& 		$\left( \frac{\lambda \mu}{\nu} + \frac{1}{\lambda} \right)$	& $1 + \alpha - \gamma$ \text{ and } $-1$  & \eqref{itm_par_choice2} \\[.5em]\hline 
$R^{{\rm time}, 1}$   & \ref{lem_Rtime1}  			&	$\frac{1}{\sigma}$  & $- \beta$		& obvious	\\[.5em]\hline 
$R^{{\rm sto}, 1}$	 & \ref{lem_Rsto1}		 		&   $\nu \ell^{\frac{1}{2} - \kappa}$ & $\gamma  - \zeta \left( \frac{1}{2} - \kappa \right)$ & \eqref{itm_par_choice5}		\\[.5em]\hline 
$R^{{\rm sto},2}$ &  \ref{lem_sto2} 					&		$\frac{\ell^{-\frac{1}{2} - \kappa}	}{\sigma}$ & $\zeta \left( \frac{1}{2} + \kappa \right) - \beta$ & \eqref{itm_par_choice4} and \eqref{itm_par_choice5}\\[.5em]\hline 
$R^{{\rm sto}, 3}$ &  \ref{lem_sto3} 					&		same as $R^{{\rm sto},2}$ & 		\\[.5em]\hline 
$R^{{\rm lin}}$  & \ref{lem_Rlin}  						&		$\mu^{-d/s} + \mu^{-d/s'}$ & $-\frac{d}{s} \alpha$ and $- \frac{d}{s'} \alpha$ & 	obvious	\\[.5em]\hline 
$R^{{\rm time}, 2}$ &  \ref{lem_Rtime2} 			&		$ \frac{\lambda \mu \sigma}{\nu} \mu^{-\frac{d}{s'}}  \left( 1 + \frac{(\lambda \mu)^N}{\nu^{N-1}} \right)$  & $1 + \left( 1 - \frac{d}{s'} \right) \alpha + \beta - \gamma$  and & \eqref{itm_par_choice4}  \\
&&& $1 + \left( 1 - \frac{d}{s'} \right) \alpha + \beta - \gamma   + (1 + \alpha) N - \gamma (N-1)$	& \eqref{itm_par_choice4} and \eqref{itm_par_choice3}\\[.5em]\hline 
$R^{{\rm time}, 3}$ & \ref{lem_Rtime3}  			&	same as $R^{{\rm time}, 2}$ & 			\\[.5em]\hline 
$R^{{\rm sto}, 4}$    & \ref{lem_Rsto4} 				&	$\mu^{-d/s'} \ell^{-1/2 - \kappa}$ & $-\frac{d}{s'} \alpha + \left( \frac{1}{2} + \kappa \right) \zeta$		& \eqref{itm_par_choice5}   \\[.5em]\hline 
$R^{{\rm sto}, 5}$   & \ref{lem_Rsto5} 				&	same as $R^{{\rm sto}, 4}$	&	 \\[.5em]\hline 
$R^{q}$   & \ref{lem_Rq} 									&		$\frac{\mu^{d/s'}}{\sigma}$ & $\frac{d}{s'} \alpha - \beta$	& \eqref{itm_par_choice4}\\[.5em]\hline 
$R^{{\rm corr}}$   & \ref{lem_Rcorr} 					&		$ \left[ 1 + \frac{\mu^{d/s'}}{\sigma} \right] \frac{\lambda \mu}{\nu}$ & $1 + \alpha - \gamma$ and $ \left( \frac{d}{s'} \alpha - \beta \right) + ( 1 + \alpha - \gamma)$ & \eqref{itm_par_choice2} and \eqref{itm_par_choice4}	
\end{tabu}}

\vspace{0.3cm}

\section{Sketch of the proof of Theorem \texorpdfstring{\ref{thm:spde-transp-diff}}{1.4}}
\label{s:diffusion}

The proof of Theorem \ref{thm:spde-transp-diff} is completely similar to the one of Theorem \ref{thm_main_SPDE} and, as in that case, it is based on a \textit{Main Proposition}. One considers the following version of the continuity-defect equation \eqref{eq_cont_defect}: 
\begin{equation}
\label{eq_cont_defect_diff}
    \begin{cases}
                \partial_{t} {\rho}(t,x;\o) + {u}(t,x+B(t; \omega)) \cdot \nabla {\rho}(t,x;\o)  - \Delta \rho(t,x; \o) & = - \divv  R(t, x; \omega), \\
                \divv {u} & = 0,
    \end{cases}
\end{equation}
for which distributional solutions as in in Definition \ref{def:cont-defect} are considered. The inductive step needed in the proof of Theorem \ref{thm:spde-transp-diff} is contained in the following proposition, which corresponds to  Proposition \ref{prop_iteration_stage} in the no-diffusion case. 
  
\begin{proposition}\label{prop_iteration_stage_diff}
    There exists an absolute constant $M > 0$ such that 
    the following holds. 
    For any $\delta> 0$ and any solution $({\rho}_{0}, u_{0}, R_{0})$ of \eqref{eq_cont_defect_diff} 
        there exists another solution $( \rho_{1},  u_{1}, R_{1})$ of \eqref{eq_cont_defect_diff}
which satisfies the estimates
    \begin{align}
        \| {\rho}_{1} - {\rho}_{0} \|_{L^\infty_\omega C_t L^{p}_x} &\leq \delta, \\
        \| {\rho}_{1} {u}_{1} ( \Psi )- {\rho}_{0} {u}_{0} ( \Psi )\|_{L^\infty_\omega C_t L^{1}_x} &\leq M  \| R_{0} \|_{L^\infty_\omega C_t L^{1}_x} + \delta,  \\
        \| {u}_{1} - {u}_{0} \|_{L^\infty_\omega C_t W^{\theta,\tilde{p}}_x} &\leq \delta, \\
        \| R_{1} \|_{L^\infty_\omega C_t L^{1}_x} &\leq \delta.  
    \end{align}
   Furthermore the following holds:  
    \begin{equation}
    \begin{aligned}
        & \text{if there is $a \in (0,1]$ such that } 
        \rho_{0}, R_{0} = 0 \ \text{ $\P$-a.s.  on } [0,\min\{\tau(\omega),a\}] \times \T^d,  \\
     &   \text{then } \rho_{1}, R_{1} = 0 \ \text{ $\P$-a.s.  on } [0,\min\{\tau(\omega),a- \delta\}] \times \T^d  .
    \end{aligned}
    \end{equation}
\end{proposition}
Theorem \ref{thm:spde-transp-diff} follows from Proposition \ref{prop_iteration_stage_diff} in the very same way as Theorem \ref{thm_main_SPDE} follows from Proposition \ref{prop_iteration_stage}. The proof is Proposition \ref{prop_iteration_stage_diff} is carried out in the same way as the one of Proposition \ref{prop_iteration_stage}, with the only difference that, compared to the error decomposition \eqref{eq_new_error}, there is an additional error term in the definition of $R_1$ given by
\begin{equation*}
R^{\rm diff} := \nabla \vartheta + \nabla q, 
\end{equation*}
which can be made arbitrarily small by choosing $\lambda$ large enough, as shown in the following lemma.
\begin{lemma}
It holds that
\begin{equation*}
\|R^{\rm diff}\|_{L^\infty_\omega C_\tau L^1_x} 
\leq C \left(\frac{\nu}{\mu^{d/s'}}  + \frac{\nu}{\sigma} \right) 
\leq C \frac{\nu}{\mu^{d/s'}}
\leq C  \lambda^{\gamma - \frac{d}{s'} \alpha},
\end{equation*}
where $C = C \left(\e,  \|R_0\|_{L^\infty_\omega C_\tau L^1_x} \right)$ and the exponent $\gamma - \frac{d}{s'} \alpha < 0$. 
%\begin{equation*}
%\|R^{\rm diff}\|_{L^\infty_\omega C_\tau L^1_x} 
%\leq C \left(\e,  \|R_0\|_{L^\infty_\omega C_\tau L^1_x} \right) \left(\frac{\nu}{\mu^{d/s'}}  + \frac{\nu}{\sigma} \right) 
%\leq C \left(\e,  \|R_0\|_{L^\infty_\omega C_\tau L^1_x} \right) \frac{\nu}{\mu^{d/s'}}
%\leq C \left(\e,  \|R_0\|_{L^\infty_\omega C_\tau L^1_x} \right) \lambda^{\gamma - \frac{d}{s'} \alpha}
%\end{equation*}
\end{lemma}
\begin{proof}
We have
\begin{equation*}
\|\nabla \vartheta(t, \cdot; \omega)\|_{L^1_x} \leq \sum_j \|R_\e^j (t, \cdot; \omega)\|_{C^1_x} \|\Theta^j (\Psi_\ell(t, \cdot; \omega))\|_{W^{1,1}_x}
=  \sum_j \|R_\e^j (t, \cdot; \omega)\|_{C^1_x} \|\Theta^j \|_{W^{1,1}_x}.
\end{equation*}
Hence, using Lemma \ref{l_mollified_error} for $R_\e^j$ and \eqref{eq_Theta} for $\Theta^j$, we get
\begin{equation*}
\|\nabla \vartheta\|_{L^\infty_\omega C_\tau L^1_x} \leq \sum_j \|R_\e^j\|_{L^\infty_\omega C_\tau C^1_x} \|\Theta^j\|_{W^{1,1}_{x}} \leq C \left(\e,  \|R_0\|_{L^\infty_\omega C_\tau L^1_x} \right) \mu^{-\frac{d}{s'}} \nu. 
\end{equation*}
A similar argument gives
\begin{equation*}
\|\nabla q\|_{L^\infty_\omega C_\tau L^1_x} \leq C \left(\e,  \|R_0\|_{L^\infty_\omega C_\tau L^1_x} \right) \frac{\nu}{\sigma},
\end{equation*}
thus proving the first inequality in the statement. The second inequality in the statement follows from the fact that $\mu^{d/s'} \leq \sigma$ by parameter choice \eqref{itm_par_choice4}. The last inequality follows from the definition of $\mu = \lambda^\alpha, \nu = \lambda^\gamma$. Finally, we have, by parameter choice \eqref{itm_par_choice2}, 
\begin{equation*}
\gamma < \underbrace{\frac{\frac{1}{2} - \kappa}{\frac{1}{2} + \kappa}}_{<1} \frac{d}{s'} \alpha < \frac{d}{s'} \alpha,
\end{equation*}
thus showing that $\gamma - \frac{d}{s'} \alpha < 0$ . This concludes the proof of the lemma, and hence also the proof of Proposition \ref{prop_iteration_stage_diff} and Theorem \ref{thm:spde-transp-diff}.  
\end{proof} 

\appendix
\section{An Interpolation Inequality for Fractional Sobolev Spaces}\label{sec_AppA}

We will need the following interpolation inequality which is most likely well-known but we could not find a precise reference for it. For the convenience of the reader, we give a short argument for it.

\begin{proposition}\label{prop_interpolation}
    Let $q \in (1,\infty)$ and $f \in W^{1,q}(\T^{d})$. Then for any $\theta \in (0,1)$,
    \begin{equation}
    \label{eq_interpolation}
        \| f \|_{W^{\theta, q}(\T^{d})} \leq C \| f \|_{L^{q}(\T^{d})}^{1-\theta} \| f \|_{W^{1,q}(\T^{d})}^{\theta}.
    \end{equation}
\end{proposition}
\begin{proof}
    Note that we have $W^{1,q} \subset W^{0,q} = L^{q}$. We then apply \cite[Ch. 13, Eq. $(6.6)$, p. 23]{Taylor3} with $M = \T^{d}$, $\theta = \theta$, $\sigma = 0$, $s = 1$, $p = q > 1$ to get
    \begin{align*}
        [W^{0,q}(\T^{d}),W^{1,q}(\T^{d})]_{\theta} = W^{\theta,q}(\T^{d}).
    \end{align*}
    We combine this with \cite[Corollary 2.8, p. 53]{Lunardi18} with $X = L^{q}(\T^{d};\C)$, $Y = W^{1,q}(\T^{d};\C)$ to get 
    \begin{align*}
        \| f \|_{W^{\theta, q}(\T^{d};\C)} \leq C \| f \|_{[W^{0,q}(\T^{d}),W^{1,q}(\T^{d})]_{\theta}} \leq C \| f \|_{L^{q}(\T^{d};\C)}^{1-\theta} \| f \|_{W^{1, q}(\T^{d};\C)}^{\theta}.
    \end{align*}
    But our function $f$ is real-valued, so the norms in the above inequality are all over real-valued $W^{\theta,q}$-spaces, which concludes the proof. 
\end{proof}

\section{Uniqueness of the Stochastic Transport Equation}\label{sec_AppB}
In this section, we want to sketch how to modify the uniqueness proof of \cite{BFGM19} in order to include the whole 
parameter range below the lines $\frac{1}{p} + \frac{1}{\tilde{p}} = 1 + \frac{\theta}{d}$ and $\frac{1}{\tilde p} = \frac{1}{d} + \frac{\theta}{d}$ of Fig. \ref{fig_main_result} (i.e., the (dotted) green and blue areas), without assuming $p \geq 2$.

The uniqueness proof of \cite{BFGM19} uses two essential ingredients: 
\begin{enumerate}
 \item A proof of regularity for the final value problem for the dual backwards transport equation if the final value is sufficiently regular, cf. \cite[Corollary 3.7, after Corollary 2.8]{BFGM19}, and Proposition \ref{prop_AppB_approx} below.
 \item A duality approach, essentially testing the solution to the stochastic transport equation \eqref{eq_STE_proper} with the (time-dependent) test function that is the solution of the dual backwards transport equation, cf. \cite[Proposition 3.12]{BFGM19}, and Proposition \ref{prop_AppB_duality} below.
\end{enumerate}
\begin{remark}
 Both proofs work in exactly the same way in the case we are considering in this Appendix, i.e. the (dotted) green and blue areas of Figure \ref{fig_main_result},  and are actually much simplified (but still rather lengthy), as in our case, many of the complications of that paper do not arise. Let us name a few examples:  in \cite[Condition 2.2]{BFGM19}, in our case, in their notation, $b^{(2)} = 0$ and $c^{(1)} = c^{(2)} = 0$. Furthermore, in the present paper, we only deal with the much easier case of the torus $\T^{d}$, so a number of weights, cutoff functions and additional integrability conditions are not necessary or automatically satisfied, respectively. For example, in \cite[Theorem 2.7, Corollary 2.8]{BFGM19}, $\chi \equiv 1$ is sufficient for us, in \cite[Lemma 3.10]{BFGM19}, we immediately get the global result for $R = \infty$, all local integrability conditions are global, all functions immediately have compact support as $\T^{d}$ is compact. Furthermore, in the crucial uniqueness theorem, \cite[Theorem 3.14]{BFGM19}, we can choose the exponent $\alpha = 0$.
\end{remark}
  In the following, we state the versions of the main ingredients as they should hold in our case, and only sketch how they can be combined to show the actual uniqueness of Eq. \eqref{eq_STE_proper}.

We assume the following conditions on our coefficients
\begin{align*}
    u \colon [0,{1}] \times \T^{d} \to \R^{d}, \quad \text{\textit{deterministic}}, \quad \divv u = 0, \quad u \in L_{t}^{\infty} W_{x}^{\theta,\tilde{p}}.
\end{align*}
We will assume that the parameters $p,\tilde{p},\theta$ satisfy $p \in (1,\infty)$, $\tilde{p} \in (1,\infty)$, $\theta \in [0,1]$ such that
\begin{align}\label{eq_main_scaling_cond_app_1}
    \frac{1}{p} + \frac{1}{\tilde{p}} < 1 + \frac{\theta}{d}, \quad \frac{d}{\tilde{p}} < 1 + \theta,
\end{align}
i.e. that  the pair $(1/p, 1/\tilde p)$ belongs to the (dotted) green or blue area of Figure \ref{fig_main_result}. By the Sobolev embedding, we have
\begin{align*}
    u \in L_{t}^{\infty}W_{x}^{\theta,\tilde{p}} \subset L_{t}^{\infty} L_{x}^{q},
\end{align*}
where $q$ now satisfies
the conditions
\begin{equation}
 \label{eq_main_scaling_cond_app}
    \frac{1}{p} + \frac{1}{q} < 1, \quad \frac{d}{q} < 1.
\end{equation}
For the greatest technical simplicity, we assume strict inequalities in \eqref{eq_main_scaling_cond_app_1} (and thus also in \eqref{eq_main_scaling_cond_app}). Beck \textit{et al.} \cite{BFGM19} also treat the critical case of the LPS condition, which corresponds to $d/q = 1$ in \eqref{eq_main_scaling_cond_app}, but this needs an additional smallness condition, cf. \cite[Condition 2.2 c)]{BFGM19}. 
We expect that, under a similar smallness condition, the argument we outline below should yield uniqueness also in the case $d/q = 1$, as the smallness condition  seems to be needed only in the proof of the \textit{a priori} estimates in Proposition \ref{prop_AppB_approx} below, cf. \cite[Theorem 2.7]{BFGM19}.

We will also exclude the case of exact equality in the first condition in \eqref{eq_main_scaling_cond_app_1}, which corresponds to exact duality of the parameters $p,q$ (cf. \eqref{eq_main_scaling_cond_app}, after Sobolev embedding). {\color{red} Note also that we assumed $p > 1$.} 
The argument sketched below, in its current form, can not accommodate either of these cases, cf. Eq.  \eqref{eq_uniq_prime_est}, where it is important that $p'r' < \infty$ and hence {\color{red}$r = q/p' > 1$ and $p > 1$.}

Our notion of solution is still the one of Definition \ref{def_local_weak_soln} for $\rho$ of class $L^{\infty}_{\o} L_{t}^{\infty} L_{x}^{p}$. Note, however, that the condition $\rho u \in L_{\o}^{\infty} L_{t}^{\infty} L_{x}^{1}$ is now automatically satisfied by Eq. \eqref{eq_main_scaling_cond_app} and H\"older's inequality, as it was for the solutions of \cite{BFGM19}. Furthermore, the results we obtain in this appendix will hold for \textit{global-in-time} solutions, so the stopping time of Definition \ref{def_local_weak_soln} can be taken as $\tau = 1$.

We will work, similar to the rest of the paper, with the shifted equation, for which we define a notion of solution as in \cite[Definition 3.3]{BFGM19}. Furthermore, we will need this also for the backward equation to apply the duality method.

\newpage
\begin{definition}[Weak solutions to the forward and backward shifted equations]\label{def_weak_soln_shifted} 
Let $p \geq 1$, $u \in L^{1}([0,1] \times \T^{d};\R^{d})$ be a deterministic vector field and let $\o \in \O$ be fixed. Define $\tilde{u}(t,x;\o) := u(t, x - B(t;\o))$. We will write $\tilde{u}(\o)$ for the function $(t,x) \mapsto \tilde{u}(t,x;\o)$ and use similar notations for other random space-time functions.
\begin{enumerate}
 \item Let $\rho_{0} \in L^{p}(\T^{d})$. A function $\tilde{\rho}(\o) \colon  [0,1] \times \T^{d} \to \R$ is a \textbf{weak solution of class ${L_{t}^{\infty}}L^{p}_{x}$ to the  shifted (forwards) equation}
\begin{equation}\label{eq_forward_shifted}
    \begin{cases}
    \partial_{t} \tilde{\rho}(\o) + \tilde{u}(\o) \cdot \nabla \tilde{\rho}(\o) &= 0, \qquad t \in [0,1],  \\
    \tilde{\rho}(0;\o) &= \rho_{0},
    \end{cases}
\end{equation}
if
\begin{enumerate}[label = (\roman*)]
 \item \label{itm_def_fwd_shifted_soln1} $\tilde{\rho}(\o)$ is in ${L^{\infty}}([0,1];L^{p}(\T^{d}))$ and $\tilde{\rho}(\o) \tilde{u}(\o)$ is in ${L^{\infty}}([0,1]; L^{1}(\T^{d};\R^{d}))$.
 \item  \label{itm_def_fwd_shifted_soln2} For every $\phi \in C^{1}([0,1];C^{\infty}(\T^{d}))$\footnote{For a function $f = f(t,x)$, by $f \in C^1([0,1]; C^\infty(\T^d))$ we mean that both $f$ and its (distributional) time derivative $\partial_t f$ belong to $C([0,1]; C^k(\T^d))$ for all $k \in \N$.} ,  $t \mapsto  \int_{\T^{d}} \tilde{\rho}(t,x;\o) \phi(t,x) dx$ is continuous; more precisely, the map has a continuous representative, where by  representative we mean a function which coincides with $t \mapsto \int \tilde{\rho}(t,x;\o) \phi(t,x) dx$ for a.e. $t \in [0,1]$.
 \item \label{itm_def_fwd_shifted_soln3} For every $\phi \in C^{1}([0,1];C^{\infty}(\T^{d}))$, for this continuous modification of $\int_{\T^{d}} \tilde{\rho}(t,x;\o) \phi(t,x) dx$, it holds that, for all $t \in [0,1]$,
 \begin{equation}
\int_{\T^d} \tilde{\rho}(t,x; \omega) \phi(t,x) dx =  \int_{\T^d} {\rho}_0(x) \phi(0,x) dx + \int_0^t  \int_{\T^d} \tilde{\rho}(s,x; \omega) \tilde{u}(s,x;\o) \cdot \nabla \phi(s,x) dx ds.
\end{equation}
\end{enumerate}
 \item Let $\psi_{0} \in L^{p}(\T^{d})$ and $t_{f} \in [0,1]$. 
A function $\tilde{\psi}(\o) \colon  [0,t_{f}] \times \T^{d} \to \R$ is a \textbf{weak solution of class ${L_{t}^{\infty}}L^{p}_{x}$ to the dual shifted backwards equation}
\begin{equation}\label{eq_backward_shifted}
    \begin{cases}
    \partial_{t} \tilde{\psi}(\o) + \tilde{u}(\o) \cdot \nabla \tilde{\psi}(\o) &= 0, \qquad t \in [0,t_{f}],  \\
    \tilde{\psi}(t_{f};\o) &= \psi_{0}(x - B(t_{f};\o)),
    \end{cases}
\end{equation}
if 
\begin{enumerate}[label = (\roman*)]
 \item \label{itm_def_shifted_soln1} $\tilde{\psi}(\o)$ is in ${L^{\infty}}([0,t_{f}];L^{p}(\T^{d}))$ and $\tilde{\psi}(\o) \tilde{u}(\o)$ is in ${L^{\infty}}([0,t_{f}]; L^{1}(\T^{d};\R^{d}))$.
 \item  \label{itm_def_shifted_soln2} For every $\phi \in C^{1}([0,t_{f}];C^{\infty}(\T^{d}))$,  $t \mapsto  \int_{\T^{d}} \tilde{\psi}(t,x;\o) \phi(t,x) dx$ is continuous; more precisely, the map has a continuous representative, where by  representative we mean a function which coincides with $t \mapsto \int \tilde{\psi}(t,x;\o) \phi(t,x) dx$ for a.e. $t \in [0,t_{f}]$.
 \item \label{itm_def_shifted_soln3} For every $\phi \in C^{1}([0,t_{f}];C^{\infty}(\T^{d}))$, for this continuous modification of $\int_{\T^{d}} \tilde{\psi}(t,x;\o) \phi(t,x) dx$, it holds that, for all $t \in [0,t_{f}]$,
 \begin{equation}
\int_{\T^d} \tilde{\psi}(t,x; \omega) \phi(t,x) dx =  \int_{\T^d} {\psi}_0(x-B(t_{f};\o)) \phi(t_{f},x) dx - \int_t^{t_f}  \int_{\T^d} \tilde{\psi}(s,x; \omega) \tilde{u}(s,x;\o) \cdot \nabla \phi(s,x) dx ds.
\end{equation}
\end{enumerate}
\end{enumerate}

\end{definition}

The first main ingredient is the following.
\begin{proposition}[Analogue of  \cite{BFGM19}, Corollary 3.7]\label{prop_AppB_approx}
    Let $m$ be an even integer and $t_{f} \in [0,1]$. 
    Let $W(t) := B(t) - B(t_{f})$ be the associated backwards Brownian motion. Assume that $u \in L_{t}^{\infty} L_{x}^{q}$ and that the conditions \eqref{eq_main_scaling_cond_app} hold. Further assume that there are $C^{\infty}([0,1]\times \T^{d})$ functions $(u_{\eps})_{\eps > 0}$  satisfying 
    \begin{equation}\label{eq_uniq_approx_ass}
        u_{\eps} \overset{\eps \downarrow 0}{\to} u \ \text{a.e. in }[0,1] \times \T^{d} \text{ and } \| u_{\eps} \|_{L^{\infty}_{t} L_{x}^{q}} \leq 2 \| u \|_{L^{\infty}_{t} L_{x}^{q}}.
    \end{equation}
    Then there exists a constant $C$, independent of $\eps$, such that, for every $\psi_{0} \in C^{\infty}(\T^{d})$, the solution $\psi_{\eps}$ of 
    \begin{equation}\label{eq_backward_STE}
        \begin{cases}
                 d\psi_{\eps} + u_{\eps} \cdot \nabla \psi_{\eps} dt + \nabla \psi_{\eps} \circ dW(t) &= 0, \quad t \in [0,t_{f}] \\
                 \psi_{\eps}(t_{f}) &= \psi_{0},
        \end{cases}
    \end{equation}
    is smooth and satisfies, for all $\eps \in (0,1)$
    \begin{align*}
        \sup_{t \in [0,1]} \E \left[ \| \psi_{\eps}(t,\cdot) \|_{W^{1,m}(\T^{d})}^{m} \right] \leq C \| \psi_{0} \|_{W^{1,2m}(\T^{d})}^{m}.
    \end{align*}
\end{proposition}
As mentioned before, the proof of this result should be the same as that of \cite[Corollary 3.7]{BFGM19}.

The second main ingredient is a duality formula which follows from relatively abstract arguments, all of which should be valid in our case as well.

\begin{proposition}[Analogue of \cite{BFGM19}, Proposition 3.12]\label{prop_AppB_duality}
    Given $t_{f} \in [0,1]$, $\psi_{0} \in C^{\infty}(\T^{d})$ and $\eps > 0$, and $\psi_{\eps}$ be the smooth solution of the backward stochastic transport equation \eqref{eq_backward_STE}. For some $\o \in \O$, assume that {$\psi_{\eps}(\o) \in C^{1}([0,1];C^{\infty}(\T^{d}))$} and that the identity \eqref{eq_backward_shifted} holds for $\tilde{\psi}_{\eps}(\o) := \psi_{\eps}(t,x-B(t;\o))$ under the assumption \eqref{eq_uniq_approx_ass}. If $\tilde{\rho}(\o)$ is any weak solution of Eq. \eqref{eq_forward_shifted} of class $L^{\infty}_{t} L^{p}_{x}$ corresponding to that $\o$, then we have 
    \begin{equation}\label{eq_duality_formula}
        \int_{\T^{d}} \tilde{\rho}(t_{f};\o) \psi_{0}(\cdot - B(t_{f};\o)) dx = \int_{\T^{d}} \rho_{0} \tilde{\psi}_{\eps}(0;\o) dx + \int_{0}^{t_{f}} \int_{\T^{d}} \tilde{\rho}(s;\o) (\tilde{u}(s;\o) - \tilde{u}_{\eps}(s;\o)) \cdot \nabla \tilde{\psi}_{\eps}(s;\o) dx ds.
    \end{equation}

\end{proposition}

Now we combine the two ingredients to get the main result of this section.

\begin{proposition}[Analogue of \cite{BFGM19}, Theorem 3.14]
 There exists a full measure set $\O_{0} \subset \O$ such that for all $\o \in \O_{0}$ the following property holds: for every $\rho_{0} \in L^{p}(\T^{d})$, Eq. \eqref{eq_forward_shifted} has at most one weak solution $\tilde{\rho}(\o) \colon [0,1] \times \T^{d} \to \R$ of class $L^{\infty}_{t} L_{x}^{p}$. In particular, pathwise uniqueness holds for the stochastic transport equation \eqref{eq_STE_proper}.
\end{proposition}
\begin{proof}[Sketch of proof]
    Since the equation is linear, if we take two solutions $\tilde{\rho}_{1}(\o), \tilde{\rho}_{2}(\o)$ to the same initial condition, their difference $\tilde{\rho}(\o) := \tilde{\rho}_{1}(\o) - \tilde{\rho}_{2}(\o)$ satisfies the same equation but with zero initial condition. From Eq. \eqref{eq_duality_formula} for given $t_{f} \in [0,1]$, $\psi_{0} \in C^{\infty}(\T^{d})$ and $\eps > 0$, we then get that 
    \begin{equation}\label{eq_uniq_proof_dual}
        \int_{\T^{d}} \tilde{\rho}(t_{f};\o) \psi_{0}(\cdot - B(t_{f};\o)) dx = \int_{0}^{t_{f}} \int_{\T^{d}} \tilde{\rho}(s;\o) (\tilde{u}(s;\o) - \tilde{u}_{\eps}(s;\o)) \cdot \nabla \tilde{\psi}_{\eps}(s) dx ds.
    \end{equation}
        
    As a first step, fix $t_{f} \in [0,1]$ and $\psi_{0} \in C^{\infty}(\T^{d})$. Let us observe that by applying the H\"older inequality to the space-time integral, we find (noting that $r = \frac{q}{p'} > 1$ by \eqref{eq_main_scaling_cond_app}, and hence $r' < \infty$)
    \begin{equation}\label{eq_uniq_prime_est}
        \begin{split}
        \E \int_{0}^{t_{f}}\int | (\tilde{u} - \tilde{u}_{\eps}) \cdot \nabla \tilde{\psi}_{\eps} |^{p'} dx ds  &\leq \E \left[ \left( \int_{0}^{t_{f}}\int | \tilde{u} - \tilde{u}_{\eps} |^{q} dx ds \right)^{1/r} \left( \int_{0}^{t_{f}}\int | \nabla \tilde{\psi}_{\eps} |^{p' r'} dx ds \right)^{1/r'} \right] \\
        &=  \left( \int_{0}^{t_{f}}\int | u - u_{\eps} |^{q} dx ds \right)^{1/r} \E \left[ \left( \int_{0}^{t_{f}}\int | \nabla \tilde{\psi}_{\eps} |^{p' r'} dx ds \right)^{1/r'} \right] \\
        &\overset{\eps \to 0}{\longrightarrow} 0, 
        \end{split}
    \end{equation}
    where we used that the first term on the right-hand side converges by \eqref{eq_uniq_approx_ass}, whereas the second term is bounded by a constant $C = C_{\psi_{0},t_{f}}$ by Proposition \ref{prop_AppB_approx} for $m =2\lceil p'r' \rceil < \infty$.
    This implies the existence of a set of full measure $\O_{t_{f},\psi_{0}}$ and a subsequence of $(\eps_{n})_{n}$ such that the convergence in \eqref{eq_uniq_prime_est} holds pointwise in $\o$ for all $\o \in \O_{t_{f},\psi_{0}}$. 
    
    This, in turn, implies that for all $\o \in \O_{t_{f},\psi_{0}}$ and a constant $C(\o) := \| \tilde{\rho}(\o) \|_{L^{\infty}_{t} L_{x}^{p}} < \infty$ (note the $\o$-dependence and that we are not assuming a uniform bound in $\o$), we find for the right-hand side of Eq. \ref{eq_uniq_proof_dual}
    \begin{align*}
        &\left| \int_{0}^{t_{f}}\int \tilde{\rho}(\o) (\tilde{u}(\o) - \tilde{u}_{\eps}(\o)) \cdot \nabla \tilde{\psi}_{\eps}(\o) dxds \right| \\
        &\leq \left( \int_{0}^{t_{f}} \int |\tilde{\rho}(\o)|^{p} dxds \right)^{1/p} \left( \int_{0}^{t_{f}}\int | (\tilde{u}(\o) - \tilde{u}_{\eps}(\o)) \cdot \nabla \tilde{\psi}_{\eps}(\o) |^{p'} dx ds  \right)^{1/p'} \\
        &\leq C(\o)\left( \int_{0}^{t_{f}}\int | (\tilde{u}(\o) - \tilde{u}_{\eps}(\o)) \cdot \nabla \tilde{\psi}_{\eps}(\o) |^{p'} dx ds  \right)^{1/p'} \overset{n \to \infty}{\longrightarrow} 0.
    \end{align*}
    The above proof worked for fixed $\psi_{0} \in C^{\infty}(\T^{d})$ and $t_{f} \in [0,1]$ and produced a set of full measure $\O_{t_{f},\psi_{0}}$. 
    A careful diagonal argument using separability to find countable, dense subsets $\mathcal{D}_{t} \subset [0,1]$ and $\mathcal{D} \subset C^{\infty}(\T^{d})$ (with respect to the $L^{p'}$-norm)
    then yields the set $\O_{0}$ as the intersection of the countably many sets $\O_{t_{f},\psi_{0}}$. Further details are given in the proof of \cite[Theorem 3.14]{BFGM19}. 
    
   Thus we have shown that for $\o \in \O_{0}$, $t_{f} \in \mathcal{D}_{t}$,
    \begin{equation}\label{eq_uniq_fundlemma}
        \int_{\T^{d}} \tilde{\rho}(t_{f};\o) \psi_{0}(\cdot -B(t_{f};\o)) dx = 0 \quad \forall \psi_{0} \in \mathcal{D}.
    \end{equation}
    By the fundamental lemma of the calculus of variations, we conclude that $\tilde{\rho}(t_{f};\o) = 0$. Recall that $\tilde{\rho}$ was continuous as a function of time with values in distributions. This, together with the denseness of $\mathcal{D}_{t} \subset [0,1]$, implies
    \begin{align*}
        \tilde{\rho}(\o) = 0
    \end{align*}
    for $\o \in \O_{0}$, which yields the desired uniqueness. 
\end{proof}

\section*{Acknowledgements}
A.S. gratefully acknowledges the support by the German Research Foundation (DFG) through the SFB 1283 and through the Walter Benjamin Programme, Project number 507913792. The authors would further like to thank Franco Flandoli, Guy Foghem, Chengcheng Ling, Mario Maurelli and Stefano Spirito for helpful discussions. 

 \bibliographystyle{plain}
\bibliography{refs-1}

\def\ocirc#1{\ifmmode\setbox0=\hbox{$#1$}\dimen0=\ht0 \advance\dimen0
  by1pt\rlap{\hbox to\wd0{\hss\raise\dimen0
  \hbox{\hskip.2em$\scriptscriptstyle\circ$}\hss}}#1\else {\accent"17 #1}\fi}
\begin{thebibliography}{10}

\bibitem{Aizenman78}
Michael Aizenman.
\newblock On vector fields as generators of flows: a counterexample to
  {N}elson's conjecture.
\newblock {\em Ann. of Math. (2)}, 107(2):287--296, 1978.

\bibitem{Ambrosio04}
Luigi Ambrosio.
\newblock Transport equation and {C}auchy problem for {$BV$} vector fields.
\newblock {\em Invent. Math.}, 158(2):227--260, 2004.

\bibitem{AF11}
Stefano Attanasio and Franco Flandoli.
\newblock Renormalized solutions for stochastic transport equations and the
  regularization by bilinear multiplication noise.
\newblock {\em Comm. Partial Differential Equations}, 36(8):1455--1474, 2011.

\bibitem{BFGM19}
Lisa Beck, Franco Flandoli, Massimiliano Gubinelli, and Mario Maurelli.
\newblock Stochastic {ODE}s and stochastic linear {PDE}s with critical drift:
  regularity, duality and uniqueness.
\newblock {\em Electron. J. Probab.}, 24:Paper No. 136, 72, 2019.

\bibitem{Berkemeier22}
Stefanie~Elisabeth Berkemeier.
\newblock On the 3{D} {N}avier--{S}tokes equations with a linear multiplicative
  noise and prescribed energy.
\newblock {\em arXiv preprint}, \url{https://arxiv.org/abs/2212.11257}, 2022.

\bibitem{Bianchini:2017vf}
Stefano Bianchini and Paolo Bonicatto.
\newblock A uniqueness result for the decomposition of vector fields in
  {$\mathbb R^d$}.
\newblock {\em Invent. Math.}, 220(1):255--393, 2020.

\bibitem{bianchini1}
Stefano Bianchini, Maria Colombo, Gianluca Crippa, and Laura Spinolo.
\newblock Optimality of integrability estimates for advection–diffusion
  equations.
\newblock {\em NODEA}, 24(33), 2017.

\bibitem{BFH20}
Dominic Breit, Eduard Feireisl, and Martina Hofmanov\'{a}.
\newblock On solvability and ill-posedness of the compressible {E}uler system
  subject to stochastic forces.
\newblock {\em Anal. PDE}, 13(2):371--402, 2020.

\bibitem{brue2021positive}
Elia Bru{\'e}, Maria Colombo, and Camillo De~Lellis.
\newblock Positive solutions of transport equations and classical nonuniqueness
  of characteristic curves.
\newblock {\em Archive for Rational Mechanics and Analysis}, 240(2):1055--1090,
  2021.

\bibitem{BDLSV18}
Tristan Buckmaster, Camillo de~Lellis, L\'{a}szl\'{o} Sz\'{e}kelyhidi, Jr., and
  Vlad Vicol.
\newblock Onsager's conjecture for admissible weak solutions.
\newblock {\em Comm. Pure Appl. Math.}, 72(2):229--274, 2019.

\bibitem{novack2}
Tristan Buckmaster, Nader Masmoudi, Matthew Novack, and Vlad Vicol.
\newblock Intermittent convex integration for the 3d euler equations.
\newblock {\em Ann. of Math. Studies}, 217, 2023.

\bibitem{BV19b}
Tristan Buckmaster and Vlad Vicol.
\newblock Convex integration and phenomenologies in turbulence.
\newblock {\em EMS Surv. Math. Sci.}, 6(1-2):173--263, 2019.

\bibitem{BV19a}
Tristan Buckmaster and Vlad Vicol.
\newblock Nonuniqueness of weak solutions to the {N}avier-{S}tokes equation.
\newblock {\em Ann. of Math. (2)}, 189(1):101--144, 2019.

\bibitem{CDZ22}
Weiquan Chen, Dong Zhao, and Xiangchan Zhu.
\newblock Sharp non-uniqueness of solutions to stochastic 3{D}
  {N}avier--{S}tokes equations.
\newblock {\em arXiv preprint}, \url{https://arxiv.org/abs/2208.08321}, 2022.

\bibitem{CL21}
Alexey Cheskidov and Xiaoyutao Luo.
\newblock Nonuniqueness of weak solutions for the transport equation at
  critical space regularity.
\newblock {\em Ann. PDE}, 7(1):Paper No. 2, 45, 2021.

\bibitem{CL22}
Alexey Cheskidov and Xiaoyutao Luo.
\newblock Extreme temporal intermittency in the linear {S}obolev transport:
  almost smooth nonunique solutions.
\newblock {\em arXiv preprint}, \url{https://arxiv.org/abs/2204.08950}, 2022.

\bibitem{CFF19}
Elisabetta Chiodaroli, Eduard Feireisl, and Franco Flandoli.
\newblock Ill posedness for the full {E}uler system driven by multiplicative
  white noise.
\newblock {\em arXiv preprint}, \url{https://arxiv.org/abs/1904.07977}, 2019.

\bibitem{Crippa1}
Gianluca Crippa, Camilla Nobili, Christian Seis, and Stefano Spirito.
\newblock {E}ulerian and {L}agrangian solutions to the continuity and {E}uler
  equations with {$L^1$} vorticity.
\newblock {\em Siam J. Math. Anal.}, 49(5):3973–3998, 2017.

\bibitem{DS17}
Sara Daneri and L\'{a}szl\'{o} Sz\'{e}kelyhidi, Jr.
\newblock Non-uniqueness and h-principle for {H}\"{o}lder-continuous weak
  solutions of the {E}uler equations.
\newblock {\em Arch. Ration. Mech. Anal.}, 224(2):471--514, 2017.

\bibitem{DLS09}
Camillo De~Lellis and L\'{a}szl\'{o} Sz\'{e}kelyhidi, Jr.
\newblock The {E}uler equations as a differential inclusion.
\newblock {\em Ann. of Math. (2)}, 170(3):1417--1436, 2009.

\bibitem{Depauw03}
Nicolas Depauw.
\newblock Non-unicit\'{e} du transport par un champ de vecteurs presque {BV}.
\newblock In {\em Seminaire: \'{E}quations aux {D}\'{e}riv\'{e}es {P}artielles,
  2002--2003}, S\'{e}min. \'{E}qu. D\'{e}riv. Partielles, pages Exp. No. XIX,
  9. \'{E}cole Polytech., Palaiseau, 2003.

\bibitem{DPL89}
Ronald~J. DiPerna and Pierre-Louis Lions.
\newblock Ordinary differential equations, transport theory and {S}obolev
  spaces.
\newblock {\em Invent. Math.}, 98(3):511--547, 1989.

\bibitem{FF11}
Ennio Fedrizzi and Franco Flandoli.
\newblock Pathwise uniqueness and continuous dependence of {SDE}s with
  non-regular drift.
\newblock {\em Stochastics}, 83(3):241--257, 2011.

\bibitem{FF13}
Ennio Fedrizzi and Franco Flandoli.
\newblock Noise prevents singularities in linear transport equations.
\newblock {\em J. Funct. Anal.}, 264(6):1329--1354, 2013.

\bibitem{FNO18}
Ennio Fedrizzi, Wladimir Neves, and Christian Olivera.
\newblock On a class of stochastic transport equations for {$L^2_{\rm loc}$}
  vector fields.
\newblock {\em Ann. Sc. Norm. Super. Pisa Cl. Sci. (5)}, 18(2):397--419, 2018.

\bibitem{Figalli08}
Alessio Figalli.
\newblock Existence and uniqueness of martingale solutions for {SDE}s with
  rough or degenerate coefficients.
\newblock {\em J. Funct. Anal.}, 254(1):109--153, 2008.

\bibitem{FGP10}
Franco Flandoli, Massimiliano Gubinelli, and Enrico Priola.
\newblock Well-posedness of the transport equation by stochastic perturbation.
\newblock {\em Invent. Math.}, 180(1):1--53, 2010.

\bibitem{Frisch68}
Uriel Frisch.
\newblock Wave propagation in random media.
\newblock In {\em Probabilistic {M}ethods in {A}pplied {M}athematics, {V}ol.
  1}, pages 75--198. Academic Press, New York, 1968.

\bibitem{Funaki79}
Tadahisa Funaki.
\newblock Construction of a solution of random transport equation with boundary
  condition.
\newblock {\em J. Math. Soc. Japan}, 31(4):719--744, 1979.

\bibitem{giri20233}
Vikram Giri, Hyunju Kwon, and Matthew Novack.
\newblock The {$L^3$}-based strong {O}nsager theorem.
\newblock {\em arXiv preprint}, \url{https://arxiv.org/abs/2305.18509}, 2023.

\bibitem{giri2023wavelet}
Vikram Giri, Hyunju Kwon, and Matthew Novack.
\newblock A wavelet-inspired {$L^3$}-based convex integration framework for the
  {E}uler equations.
\newblock {\em arXiv preprint}, \url{https://arxiv.org/abs/2305.18142}, 2023.

\bibitem{giri2021non}
Vikram Giri and Massimo Sorella.
\newblock Non-uniqueness of integral curves for autonomous {H}amiltonian vector
  fields.
\newblock {\em Differential Integral Equations}, 35(7-8):411--436, 2022.

\bibitem{Gromov86}
Mikhael Gromov.
\newblock {\em Partial differential relations}, volume~9 of {\em Ergebnisse der
  Mathematik und ihrer Grenzgebiete (3) [Results in Mathematics and Related
  Areas (3)]}.
\newblock Springer-Verlag, Berlin, 1986.

\bibitem{HLP22}
Martina Hofmanov\'{a}, Theresa Lange, and Umberto Pappalettera.
\newblock Global existence and non-uniqueness of 3{D} {E}uler equations
  perturbed by transport noise.
\newblock {\em arXiv preprint}, \url{https://arxiv.org/abs/2212.12217}, 2022.

\bibitem{HZZ19}
Martina Hofmanov\'{a}, Rongchan Zhu, and Xiangchan Zhu.
\newblock Non-uniqueness in law of stochastic 3{D} {N}avier--{S}tokes
  equations.
\newblock {\em J. Eur. Math. Soc.}, to appear, 2019.

\bibitem{HZZ22b}
Martina Hofmanov\'{a}, Rongchan Zhu, and Xiangchan Zhu.
\newblock A class of supercritical/critical singular stochastic {PDE}s:
  existence, non-uniqueness, non-{G}aussianity, non-unique ergodicity.
\newblock {\em arXiv preprint}, \url{https://arxiv.org/abs/2205.13378}, 2022.

\bibitem{HZZ22a}
Martina Hofmanov\'{a}, Rongchan Zhu, and Xiangchan Zhu.
\newblock Non-unique ergodicity for deterministic and stochastic 3{D}
  {N}avier--{S}tokes and euler equations.
\newblock {\em arXiv preprint}, \url{https://arxiv.org/abs/2208.08290}, 2022.

\bibitem{HZZ22}
Martina Hofmanov\'{a}, Rongchan Zhu, and Xiangchan Zhu.
\newblock On ill- and well-posedness of dissipative martingale solutions to
  stochastic 3{D} {E}uler equations.
\newblock {\em Comm. Pure Appl. Math.}, 75(11):2446--2510, 2022.

\bibitem{HZZ21b}
Martina Hofmanov\'{a}, Rongchan Zhu, and Xiangchan Zhu.
\newblock {Global {E}xistence and {N}on-{U}niqueness for 3{D}
  {N}avier--{S}tokes {E}quations with {S}pace-{T}ime {W}hite {N}oise}.
\newblock {\em Arch. Ration. Mech. Anal.}, 247:46, 2023.

\bibitem{HZZ21a}
Martina Hofmanov\'{a}, Rongchan Zhu, and Xiangchan Zhu.
\newblock Global-in-time probabilistically strong and {M}arkov solutions to
  stochastic 3{D} {N}avier-{S}tokes equations: existence and nonuniqueness.
\newblock {\em Ann. Probab.}, 51(2):524--579, 2023.

\bibitem{Isett18}
Philip Isett.
\newblock A proof of {O}nsager's conjecture.
\newblock {\em Ann. of Math. (2)}, 188(3):871--963, 2018.

\bibitem{KS91}
Ioannis Karatzas and Steven~E. Shreve.
\newblock {\em Brownian motion and stochastic calculus}, volume 113 of {\em
  Graduate Texts in Mathematics}.
\newblock Springer-Verlag, New York, second edition, 1991.

\bibitem{Keller62}
Joseph~B. Keller.
\newblock Wave propagation in random media.
\newblock In {\em Proc. {S}ympos. {A}ppl. {M}ath., {V}ol. {XIII}}, pages
  227--246. American Mathematical Society, Providence, R.I., 1962.

\bibitem{KY22}
Ujjwal Koley and Kazuo Yamazaki.
\newblock Non-uniqueness in law of transport-diffusion equation forced by
  random noise.
\newblock {\em arXiv preprint}, \url{https://arxiv.org/abs/2203.13456}, 2022.

\bibitem{KR05}
Nicolai~V. Krylov and Michael R\"{o}ckner.
\newblock Strong solutions of stochastic equations with singular time dependent
  drift.
\newblock {\em Probab. Theory Related Fields}, 131(2):154--196, 2005.

\bibitem{Krylov23}
N.V. Krylov.
\newblock On weak solutions of time inhomogeneous {I}t\^{o}'s equations with
  {VMO} diffusion and {M}orrey drift.
\newblock {\em arXiv preprint}, \url{https://arxiv.org/abs/2303.11238}, 2023.

\bibitem{Kunita84b}
Hiroshi Kunita.
\newblock First order stochastic partial differential equations.
\newblock In {\em Stochastic analysis ({K}atata/{K}yoto, 1982)}, volume~32 of
  {\em North-Holland Math. Library}, pages 249--269. North-Holland, Amsterdam,
  1984.

\bibitem{Kunita84a}
Hiroshi Kunita.
\newblock Stochastic differential equations and stochastic flows of
  diffeomorphisms.
\newblock In {\em \'{E}cole d'\'{e}t\'{e} de probabilit\'{e}s de
  {S}aint-{F}lour, {XII}---1982}, volume 1097 of {\em Lecture Notes in Math.},
  pages 143--303. Springer, Berlin, 1984.

\bibitem{Kunita97}
Hiroshi Kunita.
\newblock {\em Stochastic flows and stochastic differential equations},
  volume~24 of {\em Cambridge Studies in Advanced Mathematics}.
\newblock Cambridge University Press, Cambridge, 1997.
\newblock Reprint of the 1990 original.

\bibitem{LBL04}
Claude Le~Bris and Pierre-Louis Lions.
\newblock Renormalized solutions of some transport equations with partially
  {$W^{1,1}$} velocities and applications.
\newblock {\em Ann. Mat. Pura Appl. (4)}, 183(1):97--130, 2004.

\bibitem{LBL08}
Claude Le~Bris and Pierre-Louis Lions.
\newblock Existence and uniqueness of solutions to {F}okker-{P}lanck type
  equations with irregular coefficients.
\newblock {\em Comm. Partial Differential Equations}, 33(7-9):1272--1317, 2008.

\bibitem{LR15}
Wei Liu and Michael R\"{o}ckner.
\newblock {\em Stochastic partial differential equations: an introduction}.
\newblock Universitext. Springer, Cham, 2015.

\bibitem{LZ22}
Huaxiang L\"u and Xiangchan Zhu.
\newblock Global-in-time probabilistically strong solutions to stochastic
  power-law equations: existence and non-uniqueness.
\newblock {\em arXiv preprint}, \url{https://arxiv.org/abs/2209.02531}, 2022.

\bibitem{Lunardi18}
Alessandra Lunardi.
\newblock {\em Interpolation theory}, volume~16 of {\em Appunti. Scuola Normale
  Superiore di Pisa (Nuova Serie) [Lecture Notes. Scuola Normale Superiore di
  Pisa (New Series)]}.
\newblock Edizioni della Normale, Pisa, 2018.
\newblock Third edition [of MR2523200].

\bibitem{Maurelli11}
Mario Maurelli.
\newblock Wiener chaos and uniqueness for stochastic transport equation.
\newblock {\em C. R. Math. Acad. Sci. Paris}, 349(11-12):669--672, 2011.

\bibitem{MS20}
Stefano Modena and Gabriel Sattig.
\newblock Convex integration solutions to the transport equation with full
  dimensional concentration.
\newblock {\em Ann. Inst. H. Poincar\'{e} C Anal. Non Lin\'{e}aire},
  37(5):1075--1108, 2020.

\bibitem{MS18}
Stefano Modena and L\'{a}szl\'{o} Sz\'{e}kelyhidi, Jr.
\newblock Non-uniqueness for the transport equation with {S}obolev vector
  fields.
\newblock {\em Ann. PDE}, 4(2):Paper No. 18, 38, 2018.

\bibitem{MS19}
Stefano Modena and L\'{a}szl\'{o} Sz\'{e}kelyhidi, Jr.
\newblock Non-renormalized solutions to the continuity equation.
\newblock {\em Calc. Var. Partial Differential Equations}, 58(6):Paper No. 208,
  30, 2019.

\bibitem{MNP15}
Salah-Eldin~A. Mohammed, Torstein~K. Nilssen, and Frank~N. Proske.
\newblock Sobolev differentiable stochastic flows for {SDE}s with singular
  coefficients: applications to the transport equation.
\newblock {\em Ann. Probab.}, 43(3):1535--1576, 2015.

\bibitem{MS03}
Stefan M\"{u}ller and Vladim\'{i}r \v{S}ver\'{a}k.
\newblock Convex integration for {L}ipschitz mappings and counterexamples to
  regularity.
\newblock {\em Ann. of Math. (2)}, 157(3):715--742, 2003.

\bibitem{Nam20}
Kyeongsik Nam.
\newblock Stochastic differential equations with critical drifts.
\newblock {\em Stochastic Process. Appl.}, 130(9):5366--5393, 2020.

\bibitem{Nash54}
John Nash.
\newblock {$C^1$} isometric imbeddings.
\newblock {\em Ann. of Math. (2)}, 60:383--396, 1954.

\bibitem{NO15}
Wladimir Neves and Christian Olivera.
\newblock Wellposedness for stochastic continuity equations with
  {L}adyzhenskaya-{P}rodi-{S}errin condition.
\newblock {\em NoDEA Nonlinear Differential Equations Appl.}, 22(5):1247--1258,
  2015.

\bibitem{NO16}
Wladimir Neves and Christian Olivera.
\newblock Stochastic continuity equations---a general uniqueness result.
\newblock {\em Bull. Braz. Math. Soc. (N.S.)}, 47(2):631--639, 2016.

\bibitem{novack1}
Matthew Novack and Vlad Vicol.
\newblock An intermittent {O}nsager theorem.
\newblock {\em Invent. Math.}, 2023.

\bibitem{Ogawa73}
Shigeyoshi Ogawa.
\newblock A partial differential equation with the white noise as a
  coefficient.
\newblock {\em Z. Wahrscheinlichkeitstheorie und Verw. Gebiete}, 28:53--71,
  1973/74.

\bibitem{Ogawa78}
Shigeyoshi Ogawa.
\newblock Remarks on the {$B$}-shifts of generalized random processes.
\newblock In {\em Proceedings of the {I}nternational {S}ymposium on
  {S}tochastic {D}ifferential {E}quations ({R}es. {I}nst. {M}ath. {S}ci.,
  {K}yoto {U}niv., {K}yoto, 1976)}, pages 327--339. Wiley, New
  York-Chichester-Brisbane, 1978.

\bibitem{Pappalettera23}
Umberto Pappalettera.
\newblock Global existence and non-uniqueness for the {C}auchy problem
  associated to 3{D} {N}avier-{S}tokes equations perturbed by transport noise.
\newblock {\em arXiv preprint}, \url{https://arxiv.org/abs/2303.02363}, 2023.

\bibitem{pitcho2021almost}
Jules Pitcho and Massimo Sorella.
\newblock Almost everywhere non-uniqueness of integral curves for
  divergence-free sobolev vector fields.
\newblock {\em arXiv preprint}, \url{https://arxiv.org/abs/2108.03194}, 2021.

\bibitem{RS23}
Marco Rehmeier and Andre Schenke.
\newblock Nonuniqueness in law for stochastic hypodissipative {N}avier-{S}tokes
  equations.
\newblock {\em Nonlinear Anal.}, 227:Paper No. 113179, 37, 2023.

\bibitem{RY99}
Daniel Revuz and Marc Yor.
\newblock {\em Continuous martingales and {B}rownian motion}, volume 293 of
  {\em Grundlehren der Mathematischen Wissenschaften [Fundamental Principles of
  Mathematical Sciences]}.
\newblock Springer-Verlag, Berlin, third edition, 1999.

\bibitem{RZ23}
Michael R\"{o}ckner and Guohuan Zhao.
\newblock S{DE}s with critical time dependent drifts: {W}eak solutions.
\newblock {\em Bernoulli}, 29(1):757--784, 2023.

\bibitem{Tartar79}
Luc Tartar.
\newblock Compensated compactness and applications to partial differential
  equations.
\newblock In {\em Nonlinear analysis and mechanics: {H}eriot-{W}att
  {S}ymposium, {V}ol. {IV}}, volume~39 of {\em Res. Notes in Math.}, pages
  136--212. Pitman, Boston, Mass.-London, 1979.

\bibitem{Taylor3}
Michael~E. Taylor.
\newblock {\em Partial differential equations {III}. {N}onlinear equations},
  volume 117 of {\em Applied Mathematical Sciences}.
\newblock Springer, New York, second edition, 2011.

\bibitem{WLW17}
Jinlong Wei, Guangying Lv, and Jiang-Lun Wu.
\newblock On weak solutions of stochastic differential equations with sharp
  drift coefficients.
\newblock {\em arXiv preprint}, \url{https://arxiv.org/abs/1711.05058}, 2017.

\bibitem{Yamazaki21}
Kazuo Yamazaki.
\newblock Non-uniqueness in law of three-dimensional magnetohydrodynamics
  system forced by random noise.
\newblock {\em arXiv preprint}, \url{https://arxiv.org/abs/2109.07015}, 2021.

\bibitem{Yamazaki22B}
Kazuo Yamazaki.
\newblock Non-uniqueness in law for the {B}oussinesq system forced by random
  noise.
\newblock {\em Calc. Var. Partial Differential Equations}, 61(5):Paper No. 177,
  65, 2022.

\bibitem{Yamazaki22S}
Kazuo Yamazaki.
\newblock Non-uniqueness in law of the two-dimensional surface
  quasi-geostrophic equations forced by random noise.
\newblock {\em arXiv preprint}, \url{https://arxiv.org/abs/2208.05673}, 2022.

\bibitem{Yamazaki22L}
Kazuo Yamazaki.
\newblock Remarks on the non-uniqueness in law of the {N}avier-{S}tokes
  equations up to the {J}.-{L}. {L}ions' exponent.
\newblock {\em Stochastic Process. Appl.}, 147:226--269, 2022.

\end{thebibliography}
 % expects file "refs.bib"
 \end{document}